\documentclass[letterpaper,12pt]{amsart}
\usepackage{amscd, amssymb}
\usepackage{amsmath,amscd}
\usepackage{comment}
\usepackage{graphicx}
\usepackage[usenames,dvipsnames]{color}
\usepackage{bm}
\usepackage{enumitem}
\usepackage[all]{xy}   
\setlist[enumerate,1]{label={(\alph*)}}
\setlist[enumerate,2]{label={(\roman*)}}
\input xy
\xyoption{all}
\usepackage{hyperref}
\usepackage{mathrsfs}

\newtheorem{thm}{Theorem}[section]
\newtheorem{prop}[thm]{Proposition}

\newtheorem{lm}[thm]{Lemma}
\newtheorem{lemma}[thm]{Lemma}
\newtheorem{cor}[thm]{Corollary}

\newtheorem{conj}{Conjecture}
\newcommand{\p}{\partial}

\theoremstyle{definition}
\newtheorem{definition}[thm]{Definition}
\newtheorem{nn}[thm]{Notation}

\theoremstyle{remark}
\newtheorem{rmk}[thm]{Remark}
\newtheorem{ex}[thm]{Example}
\newtheorem{obs}[thm]{Observation}

\numberwithin{equation}{section}

\newcommand{\R}{\mathbb R}
\newcommand{\C}{\mathbb C}
\newcommand{\Q}{\mathbb Q}
\newcommand{\N}{\mathbb N}
\newcommand{\HH}{\mathbb H}
\newcommand{\LL}{\mathfrak L}
\newcommand{\Z}{\mathbb Z}
\newcommand{\CM}{{\mathcal{M}}}
\newcommand{\oCM}{{\overline{\mathcal{M}}}}
\newcommand{\CMm}{\CM^{main}}
\newcommand{\oCMm}{\oCM^{main}}
\newcommand{\CC}{{\mathcal{C}}}

\newcommand{\CL}{{\mathbb{L}}}
\newcommand{\CF}{{\mathcal{F}}}
\newcommand{\CS}{{\mathcal{S}}}
\newcommand{\CB}{{\mathcal{B}}}
\newcommand{\E}{{\mathbb{E}}}
\newcommand{\CE}{{\mathcal{E}}}
\newcommand{\CG}{{\mathcal{G}}}
\newcommand{\CGDi}{{\CG_{D_i}}}
\newcommand{\bv}[2]{{v_{#1}^\star(#2)}}
\newcommand{\CO}{{\mathcal{O}}}
\newcommand{\CODt}[1]{{\CO(\mathcal{D}_{#1})}}

\newcommand{\CV}{{\mathcal{V}}}
\newcommand{\CU}{{\mathcal{U}}}
\newcommand{\pB}{{\partial^B}}
\newcommand{\pu}{{\partial^!}}
\newcommand{\blangle}{{\big\langle}}
\newcommand{\bblangle}{{\big\langle}{\big\langle}}
\newcommand{\brangle}{{\big\rangle}}
\newcommand{\bbrangle}{{\big\rangle}{\big\rangle}}
\newcommand{\s}{\mathbf{s}}
\newcommand{\sr}{\mathbf{s}_1}
\newcommand{\tsr}{\tilde{\mathbf{s}}_1}
\newcommand{\ba}{\mathbf{a}}
\newcommand{\bb}{\mathbf{b}}
\newcommand{\bc}{\mathbf{c}}
\newcommand{\be}{\mathbf{e}}
\newcommand{\bze}{\mathbf{0}}
\newcommand{\tr}{\rho}
\newcommand{\trs}{\mathbf{r}}

\DeclareMathOperator{\ior}{int}
\renewcommand{\mathring}{\ior}

\DeclareMathOperator{\im}{Im}
\DeclareMathOperator{\re}{Re}
\DeclareMathOperator{\id}{Id}

\DeclareMathOperator{\rk}{rk}

\DeclareMathOperator{\supp}{supp}
\DeclareMathOperator{\HB}{HB}

\DeclareMathOperator{\Sym}{Sym}

\newcommand{\rz}[1]{#1^\circ}

\title{Intersection theory on moduli of disks, open KdV and Virasoro}

\subjclass[2020]{32G15; 14H15 (Primary); 37K20; 14N35; 53D45 (Secondary)}

\date{March 2022}

\author[R. Pandharipande]{Rahul Pandharipande}
\address{Departement Mathematik \\
ETH Z\"urich}
\email{rahul@math.ethz.ch}

\author[J. Solomon]{Jake P. Solomon}
\address{Institute of Mathematics\\ Hebrew University}
\email{jake@math.huji.ac.il}

\author[R. Tessler]{Ran J. Tessler}
\address{Department of Mathematics\\ Weizmann Institute of Science}
\email{ran.tessler@weizmann.ac.il}

\begin{document}

\begin{abstract}
We define a theory of descendent integration on the moduli spaces of
stable pointed disks. The descendent integrals are proved to be coefficients of the $\tau$-function of an open KdV hierarchy. A relation between the integrals and a representation of half the Virasoro algebra is also
proved. The construction of the theory requires an in depth study of homotopy
classes of multivalued boundary conditions. Geometric recursions based on the
combined structure of the boundary conditions and the moduli space are used to compute the integrals. We also provide a detailed analysis of orientations.

Our open KdV and Virasoro constraints uniquely specify a  theory of
higher genus open descendent integrals. As a result, we obtain an open
analog (governing all genera) of Witten's conjectures concerning descendent integrals on the Deligne-Mumford space of stable curves.
\end{abstract}

\maketitle

\pagestyle{plain}

\tableofcontents

\section{Introduction}
\subsection{Moduli of closed Riemann surfaces}
Let $C$ be a connected complex manifold of dimension 1. If $C$
is closed, the underlying topology is
classified by the genus $g$. The moduli space $\CM_g$ of complex
structures of genus $g$ has been studied since Riemann \cite{R} in the
$19^{th}$ century.
Deligne and Mumford  defined a natural compactification
$$\CM_g \subset \oCM_g$$
via stable curves (with possible nodal singularities) in 1969.
The moduli $\CM_{g,l}$ of curves
$(C,p_1,\ldots,p_l)$ with $l$ distinct marked points has a parallel
treatment
with compactification
$$\CM_{g,l} \subset \oCM_{g,l}\ .$$
We refer the reader to \cite{DM,HM} for the basic theory.
The moduli space $\oCM_{g,l}$ is a nonsingular
complex orbifold of dimension $3g-3+l$.

\subsection{Witten's conjectures} \label{wittc}
A new direction in the study of the moduli of curves
was opened by Witten \cite{Witten} in 1992 motivated
by theories of 2-dimensional quantum gravity.
For each marking index $i$, a complex {\em cotangent} line bundle
$$\CL_i \rightarrow \oCM_{g,l}\ $$
is defined as follows.
The fiber of $\CL_i$ over the point
$$[C,p_1,\ldots,p_l]\in \oCM_{g,l}$$ is the complex
cotangent{\footnote{By stability, $p_i$
lies in the nonsingular locus of $C$.}}
space $T_{C,p_i}^*$ of $C$ at $p_i$.
Let $$\psi_i\in H^2(\oCM_{g,l},\mathbb{Q})$$
 denote the first Chern class of
$\CL_i$.
Witten considered the intersection products of the classes $\psi_i$.
We will follow the standard bracket notation:
\begin{equation}
\label{products}
\blangle\tau_{a_1} \tau_{a_2} \cdots \tau_{a_l}\brangle_g = \int
_{\overline {M}_{g,l}} \psi_1^{a_1} \psi_2^{a_2} \cdots \psi_l^{a_l}.
\end{equation}
The integral on the right of \eqref{products} is
 well-defined when the stability condition
$$2g-2+l >0$$
is satisfied,
all the $a_i$ are
nonnegative integers, and the dimension constraint
\begin{equation}\label{dd344}
3g-3+l=\sum a_i
\end{equation} holds.
In all other cases, $\blangle\prod_{i=1}^{l} \tau_{a_i}\brangle_g$ is
defined to be zero.
The empty bracket $\blangle 1 \brangle_1$ is also set to zero.
The intersection products \eqref{products} are often called
{\em descendent integrals}.

By the dimension constraint \eqref{dd344}, a unique genus
$g$ is determined by the $a_i$. For brackets without
a genus subscript, the genus specified by the dimension
constraint is assumed (the bracket is set to zero if the specified
genus is fractional).
The simplest integral is
\begin{equation}\label{fvv23}
\blangle\tau_0^3\brangle =\blangle\tau_0^3\brangle_0 =1\ .
\end{equation}

Let $t_i$ (for $i\geq 0$) be a set of variables.
Let $\gamma=\sum_{i=0}^{\infty} t_i \tau_i$ be the formal sum.
Let
$$F_g(t_0, t_1, ...)= \sum_{n=0}^{\infty} \frac{\blangle\gamma^n\brangle_g}{n!}$$
be the
generating function of genus $g$ descendent integrals \eqref{products}.
The bracket $\blangle\gamma^n\brangle_g$  is defined by monomial expansion
and multilinearity in the variables $t_i$. Concretely,
$$F_g(t_0, t_1, ...)= \sum_{\{n_i\}} \prod_{i=1}^{\infty}
\frac{t_i^{n_i}}{n_i!} \blangle\tau_0^{n_0} \tau_1^{n_1} \tau_2^{n_2} \cdots\brangle_g,$$
where the sum is over all sequences of nonnegative integers $\{n_i\}$
with finitely many nonzero terms.
The generating function
\begin{equation}\label{v34}
F= \sum_{g=0}^{\infty} u^{2g-2} F_g
\end{equation}
arises as a partition function in $2$-dimensional quantum gravity. Based on a
different physical
realization of this function in terms of matrix integrals, Witten \cite{Witten}
conjectured $F$ satisfies two distinct systems of differential equations.
Each system determines $F$ uniquely and provides
explicit recursions which compute all the brackets \eqref{products}.
Witten's conjectures were proven by Kontsevich \cite{Kont}. Other
proofs can be found in \cite{Mir,OP}.

Before describing the full systems of equations,
we recall two basic properties.
The first is the {\em string equation}: for $2g-2+l>0$,
\[
\left\langle\tau_0 \prod_{i=1}^{l} \tau_{a_i}\right\rangle_g =
\sum_{j=1}^{l} \left\langle\tau_{a_j-1} \prod_{i\neq j} \tau_{a_i}\right\rangle_g.
\]
The second property is the {\em dilaton equation}:
for $2g-2+l>0$,
\begin{equation}\label{eq:cdil}
\left\langle\tau_1 \prod_{i=1}^{l} \tau_{a_i}\right\rangle_g =
(2g-2+l) \left\langle\prod_{i=1}^{l} \tau_{a_i}\right\rangle_g.
\end{equation}
The string and dilaton equations may be written as differential operators
annihilating  ${\rm exp}(F)$ in the following way:
\begin{eqnarray}
\label{lminus}
L_{-1}&=& -\frac{\p}{\p t_0} + \frac{u^{-2}}{2} t_0^2 +\sum_{i=0}^{\infty} t_{i+1}
\frac{\p}{\p t_i}\ , \\
\nonumber
L_0 &=& - \frac{3}{2} \frac{\p}{\p t_1}+ \sum_{i=0}^{\infty}
\frac{2i+1}{2} t_i \frac{\p}{\p t_i} + \frac{1}{16}\ .
\end{eqnarray}
Both the string and dilaton equations are derived \cite{Witten}
from a comparison
result describing the behavior of the $\psi$ classes under
pull-back via the forgetful map $$\pi: \oCM_{g,l+1} \rightarrow
\oCM_{g,l}\ .$$

The string equation and the evaluation \eqref{fvv23}
together determine all the genus 0 brackets.
The string equation, dilaton equation, and the
evaluation
\begin{equation}\label{bvvn}
\blangle\tau_1\brangle_1=\frac{1}{24}
\end{equation}
determine all the genus 1 brackets. In higher genus,
further constraints are needed.

The first differential equations conjectured by Witten are the
KdV equations. We define
the functions
\begin{equation}
\label{pproductss}
\blangle\blangle\tau_{a_1} \tau_{a_2} \cdots \tau_{a_l}\brangle\brangle =
\frac{\p}{\p t_{a_1}} \frac{\p}{\p t_{a_2}} \cdots \frac{\p}{\p t_{a_l}} F.
\end{equation}
Of course, we have
$$\blangle\blangle\tau_{a_1} \tau_{a_2}
\cdots \tau_{a_l}\brangle\brangle\Big|_{t_i=0,u = 1}= \blangle\tau_{a_1}
\tau_{a_2} \cdots \tau_{a_l}\brangle\ .$$
The KdV equations are equivalent to the following set
of equations for $n\geq 1$:
\begin{align*}
&(2n+1)u^{-2} \blangle\blangle\tau_n \tau_0^2 \brangle\brangle = \\
&  \qquad\qquad =\blangle\blangle\tau_{n-1} \tau_0\brangle\brangle
\blangle\blangle\tau_0^3\brangle\brangle +
2\blangle\blangle\tau_{n-1}\tau_0^2\brangle\brangle\blangle\blangle\tau_0^2\brangle\brangle+
\frac{1}{4}\blangle\blangle\tau_{n-1} \tau_0^4\brangle\brangle.
\end{align*}
For example, consider the KdV equation for $n=3$ evaluated at
$t_i=0$. We obtain
$$7\blangle\tau_3 \tau_0^2\brangle_1=
\blangle\tau_2 \tau_0\brangle_1 \blangle\tau_0^3\brangle_0 +{\frac{1}{4}} \blangle\tau_2
\tau_0^4\brangle_0.$$
Use of the string equation yields:
$$7\blangle\tau_1\brangle_1=
\blangle\tau_1\brangle_1 + \frac{1}{4} \blangle\tau_0^3\brangle_0.$$
Hence, we conclude  \eqref{bvvn}. In fact,
the KdV equations  and the string equation {\em together} determine all the products (\ref{products}) and thus
uniquely determine $F$.

The second system of differential equations for $F$ is determined by a
representation of a subalgebra of the Virasoro algebra.
Consider the Lie algebra $\mathbf{L}$ of holomorphic differential operators spanned by
$$L_n= -z^{n+1} \frac{\p}{\p z}$$
for $n\geq -1$. The bracket is given by $[L_n, L_m]=(n-m) L_{n+m}$.

The equations (\ref{lminus}) may be viewed as the beginning of
a representation of $\mathbf{L}$ in a Lie algebra of differential operators.
In fact, with certain homogeneity restrictions, there is a
unique way to extend the  assignment of
$L_{-1}$ and $L_{0}$ to a complete representation
of $\mathbf{L}$. For $n\geq 1$, the expression for $L_n$ takes the form
\begin{align}
&L_n = \notag\\
& =-\frac{3\cdot 5\cdot 7 \cdots (2n+3)}{2^{n+1}} \frac{\p}{\p t_{n+1}}\notag\\
&\quad\, + \sum_{i=0}^{\infty} \frac{(2i+1)(2i+3) \cdots (2i+2n+1)}{2^{n+1}} t_i \frac{\p}
{\p t_{i+n}}\notag\\
&\quad\, + \frac{u^2}{2} \sum_{i=0}^{n-1} (-1)^{i+1} \frac{(-2i-1)(-2i+1) \cdots
(-2i+2n-1)}{2^{n+1}} \frac{\p^2}{\p t_i \p t_{n-1-i}}.\notag
\end{align}

The second form of Witten's conjecture is  that the above representation of
$\mathbf{L}$ annihilates ${\rm exp} (F)$:
\begin{equation}
\label{vira}
\forall n\geq -1, \ \ \ \ L_n \ {\rm exp}(F) =0\ .
\end{equation}
The system of equations (\ref{vira}) also uniquely
determines $F$.

The KdV equations and the Virasoro constraints provide a very
satisfactory approach to the products \eqref{products}. The aim
of our paper is to develop a parallel theory for open Riemann surfaces.  An {\em open Riemann surface} for
us is obtained
by removing open disks from a closed Riemann surface. See Section \ref{sec:rswb} below for a more detailed discussion.
Hence, the terminology {\em Riemann surface with boundary}
is more appropriate. We will use the terms {\em open} and
{\em with boundary} synonymously.

For the remainder of the paper, a superscript $c$ will signal
integration over the moduli of closed Riemann surfaces.
For example, we will write the generating series of descendent integrals
\eqref{v34} as
$$F^c(u,t_0, t_1, ...)= \sum_{g=0}^\infty u^{2g-2} \sum_{n=0}^\infty
\frac{\blangle {\gamma^n}
\brangle_g^c}{n!}\ .$$
We will later introduce a generating series $F^o$ of descendent
integrals over the moduli of
open Riemann surfaces.

\subsection{Moduli of Riemann surfaces with boundary}\label{sec:rswb}
Let $\Delta \subset \mathbb{C}$ be the open unit disk, and let
$\overline{\Delta}$ be the closure.
An {\em extendable} embedding of the open disk in a closed Riemann surface
$$f: \Delta \rightarrow C$$
is a holomorphic map which extends to a holomorphic embedding of an open
neighborhood of $\overline{\Delta}$.
Two extendable embeddings in $C$ are {\em disjoint} if the
images of $\overline{\Delta}$ are disjoint.

A {\em Riemann surface with boundary} $(X,\partial X)$ is obtained by removing finitely many
disjoint extendably embedded open disks from a connected
closed Riemann surface.
The boundary $\partial X$ is the union of images of the
unit circle boundaries of embedded disks $\Delta$.
Alternatively, a Riemann surface with boundary is defined like a Riemann surface except that the coordinate charts are allowed to map homeomorphically to an open subset of the closed upper half plane.

Given a Riemann surface with boundary $(X,\partial X)$, we
can canonically construct a {\em double}
via Schwarz reflection through the boundary~\cite[Sec. II.1.3]{AhS60}.
The double $D(X,\partial X)$ of
$(X,\partial X)$ is a closed Riemann surface. The {\em doubled
genus} of $(X,\partial X)$ is defined to be the usual genus of
$D(X,\partial X)$.

On a Riemann surface with boundary $(X,\partial X)$, we
consider two
types of marked points. The markings of {\em interior type}
are points of $X\setminus \partial X$.
The markings of {\em boundary type} are points of
$\partial X$.
Let
$\CM_{g,k,l}$ denote the moduli space of Riemann surfaces with boundary of
doubled genus $g$ with $k$ distinct boundary markings and
$l$ distinct interior markings.
The moduli space $\CM_{g,k,l}$ is defined to be empty unless
the stability condition,
$$2g-2+k+2l >0, $$
is satisfied.
The moduli space $\CM_{g,k,l}$
may have several connected components depending upon the
topology of $(X,\partial X)$ and the cyclic orderings of the
boundary markings. Foundational issues concerning the construction
of $\CM_{g,k,l}$ are addressed in \cite{Li03}.

We view $\CM_{g,k,l}$ as a real orbifold of real dimension
$3g-3+k+2l$. Of course, $\CM_{g,k,l}$ is not compact (in addition to
the nodal degenerations present in the moduli of closed Riemann surfaces,
new issues involving the boundary approach of interior markings and
the  meeting of boundary circles arise). Furthermore, $\CM_{g,k,l}$ may
not be orientable. Non-orientability presents serious obstacles for the definition of a theory of
descendent integration over the moduli spaces of Riemann surfaces with boundary.

We will often refer to connected Riemann surfaces with boundary
as {\em open} Riemann surfaces or {\em open} geometries (as the
interior is open). The {\em genus} of an open Riemann surface will
always be the doubled genus.

\subsection{Descendents} \label{odesc}
Since interior marked points have well-defined cotangent
spaces,
there is no difficulty in defining the cotangent line bundles
$$\CL_i \rightarrow \CM_{g,k,l}$$
for each interior marking, $i = 1,\ldots,l.$
We do {\em not} consider the cotangent lines at the boundary
points.

Naively, we would like to consider a descendent theory via integration of products of the first Chern classes $\psi_i = c_1(\CL_i) \in H^2(\oCM_{g,k,l})$ over a compactification $\oCM_{g,k,l}$ of $\CM_{g,k,l}$. Namely,
\begin{equation}
\label{p9}
\blangle\tau_{a_1} \tau_{a_2} \cdots \tau_{a_l}\sigma^k\brangle_{g}^o = \int
_{\overline {\CM}_{g,k,l}} \psi_1^{a_1} \psi_2^{a_2} \cdots \psi_l^{a_l}\ .
\end{equation}
when
\[
2\sum_{i = 1}^l a_i = 3g-3 + k +2l,
\]
and in all other cases $\blangle \tau_{a_1}\cdots\tau_{a_l} \sigma^k\brangle_g^o = 0.$
Here, $\tau_a$ corresponds to the $a^{th}$ power of a cotangent class $\psi^a$
as before. The new insertion
$\sigma$ corresponds to the addition
of a boundary marking.{\footnote{The power of $\sigma$ specifies the
number of boundary markings.}}
To rigorously define the right-hand side of \eqref{p9},
at least three significant steps must be taken:
\begin{enumerate}
\item[(i)]
A compact moduli space $\oCM_{g,k,l}$ must be constructed. Because degenerations of Riemann surfaces with boundary occur in real codimension one, candidates for $\oCM_{g,k,l}$ are real orbifolds with boundary $\partial \oCM_{g,k,l}$.
\item[(ii)]
For integration over $\oCM_{g,k,l}$ to be well-defined, boundary conditions of the integrand must be specified along $\partial \oCM_{g,k,l}$. That is, the integrand must be lifted to the relative cohomology group $H^{3g-3+k+2l}(\oCM_{g,k,l},\partial\oCM_{g,k,l})$.
\item[(iii)]
Orientation issues must be addressed.
\end{enumerate}
The most challenging aspect of defining open descendent integrals is the specification of boundary conditions (ii). At first glance, one might hope to find a natural lift of $\psi_i$ to $H^2(\oCM_{g,k,l},\partial \oCM_{g,k,l}).$ However, this does not appear feasible. Rather, consider the bundle
\begin{equation}\label{eq:intrE}
E = \bigoplus_{i = 1}^l \CL_i^{\oplus a_i}.
\end{equation}
The Euler class of $E$ is given by
\[
e(E) = \psi_1^{a_1} \psi_2^{a_2} \cdots \psi_l^{a_l}.
\]
So, it suffices to find a natural lift of $e(E)$ to relative cohomology. This is the approach we follow. Such an approach leads to considerable difficulties in proving recursive relations between descendent integrals. Indeed, for closed descendent integrals, the proofs of recursive relations use heavily the factorization of the integrand as a product of cohomology classes~\cite{Witten}.

The need to specify boundary conditions and the orientation issues impose serious constraints on the ultimate definition of $\oCM_{g,k,l}.$ For $g>0,$ it appears these constraints can only be satisfied if $\oCM_{g,k,l}$ is the compactification of a covering space of $\CM_{g,k,l}$ that arises as the moduli space of open Riemann surfaces with an additional structure. The construction of $\oCM_{g,k,l}$ for $g>0$ will be given in~\cite{JSRT}. In the case $g=0$ treated here, we take $\oCM_{0,k,l}$ to be the space of stable disks studied previously in the context of the Fukaya category~\cite{FO09,Li03,Seidel}.

In this paper, we complete steps (i-iii) in the doubled genus 0 case. The outcome
is a fully rigorous theory of descendent integration on the
moduli space of disks with interior and boundary markings.
Moreover, we prove analogs for $\oCM_{0,k,l}$ of the
string and dilation equations as well as the topological recursion relations, which allow us to completely solve the theory.

\subsection{Construction of the descendent theory of pointed disks}\label{ssec:iconst}
By the Riemann Mapping Theorem,
the open geometry of genus 0 is just the disk with a single boundary circle.
The simplest moduli space is $\CM_{0,3,0}$ parameterizing
disks with 3 boundary markings. There are exactly two
disks with 3 distinct boundary points (corresponding to
the two possible cyclic orders). Thus $\CM_{0,3,0}$ is already compact, and it has no boundary. So, we can evaluate the corresponding open descendent integral without reference to boundary conditions.
In our definition of open descendent
integrals~\eqref{eq:maindef}, the geometric integral over $\oCM_{0,k,l}$ is multiplied by $2^{\frac{1-k}2}.$ In particular, for $\CM_{0,3,0}$ the power is $2^{-1}.$
We conclude that
\begin{equation}\label{vvbb}
\blangle \sigma^3 \brangle_0^o = 1\ .
\end{equation}
Similarly, the moduli space $\CM_{0,1,1}$, parameterizing disks with $1$ boundary point and $1$ interior point, consists of a single point. It follows that also $\blangle \tau_0\sigma\brangle^o = 1.$

In general, the compact moduli space $\oCM_{0,k,l}$ of our construction is a compactification of $\CM_{0,k,l}$ stemming from ideas very close to Deligne-Mumford stability. It allows for internal sphere bubbles and
boundary disk bubbles following the approach familiar from the Fukaya category~\cite{FO09,Li03,Seidel}. See Figure~\ref{fig:nodal}.
\begin{figure}[t]
\centering
\includegraphics[scale=.7]{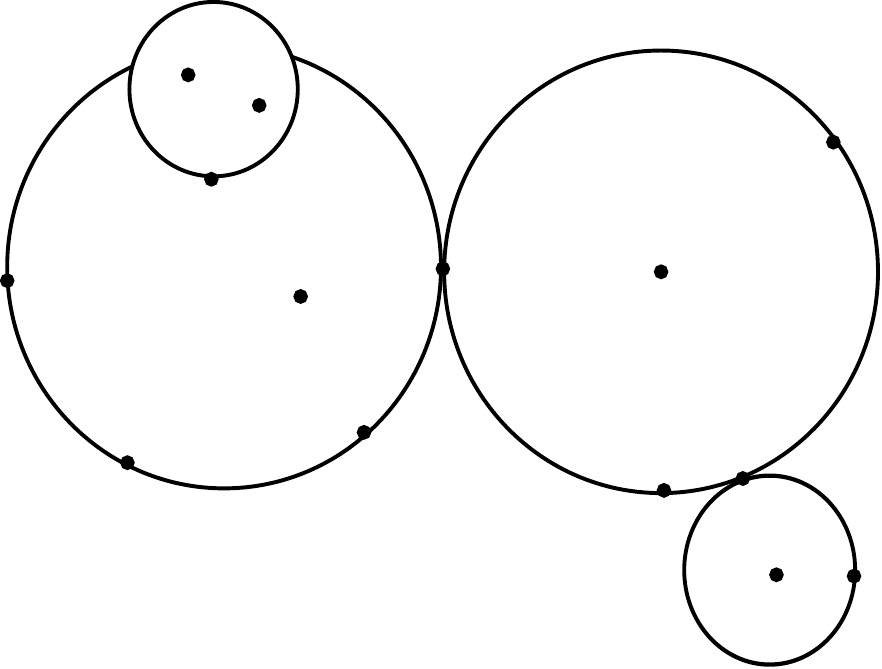}
\caption{A nodal disk with $3$ disk components, one sphere component, $5$ internal marked points and $6$ boundary marked points.}
\label{fig:nodal}
\end{figure}

The boundary conditions we impose for our definition of
the descendent integrals are the most delicate aspect of the
construction. As mentioned above, our strategy is to lift the Euler class $e(E)$ to the relative cohomology group $H^{k+2l-3}(\oCM_{0,k,l},\partial\oCM_{0,k,l})$. Here, $E\to \oCM_{0,k,l}$ is the bundle given by equation~\eqref{eq:intrE}. Such a lift can be given by constructing a non-vanishing section $\s$ of the restriction $E|_{\partial\oCM_{0,k,l}}.$ There is no unique construction of such a section $\s.$ Rather, we give a construction that is well-defined up to non-vanishing homotopy. It follows that the resulting lift of $e(E)$ to relative cohomology is well-defined. Our construction of $\s$ relies on the decomposition of the boundary $\partial \oCM_{0,k,l}$ into products of moduli spaces of open Riemann surfaces with fewer marked points.

A surprising feature of our construction is that the section $\s$ must be multi-valued. Multiple valued sections are forced on us by a non-trivial monodromy in the geometric constraint defining $\s.$ An explicit example of how this comes about is given in Remark~\ref{rm:msec}. In genus zero, the moduli space $\oCM_{0,k,l}$ is always a smooth manifold. So, the phenomenon of multi-valued sections is not the result of orbifold isotropy groups.

Another unintuitive aspect of the boundary conditions is the complexity of their dependence on the boundary marked points. Indeed, we consider only the cotangent lines $\CL_i$ at interior marked points. So, by analogy with the string equation, one would expect a simple geometric recursion to govern the dependence of open descendent integrals on the number of boundary marked points. This is not the case. To the contrary, in Section~\ref{sec:formulae} we observe a parallel between the formulas for open descendent integrals on $\oCM_{0,k,l}$ and closed $\lambda_g\lambda_{g-1}$ descendent integrals on $\oCM_{g,l},$ where $g$ is proportional to $k.$ That is, the number of boundary marked points in open genus zero descendent integrals plays a role analogous to the genus in closed $\lambda_g\lambda_{g-1}$ descendent integrals. This is one indication of the complex dependency of the boundary conditions on the boundary marked points.

Our proofs of the open analogs of the string, dilaton, and topological recursion
relations all use the boundary conditions in an essential way.
The boundary conditions are defined and constructed in Section \ref{sec:bdry}. The idea of the definition is outlined in Section~\ref{sssec:intuition}.

\subsection{Differential equations} \label{deq}
\subsubsection{Partition functions}
Though the resolution of the issues \mbox{(i-iii)} of Section \ref{odesc}
for the moduli of pointed disks (the genus 0 case)
requires a substantial mathematical development,
the evaluation of the theory is remarkably simple.
The answer guides the higher genus open cases.
We propose here an evaluation of the theory of descendent
integration over the moduli of Riemann surfaces with boundary
for {all} $g$, $k$, and $l$. For the genus $0$
case, we prove our proposal is correct using
our foundational development.
The main conjectures of the paper concern the  $g>0$
cases.  Even before giving complete definitions resolving~\mbox{(i-iii)} for $g>0$,
we are able to conjecture a complete
solution.

The solution is again via differential equations for
the generating series of descendent invariants.
Recall the descendent series for the moduli of closed Riemann
surfaces,
$$F^c(u,t_0, t_1,\ldots)= \sum_{g=0}^{\infty} u^{2g-2} F_g^c(t_0,t_1,\ldots) = \sum_{g=0}^\infty u^{2g-2} \sum_{n=0}^\infty
\frac{\blangle {\gamma^n}
\brangle_g^c}{n!}\ ,$$
where $\gamma= \sum_{i=0}^\infty t_i \tau_i$. Similarly, we define
the open descendent series as
$$F^o(u,s,t_0, t_1,\ldots)= \sum_{g = 0}^\infty u^{g-1} F_g^o(t_0,t_1,\ldots) = \sum_{g=0}^\infty u^{g-1} \sum_{n=0}^\infty
\frac{\blangle {\gamma^n \delta^k}
\brangle_g^o}{n!k!}\ ,$$
where $\gamma =\sum_{i=0}^\infty t_i \tau_i$ is as before and $\delta=s\sigma $.
The associated partition functions are
$$\mathsf{Z}^c= \exp(F^c)\ , \ \ \ \mathsf{Z}^o= \exp(F^o)\ .$$
We define the full partition function by
$$\mathsf{Z}= \exp(F^c + F^o)\ .$$

\subsubsection{Virasoro constraints}


Let $L_n$ be the differential operators in the variables
$u$ and $t_i$
defined in Section \ref{wittc}.
We define an $s$ extension
$\mathcal{L}_n$ of
$L_n$ by the following formula:
\begin{equation}
\mathcal{L}_n = L_n + u^n s \frac{\partial^{n+1}}{\partial s^{n+1}}
+\frac{3n+3}{4}u^n\frac{\partial^{n}}{\partial s^{n}}\ ,
\end{equation}
for $n\geq -1$.
Using the relations
$$[L_n,L_m]= (n-m)L_{n+m}$$
and the commutation of $L_n$ with the operators $u$, $s$, and
$\frac{\partial}{\partial s}$, we easily obtain the Virasoro relation
$$[\mathcal{L}_n,\mathcal{L}_m]= (n-m)\mathcal{L}_{n+m}$$
By Witten's conjecture, $L_n$
annihilates $\mathsf{Z}^c$.

\begin{conj} \label{cc11} The operators $\mathcal{L}_n$ annihilate the
full partition function,
$$\forall n \geq -1, \ \ \ \ \mathcal{L}_n \ \mathsf{Z} = 0.$$
\end{conj}

The restriction of the full partition function $\mathsf{Z}$ to the
subspace defined by $t_i=0$ for all $i$ is easily evaluated,
\begin{equation}\label{fbbz}
\mathsf{Z}(s,t_0=0,t_1=0,t_2=0, \ldots) = \blangle\sigma^3\brangle^o_0\
\frac{s^3}{3!}= \frac{s^3}{3!}\ .
\end{equation}
By a dimension analysis, the descendent $\blangle\sigma^3\brangle^o_0$,
evaluated by \eqref{vvbb},
is the only nonzero term which survives the restriction.
The Virasoro constraints of Conjecture \ref{cc11} then determine
$\mathsf{Z}$ from the restriction \eqref{fbbz}.
In other words, $\mathsf{Z}^o$ is uniquely and effectively
specified by
Conjecture~\ref{cc11}, the restriction \eqref{fbbz}, and
$\mathsf{Z}^c$.

Using our construction of the descendent theory of pointed disks,
we prove the  genus 0 part of Conjecture \ref{cc11}.

\begin{thm} \label{thm:vir} The operators $\mathcal{L}_n$ annihilate the genus zero partition function up to terms of higher genus. That is,
for $n \geq -1,$ the coefficient of $u^{-1}$ in
\[
\mathcal{L}_n \ \exp(u^{-2}F^c_0+u^{-1}F^o_0)
\]
vanishes.
\end{thm}

\noindent
The proof of Theorem \ref{thm:vir} is presented in Section
\ref{VirSec}.


\subsubsection{String and dilaton equations}
The string and dilaton equations for $F^o$ are
obtained from the operators $\mathcal{L}_{-1}$ and
$\mathcal{L}_0$ respectively.
The string equation for the open geometry is
\begin{equation}\label{gvvt}
\frac{\partial F^o}{\partial t_0} = \sum_{i=0}^{\infty}t_{i+1}
\frac{\partial F^o}{\partial t_i} + u^{-1}s.
\end{equation}
The dilaton equation is
\begin{equation*}
\frac{\partial F^o}{\partial t_1} = \sum_{i=0}^{\infty}\left(\frac{2i+1}{3}\right)t_{i}\frac{\partial F^o}{\partial t_i} + \frac{2}{3}s\frac{\partial
F^o}{\partial s}+\frac{1}{2}.
\end{equation*}
The string equation implies that for $2g-2+k + 2l > 0,$
\[
\left\langle\tau_0 \prod_{i = 1}^l \tau_{a_i}\sigma^k\right\rangle_{g}^o
            = \sum_j \left\langle \tau_{a_j-1} \prod_{i\neq j} \tau_{a_i}\sigma^k\right\rangle_{g}^o.
\]
The dilaton equation implies that for $2g-2 + k + 2l > 0,$
\begin{equation}\label{eq:dilaton}
\left\langle\tau_1 \prod_{i = 1}^l \tau_{a_i}\sigma^k\right\rangle_{g}^o
            = \left(g-1+k+l\right) \left\langle\prod_{i = 1}^l \tau_{a_i}\sigma^k\right\rangle_{g}^o.
\end{equation}
The string and dilaton equations for $F^c$ together with the
Virasoro relations
$$\mathcal{L}_{-1} \ \mathsf{Z} = \mathcal{L}_0 \ \mathsf{Z} = 0$$
imply the string and dilaton equations for $F^o$. The following
result is therefore a consequence of Theorem \ref{thm:vir}. It is also an important step in the proof.
\begin{thm} \label{thm:string_dilaton}
The string and dilaton equations hold for $F^o_0$.
\end{thm}

\subsubsection{KdV equations}
We have already defined \eqref{pproductss}
double brackets in the compact case. For the open invariants, the
definition is parallel:
\begin{equation*}
\blangle\blangle\tau_{a_1} \tau_{a_2} \cdots \tau_{a_l} \sigma^k
\brangle\brangle^o =
\frac{\p}{\p t_{a_1}} \frac{\p}{\p t_{a_2}} \cdots \frac{\p}{\p t_{a_l}}
\frac{\p^k}{\p s^k}
 F^o.
\end{equation*}
Also,
\begin{equation*}
\blangle\blangle\tau_{a_1} \tau_{a_2} \cdots \tau_{a_l} \sigma^k
\brangle\brangle^o_g =
\frac{\p}{\p t_{a_1}} \frac{\p}{\p t_{a_2}} \cdots \frac{\p}{\p t_{a_l}}
\frac{\p^k}{\p s^k}
 F^o_g.
\end{equation*}
We conjecture an analog of Witten's KdV equations in the compact case.

\begin{conj} For $n\geq 1$, we have \label{cc22}
\begin{eqnarray*}
(2n+1)u^{-1}\bblangle \tau_n \bbrangle^o  &= &
\ \ u\bblangle
\tau_{n-1} \tau_0\bbrangle^c  \bblangle \tau_0\bbrangle^o \\  & &
                       + 2    \bblangle \tau_{n-1}  \bbrangle^o
\bblangle\sigma\bbrangle^o
                  + 2 \bblangle \tau_{n-1} \sigma \bbrangle^o
\\ & & -\frac{u}{2} \bblangle \tau_{n-1}\tau_0^2 \bbrangle^c \ .\\
\end{eqnarray*}
\end{conj}
Together with the string equation \eqref{gvvt}, the
system of differential equations of Conjecture \ref{cc22}
uniquely determines $F^o$ from $\blangle \sigma^3 \brangle^o_0$ and
$F^c$. For example, we calculate (using $n=1$):
$$3 \blangle \tau_1 \brangle_1^o = 2 \blangle \tau_0 \sigma \brangle_0^o
-\frac{1}{2} \blangle \tau_0^3 \brangle_0^c = \frac{3}{2}\ ,$$
so $\blangle \tau_1\brangle= \frac{1}{2}$.
In fact, the system is significantly
overdetermined. We speculate
the differential equations for $F^o$ of Conjecture \ref{cc22}
have a solution
if and only if $F^c$ satisfies Witten's KdV equations.
The agreement of Conjectures \ref{cc11} and \ref{cc22}
is certainly not obvious. However, recent work of Buryak~\cite{Bur} proves they are equivalent. Moreover, Buryak proves the consistency of the open KdV equations.

Using our construction of the descendent theory of pointed disks,
we prove the  genus 0 part of Conjecture \ref{cc22}.

\begin{thm} \label{tt22} The open analogs of the KdV equations
hold in genus zero. Namely,
\begin{equation*}
(2n+1)\blangle \blangle \tau_n  \brangle \brangle_0^o
=  \blangle \blangle \tau_{n-1} \tau_0\brangle \brangle^c_0
\blangle \blangle \tau_0  \brangle \brangle^o_0
+ 2    \bblangle \tau_{n-1}  \bbrangle^o_0
\bblangle\sigma\bbrangle^o_0 \
\end{equation*}
for $n\geq 1$.
\end{thm}

A complete proposal for a theory of descendent integration in higher genus
will be presented in a forthcoming paper by J.S. and R.T.~\cite{JSRT}.
Via the construction of~\cite{JSRT}, R.T. has found a combinatorial formula that allows effective calculation of the descendent integrals in arbitrary genus~\cite{RT}. Several months after the first version of this paper appeared, Conjectures~\ref{cc11} and~\ref{cc22} were proved by A. Buryak and R.T.~\cite{ABRT} using the combinatorial formula of~\cite{RT}. A matrix model for refined open descendent integrals distinguishing contributions from surfaces with different numbers of boundary components has been conjectured in~\cite{ABT17}.

\subsection{Formulas in genus 0}\label{sec:formulae}
Descendent integration over the moduli space of compact
genus 0 Riemann surfaces with marked points has a very
simple answer,
$$\blangle \tau_{a_1} \ldots \tau_{a_l} \brangle_0^c =
\binom{l-3}{a_1,\ldots,a_l}  \ .$$
The above evaluation is easily derived from the string
equation for $F^c$ and the initial value
$$\blangle \tau_0^3 \brangle_0^o = 1 \ .$$
Alternatively, the evaluation can be derived from the
topological recursion relations \cite{Witten} for $F^c_0$.

A explicit evaluation also can be obtained for the open invariants
$$\blangle \tau_{a_1} \ldots \tau_{a_l} \sigma^k\brangle_0^o$$
in genus 0.
Using the string equation for $F^o$ of Theorem~\ref{thm:string_dilaton},
we can assume $a_i\geq 1 $ for all $i$.
By the dimension constraint,
$$-3+k+2l = \sum_{i=1}^l 2a_i\ .$$

\begin{thm}\label{ttxx}
We have the evaluation
$$\blangle \tau_{a_1} \ldots \tau_{a_l} \sigma^k\brangle_0^o =
\frac{(\sum_{i=1}^l 2a_i -l +1)!}{\prod_{i=1}^l (2a_i-1)!!} \ $$
in case $a_i \geq 1$ for all $i$.
\end{thm}

\noindent The double factorial of an odd positive integer is the product
of all odd integers not exceeding the argument,
$$9!!= 9 \cdot 7 \cdot 5 \cdot 3 \cdot 1\ .$$
While
such double factorials also occur~\cite{GetzlerPan} in the
formula for $\lambda_g\lambda_{g-1}$ descendent integrals
over the moduli space of $\oCM_{g,l}$ of higher
genus curves,
a direct connection is not known to us.

We derive Theorem \ref{ttxx} as a consequence of the following topological recursion relations for the open theory in genus 0.

\begin{thm}\label{thm:trr1_2}
For $n>0$, two topological recursion relations hold for~$F^o_0$:
\begin{align}
\left\langle\left\langle \tau_n \sigma \right\rangle\right\rangle^o_0 &= \left\langle\left\langle \tau_{n-1} \tau_0 \right\rangle\right\rangle^c_0 \left\langle\left\langle \tau_0 \sigma \right\rangle\right\rangle^o_0 + \left\langle\left\langle \tau_{n-1} \right\rangle\right\rangle^o_0
\left\langle\left\langle\sigma^2 \right\rangle\right\rangle^o_0, \tag{\emph{TRR I}} \label{TRRI} \\
\left\langle\left\langle \tau_n \tau_m \right\rangle\right\rangle^o_0 &= \left\langle\left\langle \tau_{n-1} \tau_0 \right\rangle\right\rangle^c_0 \left\langle\left\langle \tau_0 \tau_m \right\rangle\right\rangle^o_0 + \left\langle\left\langle \tau_{n-1} \right\rangle\right\rangle^o_0
\left\langle\left\langle\tau_m \sigma\right\rangle\right\rangle^o_0. \tag{\emph{TRR II}} \label{TRRII}
\end{align}
\end{thm}
As noted in~\cite{BaB19}, a straightforward manipulation of the above topological recursion relations gives the following.
\begin{cor}\label{cor:owdvv}
The following two open WDVV relations hold for $F_0^o:$
\[
\left\langle\left\langle \tau_m \tau_n \tau_0\right\rangle \right \rangle^c_0 \left\langle\left\langle \tau_0 \sigma \right\rangle \right \rangle^o_0 + \left\langle\left\langle \tau_m \tau_n \right\rangle \right \rangle^o_0 \left\langle\left\langle \sigma^2 \right\rangle \right \rangle^o_0 = \left\langle\left\langle \tau_m \sigma \right\rangle \right \rangle^o_0 \left\langle\left\langle \tau_n \sigma \right\rangle \right \rangle^o_0, 
\]
\begin{multline*}
 \left\langle\left\langle \tau_l \tau_m \tau_0 \right\rangle \right \rangle^c_0 \left\langle\left\langle \tau_0 \tau_n \right\rangle \right \rangle^o_0 + \left\langle\left\langle \tau_l \tau_m \right\rangle \right \rangle^o_0 \left\langle\left\langle \tau_n \sigma \right\rangle \right \rangle^o_0 =  \\ =\left\langle\left\langle \tau_m \tau_n \tau_0 \right\rangle \right \rangle^c_0 \left\langle\left\langle \tau_0 \tau_l \right\rangle \right \rangle^o_0 + \left\langle\left\langle \tau_m \tau_n \right\rangle \right \rangle^o_0 \left\langle\left\langle \tau_l \sigma \right\rangle \right \rangle^o_0.
\end{multline*}
\end{cor}

\subsection{Context and motivation}
\subsubsection{Boundary conditions}\label{sssec:intuition}
The original motivation for the definition of the boundary conditions of this paper was the construction of open Gromov-Witten invariants by Liu~\cite{Li03} in the presence of an $S^1$ action on the target space. Roughly speaking, Liu shows that the induced $S^1$ action on the boundary of the moduli space of stable $J$-holomorphic disks maps is free because it acts non-trivially on boundary nodes. So one can reduce the dimension of the boundary from real codimension $1$ to real codimension $2$ by taking the quotient. This is crucial for the definition of numerical invariants. In the context of the open descendent integrals constructed here, the target space is a point, which does not admit a non-trivial $S^1$ action. However, there is still a $1$-dimensional foliation of the boundary of the moduli space of stable disks reminiscent of the orbits of the induced $S^1$ action on the boundary of the moduli space used in Liu's construction. The leaves are paths traced when boundary nodes of stable maps are allowed to change in a certain way. We define the non-vanishing section $\s$ of $E|_{\partial\oCM_{0,k,l}}$ mentioned in Section~\ref{ssec:iconst} by the property that it is pulled back from the leaf space of the foliation. The leaf space is of real dimension two less than the real rank of $E,$ so any two non-vanishing sections pulled back from the leaf space can be connected by non-vanishing homotopy. Consequently, the open descendent integrals do not depend on the choice of $\s.$

\subsubsection{Mathematical and physical corroboration}
The open descendent integrals fit nicely in the broader field of open Gromov-Witten theory. The open WDVV equations for descendent integrals of Corollary~\ref{cor:owdvv} are closely analogous to the open WDVV equations for open Gromov-Witten invariants that have been studied extensively in~\cite{Che18,ChZ21,HS12,So07,ST19}. Open descendent integrals enter naturally in Zernik's~\cite{Zer17b,Zer17a} fixed point localization formula for equivariant open Gromov-Witten invariants of even dimensional projective spaces. The non-equivariant limit of Zernik's formula in dimension two gives Welschinger's invariants~\cite{We05} of the projective plane. Since the open string, dilaton, topological recursions, KdV equations and Virasoro constraints are highly overdetermined, their existence is in itself an indication of the naturality of the definition of open descendent integrals.


From a physics perspective, Witten and Dijkgraaf~\cite{DiW18} have explained how the open descendent integrals of the present work arise in the context of topological field theory in the spirit of topological sigma models~\cite{Wi88}. They have also derived the open Virasoro constraints from random matrix models for topological gravity with vector degrees of freedom that arise from including open strings. Connections with condensed matter physics are described as well.

\subsubsection{Open versus closed}\label{sssec:ovc}
The recursion relations for the open descendent integrals exhibit behavior that does not have a direct analog in the closed theory.
This is most clearly visible in the open topological recursion relations of Theorem~\ref{thm:trr1_2}. On the one hand, the closed-open terms
\begin{equation}\label{eq:coterms}
\left\langle\left\langle \tau_{n-1} \tau_0 \right\rangle\right\rangle^c_0 \left\langle\left\langle \tau_0 \sigma \right\rangle\right\rangle^o_0, \qquad \left\langle\left\langle \tau_{n-1} \tau_0 \right\rangle\right\rangle^c_0 \left\langle\left\langle \tau_0 \tau_m \right\rangle\right\rangle^o_0,
\end{equation}
from equations~\eqref{TRRI} and~\eqref{TRRII} are analogous to the right hand side of the closed topological recursion relations,
\begin{equation}\label{eq:ctrr}
\left\langle\left\langle \tau_{n} \tau_m \tau_l \right\rangle\right\rangle^c_0 =
\left\langle\left\langle \tau_{n-1} \tau_0 \right\rangle\right\rangle^c_0 \left\langle\left\langle \tau_0 \tau_m \tau_l\right\rangle\right\rangle^c_0.
\end{equation}
In both cases, the product of correlators arises from a stratum $T^{i}$ in the relevant moduli space consisting of two component stable Riemann surfaces with a single interior node. The interior node contributes an interior marked point to each component of the surface, and hence a $\tau_0$ insertion in each correlator of the product.

On the other hand, the open-open terms
\begin{equation}\label{eq:ooterms}
\left\langle\left\langle \tau_{n-1} \right\rangle\right\rangle^o_0
\left\langle\left\langle\sigma^2 \right\rangle\right\rangle^o_0, \qquad \left\langle\left\langle \tau_{n-1} \right\rangle\right\rangle^o_0
\left\langle\left\langle\tau_m \sigma\right\rangle\right\rangle^o_0
\end{equation}
from equations~\eqref{TRRI} and~\eqref{TRRII} exhibit unexpected behavior. The product of correlators arises from a stratum $T^{b}$ in the stable disk moduli space consisting of two component stable disks with a boundary node. Unlike the case of the interior node, the boundary node contributes a boundary marked point to each component of the surface, but only contributes a $\sigma$ insertion to one correlator of the product.

The unusual form of the open-open terms reflects a fundamental difference in the way they arise. The proof of both the open and closed topological recursion relations uses a special section $t$ of one of the cotangent line bundles. The section $t$ vanishes on the stratum~$T^{i}$ giving rise to the closed-open terms~\eqref{eq:coterms} and the right hand side of~\eqref{eq:ctrr}. Over the locus $T^b$, the section $t$ does not vanish, but rather it fails to agree with the boundary condition. The open-open terms~\eqref{eq:ooterms} quantify the lack of agreement.

The open dilaton equation~\eqref{eq:dilaton} also exhibits unexpected behavior. According to the closed dilaton equation~\eqref{eq:cdil}, a $\tau_1$ insertion multiplies the value of a descendent integral by the Euler characteristic of the relevant surface after removing marked points. A natural guess would be that a $\tau_1$ insertion multiplies the value of an open descendent integral by one-half the Euler characteristic of the closed double surface after removing marked points. That would result in a factor of $(g-1 + k/2 +l)$ in~\eqref{eq:dilaton} instead of the correct factor of $(g-1 + k + l).$ It will be seen in the proof that this discrepancy directly reflects the definition of the boundary conditions.

\subsubsection{The obstacle to conjugation symmetry arguments}
Before considering the approach to open descendent integrals based on boundary conditions presented in this paper, we attempted to use a complex conjugation symmetry argument as in~\cite{Cho,Ge16,So06}. Roughly speaking, the argument seeks to cancel certain boundary components of the stable disk moduli space by gluing pairs of components related by complex conjugation. For this to work, the gluing maps must always reverse orientation relative to the bundle $E$ of~\eqref{eq:intrE}. However, in the case of open descendent integrals, the signs of the gluing maps vary. Equivalently, the possibility of defining open descendent integrals by integration over the moduli space of real stable curves obtained by a doubling construction is obstructed by lack of orientability.

The phenomenon of boundaries that cannot be cancelled by complex conjugation symmetry appears also in the context of open Gromov-Witten invariants in dimensions greater than three~\cite{ST21}. The problem occurs when the invariants include interior constraints of even complex codimension or boundary constraints. Similarly, the problem with open descendent invariants arises on account of even powers of the classes $\psi_i$ and boundary marked points. Genus zero open descendent invariants necessarily involve boundary marked points. Otherwise, the real dimension of the moduli space is odd, while the cohomology class to be integrated is of even degree.

\subsubsection{Beyond the point}
Building on the present work, Buryak, Clader and R.T.~\cite{BCT20a,BCT20b} define a genus zero open $r$-spin theory and prove an open analog of Witten's conjectures~\cite{Wi91}. Buryak, R.P, R.T. and Zernik~\cite{BCTZ20} define stationary open descendent invariants for $(\C P^1, \R P^1)$ and calculate them by fixed point localization. They also give conjectural formulas for higher genus invariants.

Let $X$ be a symplectic manifold. In forthcoming work of Giterman and J.S.~\cite{GiS}, descendent open Gromov-Witten invariants are defined for a large class of Lagrangian submanifolds $L \subset X.$ The definition recasts the construction of the present paper in the general framework for open Gromov-Witten theory developed by J.S. and Tukachinsky~\cite{ST21}. The foliation used in defining the boundary conditions for open descendent invariants is related to the unit in a deformation of the Fukaya $A_\infty$ algebra of $L$ that arises from descendent classes. The unit plays a role in the definition of weak bounding cochains, which are used to cancel boundaries of moduli spaces of stable disk maps. A major challenge arises because the strict cyclic structure of the usual Fukaya $A_\infty$ algebra deforms only in the homotopy sense in the presence of descendent classes.

Basalaev and Buryak~\cite{BaB19} have formulated a conjectural extension of the open Virasoro constraints of the present work in any situation where the open WDVV equations hold. For example, J.S. and Tukachinsky~\cite{ST19} showed that the genus zero open Gromov-Witten invariants of~\cite{ST21} satisfy the open WDVV equations. So, the Virasoro constraints of Basalaev-Buryak predict the behavior of conjectural higher genus open Gromov-Witten invariants in that context.

\subsubsection{Descendents in SFT}
Descendent classes in the context of moduli spaces with boundary have previously been considered in the work of Fabert~\cite{Fab11} and Fabert-Rossi~\cite{FaR11,FaR13} on symplectic field theory (SFT), based on ideas of Eliashberg~\cite{Eli07}. A collection of homological invariants is defined that satisfy analogs of the divisor, dilaton and string equations as well as a version of the topological recursion relations.

A similar construction could be used to deform the Fukaya $A_\infty$ algebra of a Lagrangian submanifold. However, to define descendent open Gromov-Witten invariants within the framework of~\cite{ST21} additional steps are required. The descendent deformation of the Fukaya $A_\infty$ algebra must preserve the unit and the cyclic structure. The cyclic structure can only be deformed in the homotopy sense. Moreover, one must construct a canonical class of bounding cochains in the deformed algebra. These constructions are carried out in~\cite{GiS}. The open descendent integrals of the present work can be thought of as the descendent open Gromov-Witten invariants of the point. The structures and constructions used in the general definition of descendent open Gromov-Witten invariants were found in part by unwinding the definition of boundary conditions given here.

The works on descendent classes in SFT~\cite{Fab11,FaR11,FaR13} consider the algebraic structures in SFT that arise from Riemann surfaces without boundary. These structures are analogous to closed Gromov-Witten theory. There are also algebraic structures in SFT that arise from Riemann surfaces with boundary, analogous to the Fukaya category and open Gromov-Witten theory. Some aspects of the structures in SFT arising from surfaces with boundary are outlined in Section~2.8 of~\cite{EGH00} with an emphasis on the new challenges that arise.
The open descendent integrals of the present work belong to open Gromov-Witten theory. As explained in Section~\ref{sssec:ovc}, phenomena that are unique to Riemann surfaces with boundary play a central role.

\subsection{Plan of the paper}
In Section \ref{sec:mod} we review the moduli space of stable marked disks and discuss stable graphs.
In Section \ref{sec:bdry} we define the canonical boundary conditions and the open descendent integrals. We then define the more subtle special canonical boundary conditions and show they exist. We prove the string and dilaton equations and the topological recursion relations using geometric methods in Section~\ref{sec:geo_rec}. In Sections~\ref{VirSec} and~\ref{KDVSec}, we prove the genus $0$ open Virasoro relations and KdV equations using the string and dilaton equations and assuming the genus $0$ formula of Theorem~\ref{ttxx}. Finally, in Section~\ref{EVSec} we prove the genus $0$ formula using the open topological recursion relations and the dilaton equation.

\subsection{Acknowledgments}
We thank A. Buryak, Y.-P. Lee, A. Okounkov, Y. Ruan, S. Shadrin, A. Solomon, J. Walcher, E. Witten and A. Zernik,
for discussions related to the work presented here.

R.P. was supported by SNF-200020-182181, ERC-2017-AdG-786580-MACI, and SwissMAP. The project has received funding from the European Research Council (ERC) under the European Union Horizon 2020 research and innovation program (grant agreement No 786580). J.S. and R.T. were supported by BSF grant 2008314 and ERC Starting Grant 337560. J.S. was supported by ISF grant 569/18.

The project was started at Princeton University in 2007 with support by the NSF. Part of the work was completed during visits of J.S. and R.T. to the Forschungsinstitut f\"ur Mathematik at ETH Z\"urich in 2013.

\section{Moduli of disks}\label{sec:mod}
\subsection{Conventions}
We begin with some useful notations and comments.
\begin{nn}
Throughout this paper the notation $\dim_\C$ ($rk_\C$) will mean $\frac{\dim_\R}{2}$ ($\frac{rk_\R}{2}$).
\end{nn}
Throughout this paper whenever we say a manifold, unless specified otherwise, we mean a smooth manifold with corners in the sense of \cite{Joyce}. Similarly, notions which relate to manifolds or maps between them are in accordance with that article.
\begin{nn}
We write $\Delta$ for the standard unit disk in $\C,$ with the standard complex structure.
\end{nn}

\begin{nn}
For a set $A$ denote by $\rz{A}$ the set
\[
\left\{\rz{x}~\text{for}~x\in A\right\}.
\]
For $l\in\N$, we use the notation $[l]$ to denote $\left\{1,2,\ldots,l\right\}.$ We write $[0]$ for the empty set.
We also denote by $[\rz{l}]$ the set $\rz{[l]}.$
\end{nn}
\begin{nn}
For a set $A$ write $2^A_{fin}$ for the collection of finite subsets of $A$.
We say that $B\subseteq 2^A_{fin}$ is a \emph{disjoint subset} if its elements are pairwise disjoint.
\end{nn}
\begin{nn}
Put $\LL = 2^{\Z\cup\rz{\Z}}_{fin}.$ Throughout the article we identify $i\in\Z\cup\rz{\Z}$ with $\left\{i\right\}\in\LL,$ without further mention.
We denote by $2^\LL_{fin,disj}$ the collection of finite disjoint subsets $A$ of $\LL$, such that $\emptyset \notin A.$ For $A\in 2^\LL_{fin,disj},$ let $\cup A \in \LL$ denote the union of its elements as sets.
\end{nn}

\begin{rmk}
We use the set $\LL$ to label the marked points and nodal points of stable disks. Here is an informal description of how the labeling procedure works.
Typically, we begin with a stable disk with boundary marked points labeled by single element subsets of $\Z^\circ$ and interior marked points labeled by single element subsets of $\Z.$ See, for example, Notation~\ref{nn:0BI}. Each nodal point of the stable disk, whether boundary or interior, is labeled canonically on each of the two components to which it belongs by a finite set. The finite set is the union of the labels of the marked points on the components of the stable disk that can be reached by starting from the given component and passing through the given nodal point. A formal description of the labeling procedure for nodal points is given in Definition~\ref{df:iv}. Ultimately, we split the stable disk into subdisks by cutting it apart at one or more nodal points. The nodal points at which the splitting occurs become marked points on subdisks. The new marked points inherit the canonical labels of their parent nodal points. The splitting procedure is formalized in Section~\ref{subs:defs_bundles} and more generally in Definition~\ref{def:base_operator}.
\end{rmk}

\subsection{Stable disks}
Throughout the paper markings will be taken from $\LL$.
We recall the notion of a \emph{stable marked disk}.

\begin{definition}
Given $B,I\in 2^\LL_{fin,disj}$ with $B \cap I = \emptyset$ and $B\cup I$ disjoint, we define a $(B,I)-$\emph{marked smooth surface} to be a triple
\[
\left(\Sigma, \{z_i\}_{i\in B}, \{z_i\}_{i\in I}\right)
\]
where
\begin{enumerate}
\item
$\Sigma$ is a Riemann surface with boundary.
\item
For each $i\in B,~z_i\in \partial \Sigma.$
\item
For each $i\in I,~z_i\in \mathring{ \Sigma}.$
\end{enumerate}
We call \emph{B the set of boundary labels}. We call \emph{I the set of interior labels}.
\end{definition}
We sometimes omit the marked points from our notations. Given a smooth marked surface $\Sigma$, we write $B\left(\Sigma\right)$ for the set of its boundary labels. We also use $B\left(\Sigma\right)$ to denote the set of boundary marked points of $\Sigma.$ Similarly, we write $I\left(\Sigma\right)$ the set of interior labels of $\Sigma$, and again, we also write $I\left(\Sigma\right)$ for the set of internal marked points of $\Sigma.$

\begin{definition}\label{def:prestable}
Given $B,I\in 2^\LL_{fin,disj}$ with $B \cap I = \emptyset$ and $B\cup I$ disjoint, a $(B,I)-$\emph{pre-stable marked} genus $0$ surface is a tuple
\[
\Sigma = \left(\left\{\Sigma_\alpha\right\}_{\alpha\in\mathcal{D}\coprod \mathcal{S}}, \sim_B , \sim_I\right),
\]
where
\begin{enumerate}
\item
$\mathcal{D}$ and $\mathcal{S}$ are finite sets.
For $\alpha\in\mathcal{D},$ $\Sigma_\alpha$ is a smooth marked disk; for $\alpha\in\mathcal{S},$ $\Sigma_\alpha$ is a smooth marked sphere.
\item
An equivalence relation $\sim_B$ on the set of all boundary marked points, with equivalence classes of size at most $2$. An equivalence relation $\sim_I$ on the set of all internal marked points, with equivalence classes of size at most $2$.
\end{enumerate}
The two equivalence relations $\sim_B$ and $\sim_I$ taken together are denoted by $\sim.$ The above data  satisfies
\begin{enumerate}
\item
$B$ is the set of labels of points belonging to $\sim_B$ equivalence classes of size $1.$ $I$ is the set of labels of points belonging to $\sim_I$ equivalence classes of size $1.$
\item
The topological space $\coprod_{\alpha\in\mathcal{D}\cup\mathcal{S}}\Sigma_\alpha/\!\!\sim$ is connected and simply connected.
\item\label{it:g0}
The topological space $\coprod_{\alpha\in\mathcal{D}}\Sigma_\alpha/\!\!\sim_B$ is connected or empty.
\end{enumerate}
We also write $\Sigma = \coprod_{\alpha\in\mathcal{D}\cup\mathcal{S}}\Sigma_\alpha/\!\!\sim.$ If $\mathcal{D}$ is empty, $\Sigma$ is called a \emph{pre-stable marked sphere}. Otherwise it is called a \emph{pre-stable marked disk}.
We denote by $\mathcal{M}_B \left(\Sigma_\alpha\right)$ the set of labels of boundary marked points of $\Sigma_\alpha$ which belong to $\sim_B$ equivalence classes of size~$1.$ We define $\mathcal{M}_I \left(\Sigma_\alpha\right)$ similarly. The $\sim_B$ (resp. $\sim_I$) equivalence classes of size $2$ are called boundary (resp. interior) nodes.

A smooth marked disk $D$ is called \emph{stable} if
\[
|B\left(D\right)|+2|I\left(D\right)|\geq 3.
\]
A smooth marked sphere is \emph{stable} if it has at least $3$ marked points.
A pre-stable marked genus $0$ surface is called a \emph{stable marked genus $0$ surface} if each of its constituent smooth marked spheres and smooth marked disks are stable.
\end{definition}

\begin{nn}
In case $B = \rz{A}$ for some $A$, we denote the marked point~$z_{\rz{i}},$ for $\rz{i}\in B,$ by $x_i.$ In this case we also use the notation $\left(\Sigma, \mathbf x, \mathbf z\right)$ to denote a stable marked surface, where $\mathbf x = \{x_i\}_{i^\circ \in B(\Sigma)}$ and $\mathbf z = \{z_i\}_{i \in I(\Sigma)}.$
\end{nn}

\begin{definition}
Let $\Sigma,\Sigma',$ be stable marked genus $0$ surfaces with $B\left(\Sigma\right) = B\left(\Sigma'\right)$ and $I\left(\Sigma\right) = I\left(\Sigma'\right).$ An \emph{isomorphism} $f : \Sigma \to \Sigma'$ is a homeomorphism such that
\begin{enumerate}
\item
For each $\alpha \in \mathcal D \cup \mathcal S,$ the restriction $\left.f\right|_{\Sigma_\alpha}$ maps $\Sigma_\alpha$ biholomorphically to some $\Sigma_{\alpha'}$ for $\alpha' \in \mathcal D' \cup \mathcal S'.$
\item
For each $i \in B\left(\Sigma\right)\cup I\left(\Sigma\right)$ there holds $f\left(z_i\right) = z'_i.$
\end{enumerate}
\end{definition}

\begin{rmk}
The automorphism group of a stable marked genus $0$ surface is trivial.
\end{rmk}

\subsection{Stable graphs}
It is useful to encode some of the combinatorial data of stable marked disks in graphs.
\begin{definition}\label{def:stg}
A (not necessarily connected, genus $0$) \emph{pre-stable graph $\Gamma$} is a tuple
$\left(V = V^{O}\cup V^{C} ,E,\ell_I,\ell_B\right)$,
where
\begin{enumerate}
\item
$V^O , V^C$, are finite sets.
\item
$E$ is a subset of the set of (unordered) pairs of elements of $V$.
\item
$\ell_I : V\rightarrow 2^\LL_{fin,disj}$, $\ell_B : V^O\rightarrow2^\LL_{fin,disj}$.
\end{enumerate}
We call the elements of $V$ the \emph{vertices} of $\Gamma$, where $V^O$ are the \emph{open vertices}, and $V^C$  are the \emph{closed vertices}. We call the elements of $E$ the \emph{edges} of $\Gamma$. An edge between open vertices is called a \emph{boundary edge}. The other edges are called \emph{interior edges}. We call $\ell_I\left(v\right)$ the \emph{interior labels} of $v$, and $\ell_B\left(v\right)$ the \emph{boundary labels}.
We demand that $\Gamma$ satisfies
\begin{enumerate}
\item
The graph $\left(V,E\right)$ is a forest, namely, a collection of trees.
\item\label{it:blabla}
If $v,u\in V^O$ belong to the same connected component of $\Gamma$, they also belong to the same connected component in the subgraph of $\Gamma$ spanned by $V^O$.
\item
The sets $\ell_I\left(v\right)$ for $v\in V$ and $\ell_B(v)$ for $v \in V^O$ are collectively pairwise disjoint. That is, labels are unique.
\item\label{it:annoyance}
\begin{enumerate}
\item\label{it:1ct}
For $W \subset V$ spanning a connected component of $\Gamma,$ the subset $\cup_{v \in W} (\ell_B(v) \cup \ell_I(v)) \subset 2^\LL_{fin}$ is disjoint.
\item\label{it:2ct}
For $i = 1,2,$ let $W_i \subset V$ span connected components of $\Gamma$ and let $U_i \subset \cup_{v \in W_i} (\ell_B(v) \cup \ell_I(v))$ be proper subsets that are disjoint. Then $\cup U_1 \neq \cup U_2.$
\end{enumerate}
\end{enumerate}
We say that $\Gamma$ is connected if its underlying graph, $\left(V,E\right)$ is connected.
\end{definition}

\begin{rmk}\label{rmk:4}
Condition \ref{it:blabla} above means that each connected component of closed vertices is a tree rooted in a neighbor of an open vertex. Other vertices in this tree have no open neighbors. The root has a unique open neighbor. This is a combinatorial analog of the geometric condition~\ref{it:g0} of Definition~\ref{def:prestable}.
\end{rmk}

\begin{rmk}
Condition~\ref{it:annoyance} is designed to achieve the following:
\begin{itemize}
\item
The operator $\CB$ of Definition~\ref{def:base_operator} takes stable graphs to stable graphs.
\item
The operator $\partial$ of Definition~\ref{def:bdry_maps} takes stable graphs to stable graphs.
\end{itemize}
Part~\ref{it:1ct} ensures the label sets $\ell_I(v),\ell_B(v) \subset 2^\LL_{fin},$ remain disjoint under the above operations. Part~\ref{it:2ct} ensures labels remain unique.
\end{rmk}

\begin{nn}
For a vertex $v\in V$, denote by $E_v \subseteq E$ the set of edges containing $v.$ We denote by $E_v^{I}$ the set of interior edges of $v$ and by $E^I$ the set of all interior edges of $\Gamma.$
For $v\in V^{O}$, denote by $E_v^{B}$ the set of boundary edges of $v$. Denote by $E^B$ the set of all boundary edges of $\Gamma.$
We define
\[
B\left(v\right) = \ell_B\left(v\right)\cup E_v^{B}, \qquad I\left(v\right) = \ell_I\left(v\right)\cup E_v^{I},
\]
and we set $k\left(v\right) = \left|B\left(v\right)\right|, l\left(v\right) = \left|I\left(v\right)\right|.$
We also write \[
B\left(\Gamma\right) = \cup_{v\in V^O}\ell_B\left(v\right),\qquad I\left(\Gamma\right) = \cup_{v\in V}\ell_I\left(v\right).
\]
We define $k\left(\Gamma\right) = \left|B\left(\Gamma\right)\right|, l\left(\Gamma\right) = \left|I\left(\Gamma\right)\right|$.
Finally, if $i\in I\left(\Gamma\right)$ we define $v_i = v_i\left(\Gamma\right)\in V$ to be the unique vertex $v\in V$ with $i\in\ell_I\left(v\right).$

For $\Gamma$ a pre-stable graph, we write $V\left(\Gamma\right), E\left(\Gamma\right),\ell_I^\Gamma,\ell_B^\Gamma,$ for the sets of vertices, edges, interior labels and boundary labels respectively. Similarly, we write $V^C(\Gamma)$ and so on. We also use analogously defined notation $I^\Gamma(v), B^\Gamma(v).$

Given a pre-stable graph $\Gamma$, we define
\[
\varepsilon = \varepsilon_\Gamma:V\left(\Gamma\right)\to\left\{O,C\right\}
\]
by $\varepsilon\left(v\right) = O$ if and only if $v\in V^O$. In specifying a stable graph, we may specify $\varepsilon$ instead of specifying the partition $V = V^O \cup V^C.$
\end{nn}

Although $\ell_B$ was defined only for boundary vertices, we sometimes write $\ell_B\left(v\right)=\emptyset$ for $v\in V^C.$ Similarly, we set $B\left(\Gamma\right)=\emptyset$ in case $V^O=\emptyset.$
\begin{definition}
An open vertex $v$ in a pre-stable graph $\Gamma$ is called \emph{stable} if
$k\left(v\right)+2l\left(v\right)\geq 3.$
A closed vertex $v$ in a pre-stable graph $\Gamma$ is called \emph{stable} if
$l\left(v\right)\geq 3.$
If all the vertices of $\Gamma$ are stable we say that $\Gamma$ is \emph{stable}.
We denote by $\CG$ the collection of all stable graphs.
\end{definition}

To each stable marked genus $0$ surface $\Sigma$ we associate a connected stable graph as follows. We set $V^O = \mathcal D$ and $V^C = \mathcal S.$ For $v \in V,$ we set
\[
\ell_B \left(v\right) = \mathcal{M}_B\left(\Sigma_v\right),\qquad \ell_I\left(v\right) = \mathcal{M}_I\left(\Sigma_v\right).
\]
An edge between two vertices corresponds to a node between their corresponding components. One easily checks that the associated stable graph is well defined and satisfies all the requirements of the definitions. Moreover, $\Sigma$ is a stable marked disk if and only if $V^O\neq\emptyset$. Otherwise it is a stable marked sphere.
\begin{nn}
The graph associated to a stable disk $\Sigma$ is denoted by $\Gamma\left(\Sigma\right)$.
\end{nn}

\subsection{Smoothing and boundary}
\begin{definition}\label{def:smoothing}
The \emph{smoothing} of a stable graph $\Gamma$ at an edge $e$ is the stable graph
\[
d_e\Gamma = d_{\left\{e\right\}}\Gamma = \Gamma' = \left(V', E', \ell'_I, \ell'_B\right)
\]
defined as follows. Write $e = \left\{u,v\right\}.$ The vertex set is given by
\[
V'= \left(V\setminus\left\{u,v\right\}\right)\cup\left\{uv\right\}.
\]
The new vertex $uv$ is closed if and only if both $u$ and $v$ are closed.
Writing
\begin{equation*}
E_{uv}' = \left\{\left.\left\{w,uv\right\}\right| \left\{w,u\right\}\in E \text{ or }\left\{w,v\right\}\in E\text{ and }w\neq u,v\right\},
\end{equation*}
we set
\[
E' = \left(E\setminus \left(E_u \cup E_v\right)\right) \cup E_{uv}'.
\]
Furthermore,
\begin{align*}
&\ell'_I (w) = \ell_I (w), \qquad w\in V'\setminus\left\{uv\right\}, \\
&\ell'_I (uv) = \ell_I (u)\cup\ell_I (v),
\end{align*}
and similarly for $\ell_B$.

Observe that there is a natural proper injection $E'\hookrightarrow E$, so we may identify $E'$ with a subset of $E$. Using the identification, we extend the definition of smoothing in the following manner.
Given a set $S=\left\{e_1 , \ldots, e_n\right\}\subseteq E\left(\Gamma\right)$, define the smoothing at $S$ as
\[
d_S\Gamma = d_{e_n}\left(\ldots d_{e_2}\left(d_{e_1}\Gamma\right)\ldots\right).
\]
Observe that $d_S\Gamma$ does not depend on the order of smoothings performed.
\end{definition}
Note that in case $\Gamma = d_S\Gamma'$, we have a natural identification between $E\left(\Gamma\right)$ and $E\left(\Gamma'\right)\setminus S$.
\begin{definition}\label{def:bdry_maps}
We define the boundary maps
\begin{gather*}
\partial : \CG \to 2^\CG, \qquad \pu : \CG \to 2^{\CG}, \qquad \pB : \CG \to 2^{\CG},
\end{gather*}
by
\begin{gather*}
\partial\Gamma = \left\{\left.\Gamma' \right| \exists \emptyset\neq S\subseteq E\left(\Gamma'\right),\, \Gamma = d_S\Gamma'\right\},\qquad
\pu \Gamma = \{\Gamma\} \cup \partial \Gamma,\\
\pB\Gamma = \left\{\Gamma' \left| \Gamma' \in \partial \Gamma, \, |E^B(\Gamma')| \geq 1\right.\right\}.
\end{gather*}
Denote also by $\partial$ the map $2^\CG \to 2^\CG$ given by
\[
\partial\{\Gamma_\alpha\}_{\alpha \in A}  =  \bigcup_{\alpha \in A} \partial \Gamma_\alpha.
\]
and similarly for $\pu,\pB$.
\end{definition}

\subsection{Moduli and orientations}
\begin{nn}\label{nn:0BI}
For $B,I\in 2^\LL_{fin,disj}$ with $B \cap I = \emptyset$ and $B\cup I$ disjoint, denote by $\oCM_{0,B,I}$ the set of isomorphism classes of stable marked disks whose set of boundary labels is $B$ and whose set of interior labels is $I.$
Denote by $\CM_{0,B,I}$ the subset of $\oCM_{0,B,I}$ consisting of isomorphism classes of smooth marked disks.
We denote by $\oCM_{0,I}$ the set of isomorphism classes of stable marked spheres whose label set is $I.$
Let $\CM_{0,I}$ be the set of isomorphism classes of smooth marked spheres with label set $I$.
For $\Gamma \in \CG$, denote by $\CM_\Gamma$ the set of isomorphism classes of stable marked genus zero surfaces with associated graph $\Gamma$.
Define
\[
\oCM_\Gamma = \coprod_{\Gamma'\in\pu\Gamma}\CM_{\Gamma'}.
\]
We abbreviate $\oCM_{0,k,l} = \oCM_{0,[\rz{k}],[l]}$. We may also write $\oCM_{0,k,I}$, $\oCM_{0,B,l},$ with the obvious meanings. Similarly, we abbreviate $\oCM_{0,n} = \oCM_{0,[n]}$.
\end{nn}
When we say that a stable marked disk belongs to $\oCM_{0,B,I},$ we mean that its isomorphism class is in $\oCM_{0,B,I}$. The same applies to the other sets defined above as well.
\begin{nn}\label{nn:G0kl}
Given nonnegative integers $k,l$ with $k+2l\geq 3,$ denote by $\Gamma_{0,k,l}$ the stable graph with $V^O = \left\{*\right\}, V^C = \emptyset,$ and with
\[
\ell_B\left(*\right)=[\rz{k}], \qquad \ell_I\left(*\right) = [l].
\]
\end{nn}
\begin{rmk}
The above moduli of stable marked disks are smooth manifolds with corners. We have
\[
\dim_\R \oCM_{0,k,l} = k + 2l - 3.
\]
A stable marked disk with $b$ boundary nodes belongs to a corner of the moduli space $\oCM_{0,k,l}$ of codimension $b.$ Thus $\partial \oCM_{0,k,l}$ consists of stable marked disks with at least one boundary node. That is,
\[
\partial\oCM_{0,k,l} = \coprod_{\Gamma \in \pB\Gamma_{0,k,l}} \CM_{\Gamma}.
\]
For foundational work on the manifold with corners structure of moduli spaces of stable marked disks, several references are available. One approach is Theorem~7.1.44 of~\cite{FO09}. A treatment based on hyperbolic geometry is given in Section~4 of~\cite{Li03}, which considers moduli spaces of stable marked bordered Riemann surfaces of arbitrary genus. Moduli spaces of stable disk maps to homogenous spaces are studied in~\cite{Zer17} using complex geometry. In the special case the target symplectic manifold and Lagrangian submanifold are both the point, one obtains moduli spaces of stable marked disks.
\end{rmk}
In the following, building on the discussion in~\cite[Section 2.1.2]{FO09}, we describe a natural orientation on the spaces $\oCM_{0,k,l}$ for $k$ odd. We start by recalling a few useful facts and conventions. Let $\Sigma$ be a genus zero smooth marked surface with boundary and denote by $j$ its complex structure. For $p \in \Sigma,$ and $v \in T_p\Sigma,$ we follow the convention that $\{v,jv\}$ is a complex oriented basis. The complex orientation $\Sigma$ induces an orientation of $\partial\Sigma$ by requiring the outward normal at $p \in \partial\Sigma$ followed by an oriented vector in $T_p\partial\Sigma$ to be an oriented basis of $T_p\Sigma.$ The orientation of $\partial\Sigma$ gives rise to a cyclic order of the boundary marked points. Denote by $\CMm_{0,k,l} \subset \CM_{0,k,l}$ the component where the induced cyclic order on the boundary marked points is the usual order on $[k].$  Denote by $\oCMm_{0,k,l}$ the corresponding component of $\oCM_{0,k,l}.$

The fiber of the forgetful map $\oCM_{0,k,l+1} \to \oCM_{0,k,l}$ is homeomorphic to a disk. It inherits the complex orientation from an open dense subset that carries a tautological complex structure. For $k \geq 1,$ the fiber of the forgetful map $\oCM_{0,k+1,l} \to \oCM_{0,k,l}$ is homeomorphic to a closed interval. An open subset of this closed interval comes with a canonical embedding into the boundary of a disk, which induces an orientation on the fiber. The fiber of the forgetful map $\oCMm_{0,k+2,l} \to \oCMm_{0,k,l}$ is homeomorphic to $[0,1]^2.$
\begin{figure}[ht]
\centering
\includegraphics[scale=.15]{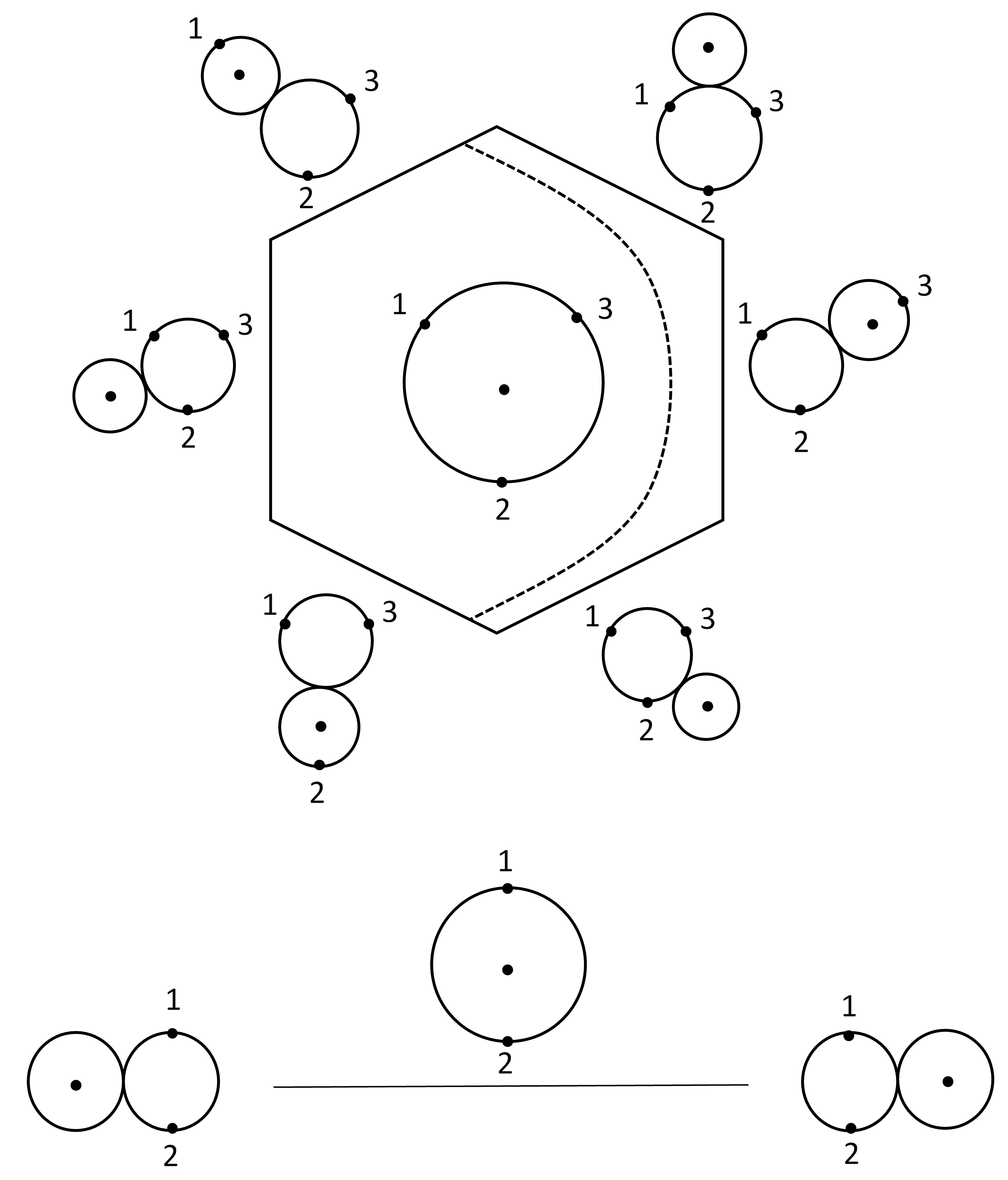}
\caption{The hexagon represents $\oCMm_{0,3,1}$ and the line segment represents $\oCMm_{0,2,1}.$ The dotted line represents a typical fiber of the forgetful map. Stable marked disks are shown that represent a typical point of each space as well as a typical point of each boundary component.}
\label{fig:fibers}
\end{figure}
\begin{ex}
The space $\oCMm_{0,2,1} = \oCM_{0,2,1}$ is diffeomorphic to an interval. The endpoints are given by products of the form $\oCM_{0,1,1} \times \oCM_{0,3,0}.$ See the bottom of Figure~\ref{fig:fibers}.
The space $\oCMm_{0,2,1}$ is the fiber of the forgetful map $\oCMm_{0,2,1} \to \oCMm_{0,1,1}.$

The moduli space $\oCMm_{0,3,1}$ is diffeomorphic to a hexagon. Boundary components can be either $\oCMm_{0,3,0} \times \oCMm_{0,2,1}$ or $\oCMm_{0,4,0} \times \oCMm_{0,1,1}.$ See the top of Figure~\ref{fig:fibers}.


The space $\oCMm_{0,3,1}$ is the fiber both of the forgetful map $\oCMm_{0,3,1} \to \oCMm_{0,1,1}$ and the forgetful map $\oCMm_{0,3,1} \to \oCMm_{0,3,0}.$ It is homeomorphic but not diffeomorphic to both $D^2$ and $[0,1]^2.$ The dotted line in the hexagon in Figure~\ref{fig:fibers} represents the fiber of the forgetful map $\oCMm_{0,3,1} \to \oCMm_{0,2,1}.$ It is diffeomorphic to the interval $[0,1].$ The fiber over the right hand endpoint of the interval is the union of three boundary components, so it is not a smooth submanifold, but it is still homeomorphic to the interval $[0,1].$
\end{ex}

We fix the orientation of the fiber of the forgetful map $\oCMm_{0,k+2,l} \to \oCMm_{0,k,l}$ by identifying the first factor of $[0,1]^2$ with the fiber of the map forgetting the $k+1$ marked point and the second factor with the fiber of the map forgetting the $k+2$ marked point. This orientation is called the \emph{natural orientation} below.
\begin{lm}\label{lm:or}
Let $k$ be odd and $l$ arbitrary. There exists a unique collection of orientations $o_{0,k,l}$ for the spaces $\oCM_{0,k,l}$ with the following properties:
\begin{enumerate}
\item\label{it:zero}
In the zero dimensional cases $k = 1,l=1,$ and $k = 3,l=0,$ the orientations $o_{0,k,l}$ are positive at each point.
\item\label{it:perm}
$o_{0,k,l}$ is invariant under permutations of interior and boundary labels.
\item\label{it:int}
$o_{0,k,l+1}$ agrees with the orientation induced from $o_{0,k,l}$ by the fibration $\oCM_{0,k,l+1} \to \oCM_{0,k,l}$ and the complex orientation on the fiber.
\item\label{it:bdry}
$o_{0,k+2,l}$ agrees with the orientation induced from $o_{0,k,l}$ by the fibration $\oCMm_{0,k+2,l} \to \oCMm_{0,k,l}$ and the natural orientation on the fiber.
\end{enumerate}
\end{lm}
\begin{rmk}
In the preceding lemma, it does not matter which ordering convention we use for the induced orientation on the total space of a fibration. Indeed, the base and fiber are always even dimensional.
\end{rmk}
\begin{rmk}
The orientation induced on the boundary of $\oCMm_{0,k,l}$ by $o_{0,k,l}$ agrees with the orientation induced by the product structure up to a sign computed in Proposition~8.3.3 of~\cite{FO09}. However, we will not need to know this sign in the present work. Rather, we will use the compatibility property of the induced orientation on the boundary proved in Lemma~\ref{lm:Bor} below. This compatibility property arises naturally in the proof of the open topological recursion relations in Lemma~\ref{lm:brbs} below.
\end{rmk}
\begin{proof}[Proof of Lemma~\ref{lm:or}]
If the orientations $o_{0,k,l}$ exist, properties~\ref{it:zero}-\ref{it:bdry} imply they are unique. It remains to check existence.

For property~\ref{it:perm} to hold, we must show that permutations of labels that map the component $\oCMm_{0,k,l}$ to itself are orientation preserving. Indeed, let
\[
U  = \left\{(z,w)\left|
\begin{array}{lll}
z = (z_1,\ldots,z_k) \in (S^1)^k, & z_i \neq z_j, & i\neq j \\
w = (w_1,\ldots,w_l) \in (\ior D^2)^l, & w_i \neq w_j, & i \neq j
\end{array}
\right.\right\}.
\]
Denote by $U^{main} \subset U$ the subset where the cyclic order of $z_1,\ldots,z_k,$ on $S^1 = \partial D^2$ with respect to the orientation induced from the complex orientation of $D^2$ agrees with the standard order of $[k]$. Then
\[
\CMm_{0,k,l} = U^{main}/PSL_2(\R).
\]
Permutations of labels preserving the component $\oCMm_{0,k,l}$ are powers of the cyclic permutation $(z_1,\ldots,z_k) \to (z_k,z_1,\ldots,z_{k-1}).$ The cyclic permutation diffeomorphism $\CMm_{0,k,l} \to \CMm_{0,k,l}$ is induced by the diffeomorphism $U^{main} \to U^{main}$ obtained by restricting the diffeomorphism $\varpi: (S^1)^k \to (S^1)^k$ given by cyclically permuting the factors. Since $\dim S^1 = 1,$ the sign of $\varpi$ is $(-1)^{1 \cdot (k-1)}.$
Since $k$ is odd, we see that $\varpi$ preserves orientation. So, cyclic permutations of the boundary labels preserve the orientation of $U^{main}$ and thus also $\CM^{main}_{0,k,l}$ and $\oCM^{main}_{0,k,l}.$ Similarly, since $\dim D^2 = 2,$ arbitrary permutations of the interior labels preserve the orientation of $\oCM^{main}_{0,k,l}.$

A direct calculation shows that the orientation on $\oCM_{0,3,1}$ induced by property~\ref{it:int} from $o_{0,3,0}$ agrees with the orientation induced by property~\ref{it:bdry} from $o_{0,1,1}.$ See Figure~\ref{fig:fibers}. So $o_{0,3,1}$ exists. Existence of $o_{0,k,l}$ satisfying properties~\ref{it:int} and~\ref{it:bdry} for other $k,l,$ follows from the commutativity of the diagram of forgetful maps
\[
\xymatrix{
\oCMm_{0,k+2,l+1} \ar[r]\ar[d] & \oCMm_{0,k+2,l} \ar[d] \\
\oCMm_{0,k,l+1} \ar[r] & \oCMm_{0,k,l}.
}
\]
\end{proof}

For the remainder of the paper, we always consider the spaces $\oCM_{0,k,l}$ for $k$ odd equipped with the orientations $o_{0,k,l}.$

\subsection{Edge labels}
In constructing the boundary conditions for open descendent integrals, it is necessary to be able to refer to the edges of a vertex and their corresponding nodal points unambiguously. We do this as follows.
\begin{nn}
Let $\Gamma$ be a stable graph and let $e$ be an edge connecting vertices $u,v \in V(\Gamma).$ That is,
\[
e = \left\{u,v\right\}\in E\left(\Gamma\right).
\]
We denote by $\Gamma_e$ the (not necessarily stable) graph obtained from $\Gamma$ by removing $e$.
We denote by $\Gamma_{e,u}$ the pre-stable graph which is the connected component of $u$ in $\Gamma_e$. We define $\Gamma_{e,v}$ similarly.
\end{nn}
\begin{definition}\label{df:iv}
Denote by $i_v^\Gamma : I\left(v\right)\cup B\left(v\right)\to\LL$ the map defined by
\begin{enumerate}
\item
for $x\in\ell_I\left(v\right)\cup \ell_B(v),~i_v^\Gamma\left(x\right)=x,$
\item
for $e=\left\{u,v\right\}\in E, ~i_v^\Gamma\left(e\right) = \cup I\left(\Gamma_{e,u}\right)\cup\cup B\left(\Gamma_{e,u}\right),$
\end{enumerate}
When the graph $\Gamma$ is clear from the context, we write $i_v$ instead of $i_v^\Gamma.$
\end{definition}
\begin{rmk}\label{rmk:iv}
It is easy to see that $i_v$ is actually an injection.
Hence, we may identify $I\left(v\right) \cup B\left(v\right)$ with its image under $i_v.$ In addition, it follows from Definition~\ref{def:stg} part~\ref{it:annoyance}\ref{it:1ct} that
\[
i_v\left(I\left(v\right)\cup B\left(v\right)\right)\in 2^\LL_{fin,disj}.
\]
\end{rmk}

\begin{definition}\label{def:si}
Let $\Gamma \in \CG$ and $\Lambda \in \pu\Gamma.$ In light of Definitions~\ref{def:smoothing} and~\ref{def:bdry_maps}, we have canonical maps
\[
\varsigma = \varsigma_{\Lambda,\Gamma} : V(\Lambda) \longrightarrow V(\Gamma), \qquad \iota = \iota_{\Gamma,\Lambda} : E(\Gamma) \longrightarrow E(\Lambda).
\]
Namely, let $S \subset E(\Lambda)$ such that $\Gamma = d_S \Lambda.$ Then $V(\Gamma)$ is the quotient of $V(\Lambda)$ under the equivalence relation generated by $u \sim v$ if $\{u,v\} \in S.$ The map $\varsigma$ is the quotient map. Let
\[
\Sym^2 \varsigma : \Sym^2 V(\Lambda) \to \Sym^2 V(\Gamma)
\]
denote the induced map on unordered pairs. Then
\[
E(\Gamma) = \Sym^2 \varsigma(E(\Lambda)\setminus S).
\]
Since $\Lambda$ is a forest, it follows that $\Sym^2 \varsigma|_{E(\Lambda)\setminus S}$ is injective. Thus, we define $\iota = \left(\Sym^2\varsigma|_{E(\Lambda)\setminus S}\right)^{-1}.$

Alternatively, the maps $\varsigma$ and $\iota$ can be characterized in terms of labels as follows.
The map $\varsigma$ is uniquely determined by the condition that for $v \in V(\Lambda),$ there exists a partition $\im i^\Gamma_{\varsigma(v)} = P_1 \coprod \ldots \coprod P_n$ such that
\[
\im i^\Lambda_v = \{\cup P_i\}_{i = 1}^n.
\]
If $e = \{u,v\} \in E(\Gamma),$ then $\iota(e)$ is the unique edge of the form $\{\tilde u, \tilde v\}$ where $\varsigma(\tilde u) = u$ and $\varsigma(\tilde v) = v.$
\end{definition}

\begin{figure}[t]
\centering
\includegraphics[scale=.13]{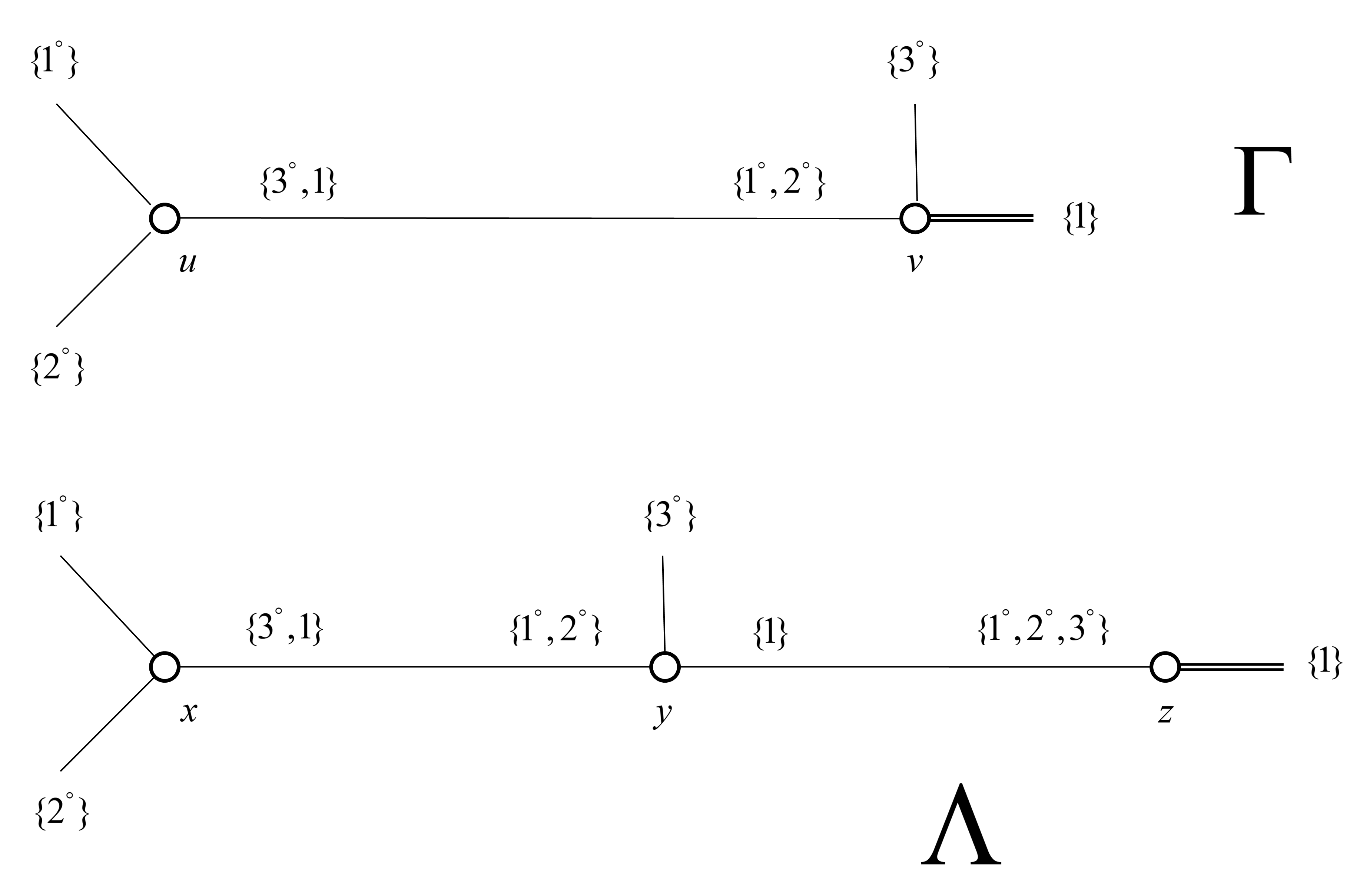}
\caption{We show the stable graphs from Example~\ref{ex:edgelabels}. Boundary labels are shown as half-edges and interior labels are shown as double half-edges.}
\label{fig:edgelabels}
\end{figure}
\begin{ex}\label{ex:edgelabels}
To see how the edge labels of Definition~\ref{df:iv} and the maps $\varsigma$ and $\iota$ of Definition~\ref{def:si} work, it is helpful to consider concrete examples. We refer the reader to Figure~\ref{fig:edgelabels}. Let $\Gamma \in \partial \Gamma_{0,3,1}$ be given by
\begin{gather*}
V(\Gamma) = V^o(\Gamma) = \{u,v\}, \qquad E(\Gamma) = \{\{u,v\}\}, \\
\ell_B(u) = \{\{\rz{1}\},\{\rz{2}\}\},
 \qquad \ell_B(v) = \{\{\rz{3}\}\},  \\
 \qquad \ell_I(u) = \emptyset, \qquad \ell_I(v) = \{\{1\}\}.
\end{gather*}
Then,
\[
i^\Gamma_u(\{u,v\}) = \{\rz{3},1\}, \qquad i^\Gamma_v(\{u,v\}) = \{\rz{1},\rz{2}\}.
\]
Let $\Lambda \in \partial \Gamma$ be given by
\begin{gather*}
V(\Lambda) = V^o(\Lambda) =  \{x,y,z\}, \qquad E(\Lambda) = \{\{x,y\},\{y,z\}\}, \\
 \ell_B(x) = \{\{\rz{1}\},\{\rz{2}\}\},
\qquad \ell_B(y) = \{\{\rz{3}\}\}, \qquad \ell_B(z) = \emptyset, \\
\ell_I(x) = \emptyset, \qquad \ell_I(y) = \emptyset, \qquad \ell_I(z) = \{\{1\}\}.
\end{gather*}
Then,
\begin{gather*}
i^\Lambda_x(\{x,y\}) = \{\rz{3},1\}, \qquad i^\Lambda_y(\{x,y\}) = \{\rz{1},\rz{2}\}, \\
i^\Lambda_y(\{y,z\}) = \{1\}, \qquad i^\Lambda_z(\{y,z\}) = \{\rz{1},\rz{2},\rz{3}\}.
\end{gather*}
The map $\varsigma = \varsigma_{\Lambda,\Gamma} : V(\Lambda) \longrightarrow V(\Gamma)$ is given by
\begin{equation*}
\varsigma(x) = u, \qquad \varsigma(y) = v, \qquad \varsigma(z) = v.
\end{equation*}
The map $\iota = \iota_{\Gamma,\Lambda} : E(\Gamma) \longrightarrow E(\Lambda)$ is given by
\[
\iota(\{u,v\}) = \{x,y\}.
\]
To verify the characterization of $\varsigma$ in terms of labels from Definition~\ref{def:si}, we observe that
\begin{gather*}
\im^\Gamma_u =\im i^\Lambda_x = \{\{\rz{1}\},\{\rz{2}\},\{\rz{3},1\}\} \\
\im i^\Gamma_v = \im i^\Lambda_y = \{\{\rz{1},\rz{2}\}, \{\rz{3}\}, \{1\} \}, \\
\im i^\Lambda_z = \{ \{\rz{1},\rz{2},\rz{3}\},\{1\}\}.
\end{gather*}
So, to verify the characterization for $\varsigma(x) = u,$ and $\varsigma(y) = v,$ the required partitions of $\im^\Gamma_u$ and $\im i^\Gamma_v$ are trivial. To verify it for $\varsigma(z) = v,$ consider the partition
\[
\im i^\Gamma_v = \{\{\rz{1},\rz{2}\}, \{\rz{3}\} \} \coprod \{ \{1\} \}.
\]

\end{ex}

\begin{definition}\label{def:span}
For $\Gamma \in \CG$ and $U \subset V(\Gamma),$ let $\Gamma_U$ be the stable graph spanned by $U$ with labels added in place of edges connecting $U$ to its complement. Specifically,
\begin{gather*}
V(\Gamma_U) = U, \qquad \varepsilon_{\Gamma_U} = \varepsilon_\Gamma|_U, \qquad
E(\Gamma_U) = \{\{u,v\} \in E(\Gamma) | u,v \in U\}, \\
\forall v \in U, \qquad \ell_I^{\Gamma_U}(v) = \ell_I^\Gamma(v) \cup \left\{\left.i_v^\Gamma(e)\right|e = \{u,v\}\in E^I(\Gamma), u \notin U\right\}, \\
\ell_B(v) = \ell_B^\Gamma(v) \cup \left\{\left.i_v^\Gamma(e)\right|e = \{u,v\}\in E^B(\Gamma), u \notin U\right\}.
\end{gather*}
For $v \in V(\Gamma),$ abbreviate $\Gamma_v = \Gamma_{\{v\}}$ and
\[
\CM_v = \CM_{\Gamma_v}.
\]
\end{definition}

\subsection{Forgetful maps}
We now define forgetful maps for stable graphs.
Let $\Gamma$ be a connected pre-stable graph. Set
\[
k = k\left(\Gamma\right),\quad l=l\left(\Gamma\right),\quad I = I\left(\Gamma\right),\quad B=B\left(\Gamma\right).
\]
In case $V^O=\emptyset,$ assume $l\geq 3.$ In case $V^O\neq\emptyset,$ assume $k+2l\geq 3.$
Define the graph $stab\left(\Gamma\right)$ as follows. Take any unstable vertex $v\in V^O\cup V^C.$
\begin{enumerate}
\item
In case $v\in V^O$ ($V^C$) has no boundary or interior labels and exactly $2$ boundary (interior) edges
\[
e_1 = \left\{v,u\right\},\quad e_2 = \left\{v,w\right\},
\]
remove $v$ and its edges from the graph and add the new boundary (interior) edge $\left\{u,w\right\}.$
\item
In case $v\in V^O$ has a single boundary edge $\left\{v,u\right\}$, a single boundary label $i$ and no interior edges or labels, remove $v$ and its edge from the graph and add $i$ to $\ell_B\left(u\right).$
\item
In case $v\in V^C$ has a single interior edge $\left\{v,u\right\}$ and a single interior label $i$, remove $v$ and its edge from the graph and add $i$ to $\ell_I\left(u\right).$
\item
In case $v$ has a single edge, and no labels, remove $v$ and its edge from the graph.
\end{enumerate}
Other cases are not possible.
We iterate this procedure until we get a stable graph. Note that the process does stop, and that the final result does not depend on the order of the above steps. We extend the definition of $stab$ to not necessarily connected graphs by applying it to each component if each component satisfies the assumptions.
\begin{definition}
The graph $stab\left(\Gamma\right)$ is called the \emph{stabilization} of $\Gamma.$
\end{definition}

\begin{nn}
Consider a stable graph $\Gamma$ such that
\[
k\left(\Gamma\right)+2\left(l\left(\Gamma\right)-1\right)\geq 3.
\]
If $i \notin I(\Gamma),$ we define $for_i(\Gamma) = \Gamma.$
If $i\in I(\Gamma),$ we define $for_i\left(\Gamma\right)$ to be the graph obtained by removing the label $i$ from the vertex $v_i$ and stabilizing.
\end{nn}
\begin{obs}\label{obs:very_trivial}
Let $\Gamma\in \CG.$ The natural map
\[
\prod_{v\in V\left(\Gamma\right)}\CM_v\to\CM_\Gamma,
\]
is an isomorphism of smooth manifolds with corners.
\end{obs}
We shall use the preceding observation to identify the two moduli spaces throughout the article.

\begin{nn}
Let $\Gamma,\Gamma',$ be stable graphs. Let
\[
f : V\left(\Gamma'\right) \longrightarrow V\left(\Gamma\right),
\]
and
\[
f_v : \im \left(i_{v}^{\Gamma'}\right)\longrightarrow \im\left(i_{f(v)}^\Gamma\right), \qquad v \in V(\Gamma'),
\]
be injective mappings such that
\begin{equation}
S \subset f_v(S), \qquad \forall v \in V(\Gamma'), \quad S \in \im \left(i_{v}^{\Gamma'}\right).
\end{equation}
Given $f,$ if the maps $f_v$ exist, they are unique by Remark~\ref{rmk:iv}.
We say that $(\Gamma',f)$ is a \emph{stable subgraph} of $\Gamma.$ When the map $f$ is clear from context, it is omitted from the notation. For example, $for_i(\Gamma)$ is a stable subgraph of $\Gamma.$ Indeed, $V(for_i(\Gamma))$ is a subset of $V(\Gamma),$ and one verifies that the maps $f_v$ exist.

The map $f_v$ induces a forgetful map
\[
For_{v}:\CM_{f(v)}\to\CM_{v}, \qquad v \in V(\Gamma').
\]
Denote by
\[
\pi_{\Gamma,\Gamma'} : \prod_{v \in V(\Gamma)} \CM_v \to \prod_{v \in V(\Gamma')} \CM_{f(v)}
\]
the projection.
We define the forgetful map
\[
For_{\Gamma,\Gamma'}:\CM_\Gamma\to\CM_{\Gamma'}
\]
by
\[
For_{\Gamma,\Gamma'} = \left(\prod_{v \in V(\Gamma')} For_v\right) \circ \pi_{\Gamma,\Gamma'}.
\]
We abbreviate
\[
For_i = For_{\Gamma,for_i(\Gamma)} : \CM_{\Gamma} \longrightarrow \CM_{for_i(\Gamma)}.
\]
\end{nn}

\begin{obs}\label{obs:fuf}
If $\Gamma''$ is a stable subgraph of $\Gamma'$ and $\Gamma'$ is a stable subgraph of $\Gamma,$ then
\[
F_{\Gamma,\Gamma''} = F_{\Gamma',\Gamma''} \circ F_{\Gamma,\Gamma'}.
\]
\end{obs}

\section{Line bundles and relative Euler classes}\label{sec:bdry}
\subsection{Cotangent lines and canonical boundary conditions}\label{subs:defs_bundles}
For $i \in I,$ denote by
\[
\CL_i\to\oCM_{0,B,I}
\]
the $i^{th}$ tautological line bundle. The fiber of $\CL_i$ over a stable disk $\Sigma$ is the cotangent line at the $i^{th}$ marked point $T_{z_i}\Sigma.$ For any stable graph $\Gamma$ with $i\in I\left(\Gamma\right)$, define
\[
\CL_i\to\CM_\Gamma
\]
using the canonical identification of Observation \ref{obs:very_trivial}. This definition of $\CL_i \to \CM_\Gamma$ agrees with restriction of $\CL_i \to \oCM_{0,k,l}$ to $\CM_\Gamma\subset \oCM_{0,k,l}$ for~$\Gamma \in \partial \Gamma_{0,k,l}.$

Let
\[
E = \bigoplus_{i\in [l]}\CL_i^{\oplus a_i}\rightarrow \oCM_{0,k,l},
\]
where $a_i,k,l,$ are non-negative integers such that
\begin{gather}\label{eq:cdim}
\rk_\C E = \sum_{i \in [l]} a_i =  \dim_\C\oCM_{0,k,l} = \frac{2l+k-3}{2}, \\
k + 2l - 2 > 0. \notag
\end{gather}
In particular, since $\dim_\C\oCM_{0,k,l}$ is an integer, $k$ must be odd.
We shall begin by defining the vector space $\CS$ of canonical boundary conditions for $E.$ It is a vector subspace of the vector space of multisections of~$E|_{\partial\oCM_{0,k,l}}.$ See Appendix~\ref{app:euler} for background on multisections.

Consider a stable graph $\Gamma \in \partial \Gamma_{0,k,l}$ corresponding to a codimension one corner of $\oCM_{0,k,l}.$ Thus,
\[
\left|E(\Gamma)\right| = \left|E^B(\Gamma)\right| = 1.
\]
Write $V(\Gamma)=\left\{v_1, v_2\right\}$.
Exactly one of $k(v_1), k(v_2),$ is even. Without loss of generality, it is $k(v_2)$. Let $\Gamma'$ be the stable graph with no edges and two open vertices $v'_1,v'_2,$ with
\begin{gather*}
\ell_B(v_1') = i_{v_1}(B(v_1)), \qquad \ell_B(v_2') = \ell_B(v_2)\\
\ell_I(v_1') = \ell_I(v_1'), \qquad \ell_I(v_2') = \ell_I(v_2).
\end{gather*}
Here, $\Gamma'$ is stable because of the assumption on the parity of $k(v_2).$
The definition of $\Gamma'$ implies that $\CM_{\Gamma'}$ is the same as $\CM_\Gamma$ except that the marked point corresponding to the edge of $\Gamma$ on the component of $v_2$ has been forgotten.
Let $E'$ be the vector bundle given by
\[
E' = \bigoplus_{i\in [l]}\CL_i^{\oplus a_i}\rightarrow \CM_{\Gamma'}.
\]
Since the map $For_{\Gamma,\Gamma'}$ does not contract any components of the stable disks in $\CM_\Gamma,$ we have
\[
\left.E\right|_{\CM_\Gamma} \simeq For_{\Gamma,\Gamma'}^* E'.
\]
See Observation~\ref{obs:identification_bundle_over_base} and the preceding discussion for details.

Recall that the boundary $\partial X$ of a manifold with corners $X$ is itself a manifold with corners, equipped with a map
\[
i_X : \partial X\to X,
\]
which may not be injective. A section $s$ of a bundle $F\to \partial X$ is \emph{consistent} if
\[
\forall~p_1,p_2\in X,\text{~such that~}i_X\left(p_1\right) = i_X\left(p_2\right)\text{~we have~} s\left(p_1\right) = s\left(p_2\right).
\]
Consistency for multisections is similar. For a vector bundle $F \to X,$ we write $F |_{\partial X} = i_X^*F$ and similarly for sections of $F.$

\begin{definition}
A smooth consistent multisection $s$ of $\left. E\right|_{\partial \oCM_{0,k,l}}$ is called a \emph{canonical multisection} if for each graph $\Gamma \in \pB\Gamma_{0,k,l}$ with a single edge,
\[
\left.s\right|_{\CM_\Gamma} = For_{\Gamma,\Gamma'}^*s',
\]
where $s'$ is a multisection of $E'\rightarrow\CM_{\Gamma'}$. The vector space of all canonical multisections is denoted by $\CS$.
\end{definition}

\subsection{Definition of open descendent integrals}
\begin{nn}
Given a complex vector bundle $F\to X,$ where $X$ is a manifold with corners,
denote by $C_m^\infty\left(F\right)$ the space of smooth multisections.
Given a nowhere vanishing smooth consistent multisection
\[
\mathbf{s}\in C_m^\infty\left(F|_{\partial X}\right),
\]
denote by
\[
 e\left(F ; \mathbf{s} \right) \in H^*\left(X,\partial X\right)
\]
the \emph{relative Euler class}. This is by definition the Poincar\'e dual of the vanishing set of a transverse extension of $\mathbf{s}$ to $X.$ See Appendix \ref{app:euler} for details.
\end{nn}

\begin{thm}\label{thm:intersection_numbers_well_defined}
When condition~\eqref{eq:cdim} holds, one can find a nowhere vanishing multisection $\mathbf{s}\in\CS$. Hence one can define
$e\left(E ; \mathbf{s}\right)$. Moreover, any two nowhere vanishing multisections of $\CS$ define the same relative Euler class.
\end{thm}
\begin{definition}
When condition~\eqref{eq:cdim} holds, define $e\left(E ; \CS\right)$ to be the relative Euler class $e(E,\s)$ for any $\s \in \CS.$ This notation is unambiguous by the preceding theorem. The genus zero \emph{open descendent integrals} are defined by
\begin{equation}\label{eq:maindef}
\left\langle \tau_{a_1}\ldots\tau_{a_l}\sigma^k \right\rangle _{0}^o = 2^{-\frac{k-1}{2}}\int_{\overline{\CM}_{0,k,l}}e\left(E,\CS\right)
\end{equation}
when condition~\eqref{eq:cdim} holds. Otherwise, they are defined to be zero.
\end{definition}
The division by the power of $2$ in the preceding definition is only for convenience. When $r_0$ of the $a_i$ are equal to $0$, $r_1$ of them equal to $1$, and so on, we sometimes use the notation $\left\langle \tau_0^{r_0}\tau_1^{r_1}\ldots\sigma^k\right\rangle_0^o$ for the above quantity by analogy with the widely used notation for the closed correlators. When clear from the context, we may drop the genus subscript or the $o$ superscript, and write $\left\langle \tau_0^{r_0}\tau_1^{r_1}\ldots\sigma^k\right\rangle^o$ or
$\left\langle \tau_0^{r_0}\tau_1^{r_1}\ldots\sigma^k\right\rangle.$
\begin{rmk}\label{rm:msec}
A surprising feature of our construction is the use multisections rather than sections. The reason for this is that in general one cannot find a non-vanishing section in $\CS.$ This fact will be transparent later when we calculate intersection numbers. We shall see that often the intersection numbers will not be a multiple of the number of components of $\oCM_{0,k,l}.$ However, each component contributes equally to the intersection number, so each component must contribute a non-integer to the intersection number.

But we want also to understand geometrically what happens. Consider the case of $\oCM_{0,5,1}$ and $E = \CL_1^{\oplus 2}.$ For simplicity we illustrate a section of $\CL_1$ as a tangent vector at the interior marked point. Consider Figure \ref{fig:why_multi}. We may take the interior marked point to be the center of the disk, as a result of the $PSL_2\left(\R\right)$ equivalence relation.
Let $\Gamma\in\partial\Gamma_{0,5,1}$ be the unique stable graph with two vertices, $v_1,v_2,$ both open, and
\[
\ell_I(v_2) = \emptyset, \qquad \ell_B(v_2) = \left\{4,5\right\}.
\]
Consider a non-vanishing section $s$ of $\CL_1,$ which is a component of a canonical section of $E|_{\partial \oCM_{0,5,1}}.$

\begin{figure}[t]
\centering
\includegraphics[scale=.7]{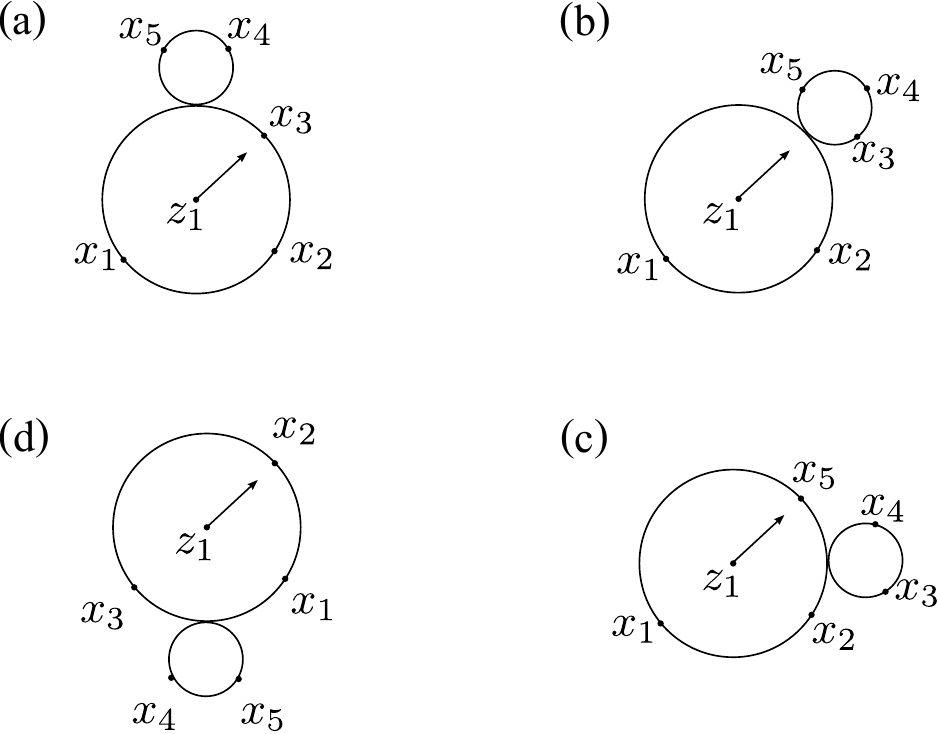}
\caption{A canonical multisection at different boundary points.}
\label{fig:why_multi}
\end{figure}

In item $\left(a\right)$ of the figure, we depict a stable disk $\Sigma \in \CM_\Gamma,$ at which $s$ points to the boundary marked point $x_3.$ Note that pointing at $x_3$ is preserved by the action of $PSL_2\left(\R\right).$ When the bubble of $x_4,x_5,$ approaches $x_3,$ while nothing else changes in the component of $z_1,x_1,x_2,x_3$, the section keeps on pointing towards $x_3$ by the definition of canonical boundary conditions. After the bubble reaches $x_3,$ we move to item $\left(b\right)$ of the figure. By continuity, the section still points in the direction of $x_3,$ only that now there is a boundary node there. Continuous changes in the component of $x_3,x_4,x_5,$ do not affect $s,$ again by the definition of canonical boundary conditions. In particular, $s$ does not change when $x_5$ approaches the node. After $x_5$ reaches the node we pass to item $\left(c\right).$ Again continuity guarantees no change in $s.$ Continuing in this manner, we finally reach item $\left(d\right).$ When we finish, $s$ points at $x_{2}.$

Now, let $\Sigma$ be the unique marked disk in $\CM_{\Gamma}$ such that if we take the interior marked point $z_1$ to be the center of the disk as before, we have the angle condition
\[
\measuredangle x_1 z_1 x_2 = \measuredangle x_2 z_1 x_3 = \measuredangle x_3 z_1 x_1 = \frac{2\pi}{3}.
\]
If a canonical multisection of $E|_{\partial\oCM_{0,5,1}}$ does not vanish at $\Sigma,$ then without loss of generality we may assume that its first component, $s,$ does not vanish there. Moreover, after possibly multiplying by a complex scalar, we may assume $s$ points at $x_3.$ So, we are in item $\left(a\right).$ Using the above reasoning we see that on the surface $\Sigma'$ of item $\left(d\right),$ the section $s$ must point at $x_{2}.$ As a consequence of the choice of $\Sigma,$
\[
\Sigma' \simeq \Sigma,
\]
which is a contradiction. Of course, this example generalizes beyond $\oCM_{0,5,1}$ and establishes the need for multisections.
\end{rmk}

\subsection{The base}
In order to prove Theorem \ref{thm:intersection_numbers_well_defined} we need to understand how canonical multisections behave on boundary strata of arbitrary codimension. We encode the relevant combinatorics in an operation on graphs called the base.

\begin{definition}
Let $\Gamma \in \CG.$ A boundary edge $e = \left\{u,v\right\} \in E^B(\Gamma)$ is said to be \emph{illegal} for the vertex $v$ if $k(\Gamma_{e,v})$ is odd. Otherwise it is \emph{legal}. Denote by $E_{legal}\left(v\right)$ the set of legal edges of $v$.
Recall that a boundary node in a stable curve $\Sigma\in\CM_{\Gamma}$ corresponds to a boundary edge of $\Gamma$, and a component of $\Sigma$ corresponds to a vertex of $\Gamma.$ We define a boundary node of $\Sigma$ to be \emph{legal} for a component $\Sigma_\alpha$ if the corresponding edge is legal for the corresponding vertex. Otherwise it is \emph{illegal}.
\end{definition}

\begin{nn}
Denote by $\CG_{odd}$ the set of all $\Gamma \in \CG$ such that for every connected component $\Gamma_i$ of $\Gamma$, either $V^O\left(\Gamma_i\right)=\emptyset$ or $k\left(\Gamma_i\right)$ is odd.
\end{nn}

Simple parity considerations show the following.
\begin{obs}
If $\Gamma\in\CG_{odd}$ and $e=\left\{u,v\right\} \in E^B(\Gamma),$ then $e$ is legal for exactly one of $u,v$.
\end{obs}

\begin{obs}\label{obs:parity}
Let $\Gamma\in\CG_{odd}$ and $v \in V^O(\Gamma).$ Then the total number of legal edges and boundary labels of $v$ is an odd number.
Moreover, in case $\ell_I\left(v\right)\neq\emptyset$, even if we erase from $\Gamma$ the edges which are illegal for $v$, the vertex $v$ remains stable.
\end{obs}
\begin{proof}
Let $e = \left\{u,v\right\}$ be a boundary edge of $v$.
If $e$ is legal for $v$, then it is illegal for $u$, so $k(\Gamma_{e,u})$ is odd. Otherwise, $k(\Gamma_{e,u})$ is even. Thus
\begin{align*}
\left|\ell_B(v) \cup E_{legal}(v)\right| &\cong |\ell_B(v)| + \sum_{e = \{u,v\} \in E_{legal}(v)} k(\Gamma_{e,u}) \\
& \cong k(\Gamma) \cong 1 \pmod 2,
\end{align*}
which is the first claim of the lemma.

Regarding stability, we have just seen that $\left| \ell_B(v)\cup E_{legal}(v) \right|$ is odd, hence at least $1$, and by assumption $\ell_I(v)\neq\emptyset$. Stability follows.
\end{proof}

\begin{definition}\label{def:base_operator}
The \emph{base} is an operation on graphs
\[
\CB:\CG_{odd}\to\CG_{odd}
\]
defined as follows. For $\Gamma \in \CG_{odd}$ the graph $\CB \Gamma$ is given by
\begin{gather*}
V(\CB \Gamma) = \left\{ v \in \Gamma \left| 2l(v) + \left| \ell_B(v) \cup E_{legal}(v)\right| \geq 3 \right.\right\}, \\
\varepsilon_{\CB \Gamma} = \varepsilon_\Gamma|_{V(\CB \Gamma)}, \qquad \qquad E(\CB \Gamma) = \emptyset, \\
\ell_I^{\CB\Gamma}(v) = i_v^\Gamma\left(I^\Gamma(v)\right), \quad \ell_B^{\CB\Gamma}(v)= i_v^\Gamma\left(\ell^\Gamma_B(v) \cup E_{legal}^\Gamma(v)\right), \quad v \in V(\CB\Gamma).
\end{gather*}
We abbreviate
\[
F_\Gamma = For_{\Gamma,\CB\Gamma} : \CM_\Gamma \to \CM_{\CB\Gamma}.
\]
\end{definition}

\begin{obs}\label{obs:extension_of_canonical_conds}
A multisection $s$ of
\[
E = \bigoplus_{i\in\left[l\right]}\CL_i^{\oplus a_i}\rightarrow\partial\oCM_{0,k,l}
\]
is canonical if and only if for each $\Gamma \in \pB \Gamma_{0,k,l}$, there exists a multisection $s^{\CB\Gamma}$ of
\[
\bigoplus_{i\in\left[l\right]}\CL_i^{\oplus a_i}\rightarrow\CM_{\CB\Gamma}
\]
such that $s|_{\CM_\Gamma} = F_{\Gamma}^*s^{\CB\Gamma}.$
\end{obs}
\begin{proof}
The case where $\Gamma$ has a single edge is exactly the definition. The general case follows from the continuity of $s.$
\end{proof}

\begin{obs}\label{obs:Bmaps}
Observation~\ref{obs:parity} implies that
$I(\Gamma) \subseteq I(\CB\Gamma).$
It follows from the definition of $\CB$ that there is a canonical inclusion $$\iota_\CB^V : V(\CB \Gamma) \hookrightarrow V(\Gamma).$$
\end{obs}

The following observation is straightforward.
\begin{obs}\label{obs:illegals_increase}
Recall Definition~\ref{def:si}. Let $e = \left\{u,v\right\}\in E\left(\Gamma\right)$ and let $\Lambda\in\partial\Gamma.$  Let $\tilde e = \iota_{\Gamma,\Lambda}(e)$ and let $\tilde u,\tilde v \in V(\Lambda)$ be such that $\varsigma_{\Lambda,\Gamma}(\tilde u) = u,\,\varsigma_{\Lambda,\Gamma}(\tilde v) =v$ and $\tilde e = \{\tilde u,\tilde v\}.$
Then $\tilde e$ is illegal for $\tilde v$ if and only if $e$ is illegal for $v.$
\end{obs}

The following is a consequence of the preceding observation.
\begin{obs}\label{obs:BpB}
We have
\[
\CB \circ \pu = \CB \circ \pu \circ \CB.
\]
\end{obs}

The key to constructing homotopies between canonical multisections is the following.
\begin{obs}\label{obs:dim_of_base_mod}
For $\Gamma \in \partial \Gamma_{0,k,l}$ with $k$ odd, we have
\[
\dim_\C \CM_{\CB\Gamma} \leq \dim_\C \CM_{0,k,l}-1.
\]
In addition, for any $v \in V(\CB\Gamma),$ we have $\dim_\C \CM_{v} \in \Z.$ It follows that $\dim_\C \CM_{\CB\Gamma}\in\Z.$
\end{obs}
\begin{proof}
If $\Gamma$ has at least one interior edge or two boundary edges, then $\dim_\C \CM_\Gamma \leq \dim_\C \CM_{0,k,l}-1.$ So, since $\dim_\C\CM_{\CB\Gamma} \leq \dim_\C \CM_{\Gamma},$ the desired inequality follows.
It remains to consider the case that $\Gamma$ consists of two vertices $u,v,$ connected by a single boundary edge $e$. Then $e$ is illegal for exactly one of the vertices, say $v.$ The stability of $v$ and the illegality of $e$ for $v$ imply $k(v) \geq 4$. So, dropping $e$ in passing to $\CB\Gamma$ does not destabilize $v.$ Thus there is a corresponding vertex $v'$ in $\CB\Gamma$ with $k(v') = k(v)-1.$ It follows that
\begin{equation*}
\dim_\C \CM_{\CB\Gamma} \leq \dim_\C \CM_{\Gamma}-\frac{1}{2} = \dim_\C \CM_{0,k,l} - 1.
\end{equation*}
This completes the proof of the first claim.
The integrality follows immediately from Observation~\ref{obs:parity}.
\end{proof}
Recall Lemma~\ref{lm:or}. Let $k$ be odd, and let $\Gamma \in \pB \Gamma_{0,k,l}$ consist of two open vertices $v^\pm_\Gamma,$ connected by a single boundary edge that is legal for $v^+_\Gamma$ and illegal for $v^-_\Gamma.$ In particular, $\CM_\Gamma$ is an open subset of $\partial \oCM_{0,k,l}$. Denote by $o_\Gamma$ the orientation on $\oCM_\Gamma$ induced by $o_{0,k,l}$ and the outward normal vector, ordering the outward normal first. Furthermore, writing $k_\Gamma^+ = k(v^+_\Gamma), k_\Gamma^- = k(v^-_\Gamma)-1$ and $l_\Gamma^\pm = l(v^\pm_\Gamma),$ we have
\begin{equation}\label{eq:Bprod}
\CM_{\CB\Gamma} \simeq \CM_{0,k_\Gamma^+,l_\Gamma^+} \times \CM_{0,k_\Gamma^-,l_\Gamma^-}
\end{equation}
where $k^\pm_\Gamma$ are both odd. So we define the orientation $o_{\CB\Gamma}$ of $\CM_{\CB\Gamma}$ to be the product of the orientations $o_{0,k^\pm_\Gamma,l^\pm_\Gamma}$. The choice of isomorphism~\eqref{eq:Bprod} does not affect $o_\Gamma$ because of property~\ref{it:perm} of $o_{0,k,l}.$ The fiber of the map $F_\Gamma$ is a collection of open intervals in the boundary of a disk, and thus carries an induced orientation, which we call \emph{natural} below.
\begin{lm}\label{lm:Bor}
The orientation $o_\Gamma$ agrees with the orientation induced from $o_{\CB \Gamma}$ by the fibration $F_\Gamma:\CM_\Gamma \to \CM_{\CB\Gamma}$ and the natural orientation on the fiber.
\end{lm}
\begin{proof}
The claim can be checked explicitly in the three cases when $\dim_\R \oCM_{0,k,l} = 2.$ We use induction on $\dim_\R \oCM_{0,k,l}$ to reduce to the two dimensional case.
Indeed, assume $\dim_\R \oCM_{0,k,l} \geq 4.$ Since $k_\Gamma^\pm$ are odd and
\[
4 \leq \dim_\R \oCM_{0,k,l} = k_\Gamma^+ + k_\Gamma^- + 2(l_\Gamma^+ + l_\Gamma^-)-4,
\]
either $k_\Gamma^+ +2 l_\Gamma^+ \geq 5$ or $k_\Gamma^- + 2l_\Gamma^- \geq 5.$ By property~\ref{it:perm} of $o_{0,k,l}$, in case $k_\Gamma^+ + 2l_\Gamma^+ \geq 5,$ we may assume $\ell_B(v^+_\Gamma) = \{k^-_\Gamma+1,\ldots,k\}.$ Otherwise, we may assume $\ell_B(v^-_\Gamma) = \{k^+_\Gamma,\ldots,k\}.$  Again by property~\ref{it:perm}, it suffices to prove the claim for $o_\Gamma$ restricted to $\CMm_\Gamma : = \CM_\Gamma \cap \CMm_{0,k,l}.$ Denote by $\CMm_{\CB\Gamma}$ the corresponding component of $\CM_{\CB\Gamma}.$ We choose isomorphism~\eqref{eq:Bprod} so it induces an isomorphism
\[
\CMm_{\CB\Gamma} \simeq \CMm_{0,k_\Gamma^+,l_\Gamma^+} \times \CMm_{0,k_\Gamma^-,l_\Gamma^-}
\]
such that the $k^{th}$ marked point of $\CM_{\CB\Gamma}$ corresponds to either the $k^+_\Gamma$ boundary marked point of $\CM_{0,k_\Gamma^+,l_\Gamma^+}$ or the $k_\Gamma^-$ boundary marked point of $\CM_{0,k_\Gamma^-,l_\Gamma^-}.$
If $k_\Gamma^+ +2 l_\Gamma^+ \geq 5$ and $k_\Gamma^+ \geq 3,$ or if $k_\Gamma^- +2 l_\Gamma^- \geq 5$ and $k_\Gamma^- \geq 3$, let $(k',l') = (k-2,l).$ Otherwise, let $(k',l') = (k,l-1).$
If $k' = k-2,$ let $\Gamma'\in \pB\Gamma_{0,k',l'}$ be the graph obtained from $\Gamma$ by forgetting the boundary markings $k,k-1.$ Otherwise, let $\Gamma'= for_l(\Gamma).$
Consider the following commutative diagram.
\[
\xymatrix@!=.4cm{
&&&& D \ar[ddd]\ar@{_{(}->}[dl]\\
&&& D' \ar[ddd] \\
&&A\ar[dl]&&& B \ar[ld] \ar@{_{(}->}'[l]'[ll][lll]\\
&  \oCMm_{0,k,l} \ar[ld]^(.4)a &&& \CMm_\Gamma \ar@{_{(}->}'[l][lll]\ar[ld]^(.4)b\ar[ddd]_{F_\Gamma} \\
 \oCMm_{0,k',l'} &&& \CMm_{\Gamma'} \ar[ddd]_{F_{\Gamma'}} \ar@{_{(}->}[lll] \\
&&& &&&& C\ar@/_/[llld]\ar@/^1pc/[d] \ar@{_{(}->}[uuull]\\
&&&&\CMm_{\CB\Gamma}\ar[ld]^(.4)c &&& \ar@{-}[lll]_(.65){\sim} \CMm_{0,k_\Gamma^+,l_\Gamma^+} \times \CMm_{0,k_\Gamma^-,l_\Gamma^-} \ar[dl]^(.35){c'}\\
&&& \CMm_{\CB\Gamma'} &&& \CMm_{0,k_{\Gamma'}^+,l_{\Gamma'}^+} \times \CMm_{0,k_{\Gamma'}^-,l_{\Gamma'}^-}\ar@{-}[lll]_(.65){\sim}
}
\]
The spaces $A,B$ and $C,$ are the fibers of the forgetful maps $a,b,$ and~$c,$ respectively. So, if $k' = k-2$, then $A\simeq[0,1]^2.$ Otherwise, $A~\simeq~D^2.$ The fibers $B$ and $C$ are open subsets of $A$ and the inclusions preserve the natural or complex orientations. The spaces $D,D',$ are the fibers of the maps $F_{\Gamma},F_{\Gamma'},$ respectively. Both $D$ and $D'$ are homeomorphic to an open interval, and the open inclusion $D \hookrightarrow D'$ preserves the natural orientations. By our assumption on $\ell_B(v^\pm_\Gamma),$ the map $c'$ is the identity on one of the factors $\CMm_{0,k_\Gamma^\pm,l_\Gamma^\pm}$ and the forgetful map on the other.

We say a fibration is oriented if the orientation on the total space is induced by that on the base and fiber. Thus $a$ and $c'$ are oriented by properties~\ref{it:int} and~\ref{it:bdry} of the orientations $o_{0,k,l}.$ By the definition of the orientations $o_\Gamma$ and $o_{\CB\Gamma},$ it follows that $b$ and $c$ are oriented. By induction $F_{\Gamma'}$ is oriented. So the diagram implies $F_\Gamma$ is oriented as well.
\end{proof}

\subsection{Abstract vertices}
For proving theorems, a refinement of canonical multisections is helpful. The relevant definition, given in Section~\ref{ssec:scbc}, uses the notion of an abstract vertex.
\begin{definition}
An \emph{abstract vertex} $v$ is a triple $\left(\varepsilon, k, I\right)$, where
\begin{enumerate}
\item
$\varepsilon\in \left\{C , O\right\}$.
\item
$k\in\mathbb{Z}_{\geq 0}$.
\item
$I = \left\{i_1 , i_2 ,\ldots, i_l\right\}\in 2^\LL_{fin,disj}$.
\end{enumerate}
We demand that if $\varepsilon = C$, then $k=0$.
We call $k = k\left(v\right)$ the number of boundary labels of the abstract vertex and $l = l\left(v\right) = \left|I\right|$ the number of interior labels. We also use the notation $I\left(v\right)$ for $I$, and we call the elements of $I\left(v\right)$ the interior labels of $v$. An abstract vertex is said to be \emph{open} if $\varepsilon = O$, and otherwise it is \emph{closed}.
An abstract vertex is called \emph{stable} if $k+2l \geq 3$.

Denote by $\mathcal{V}$ the set of all stable abstract vertices.
\end{definition}
\begin{nn}
Let $v\in\mathcal{V}$. We define
\[
\CM_v =
\begin{cases}
\CM_{0,I\left(v\right)}, & \varepsilon\left(v\right) = C, \\
\CM_{0,k\left(v\right),I\left(v\right)}, & \varepsilon\left(v\right) = O.
\end{cases}
\]
We define $\oCM_v$ similarly.
\end{nn}
We turn to the boundary of an abstract vertex, soon to be related with the boundary of a stable graph.
\begin{definition}
Given an abstract vertex $v= \left(\varepsilon, k, I\right),$ we define the \emph{boundary} of $v$, denoted by $\partial v$, as the collection of abstract vertices $v' = \left(\varepsilon', k', I'\right)\neq v$ which satisfy
\begin{enumerate}
\item
If $\varepsilon = C$, then $\varepsilon' = C$.
\item
$k'\leq k.$
\item
Every element in $I'$ is a union of elements of $I.$
\end{enumerate}
\end{definition}
\begin{definition}
For $\Gamma \in \CG$, we define the map
\[
\eta=\eta_\Gamma:V\left(\Gamma\right)\to\mathcal{V}
\]
by
\[
\eta\left(v\right) = \left(\varepsilon\left(v\right),k\left(v\right),i_v\left(I\left(v\right)\right)\right).
\]
Here, $i_v$ is as in Definition~\ref{df:iv}.
\end{definition}

\begin{definition}\label{def:CM_av}
Let $\Gamma \in \CG$ and $v\in V\left(\Gamma\right)$. Each bijection
\[
B\left(v\right)\simeq [\rz{k\left(v\right)}]
\]
induces a natural diffeomorphism
\[
\phi_v:\CM_v\to\CM_{\eta\left(v\right)}.
\]
For the rest of the article, we fix one such diffeomorphism for each open vertex in each stable graph.
For $v\in V^C\left(\Gamma\right),$ we have a natural identification $\phi_v : \CM_v \to \CM_{\eta(v)}$ without making any choices. The maps $\phi_v$ extend to the boundary of $\oCM_v,$ and we use the same notation for the extension:
\[
\phi_v:\oCM_v\to\oCM_{\eta\left(v\right)}.
\]
\end{definition}

\begin{rmk}\label{rmk:perm_grp_action}
The symmetric group $S_k$ acts naturally on $\CM_{0,k,l}$ by sending
$\left(\Sigma, \left\{x_i\right\}_1^k , \left\{z_i\right\}_1^l\right)$  to
$\left(\Sigma, \left\{x_{\sigma\cdot{i}}\right\}_1^k, \left\{z_i\right\}_1^l\right)$ for $\sigma\in S_k$.
In the same way, given $\Gamma \in \CG$, the group $\prod_{v\in V\left(\Gamma\right)} S_{k\left(v\right) }$ acts on $\CM_\Gamma$. In particular, for every vertex $v\in V\left(\Gamma\right),$ the symmetric group $S_{k\left(v\right)}$ acts on $\CM_\Gamma$.

Thus, if we had chosen another map $\phi{'}_{v}$ for a vertex $v\in V\left(\Gamma\right)$, then $\phi_v, \phi{'}_v,$ would differ by the action of some $\sigma \in S_{k\left(v\right) }$. That is,
\[
\phi{'}_{v} = \phi_{v}\circ\mathbf{\sigma},
\]
where we denote the group element and its action by the same notation.
\end{rmk}

We also have a map $\nu: \mathcal{V}\to\CG$ which takes the abstract vertex $v = \left(\varepsilon, k, I\right)\in\mathcal{V}$ to the stable graph
$\left(V = V^{O}\cup V^{C} ,E,\ell_I,\ell_B\right)$ such that
\begin{enumerate}
\item
If $\varepsilon = C$, then $V=V^C=\left\{*\right\}$. Otherwise $V=V^O=\left\{*\right\}$.
\item
$E =\emptyset$.
\item
If $\varepsilon = O,$ then $\ell_B\left(*\right)=[\rz{k}]$.
\item
$\ell_I\left(*\right) = I$.
\end{enumerate}
One can easily verify that $\eta\circ\nu = id.$
\begin{nn}
Denote by $\CV_{odd} \subset \CV$ the set of all abstract vertices $v$ such that either $\varepsilon(v) = C$ or $k(v)$ is odd.
\end{nn}

\begin{nn}
Let $\Gamma \in \CG_{odd}$ and $i \in I(\Gamma).$ By Observation~\ref{obs:Bmaps}, we have $i \in I(\CB\Gamma).$ So, we write
\[
\bv{i}{\Gamma} = \eta(v_i(\CB\Gamma)).
\]
\end{nn}

From Observations~\ref{obs:parity} and~\ref{obs:illegals_increase}, we immediately obtain the following.
\begin{obs}\label{obs:bdry_of_abs_vrtx}
Let $\Gamma\in\CG_{odd}, i\in I\left(\Gamma\right),$ and let $v = \bv{i}{\Gamma}.$ Then $v \in \CV_{odd}.$ In addition, for every $\Gamma'\in\partial\Gamma$, either $\bv{i}{\Gamma'} = v$ or $\bv{i}{\Gamma'}\in\partial v$. In the latter case,
\[
\dim_{\C}\CM_\bv{i}{\Gamma'} < \dim_{\C}\CM_v.
\]
\end{obs}

\begin{definition}
Let $\Gamma \in \CG_{odd}$. The \emph{base component} of the interior label $i\in I\left(\Gamma\right)$ is $\CM_\bv{i}{\Gamma}.$
The \emph{base moduli} of $\Gamma$ is the space $\CM_{\CB\Gamma}.$
\end{definition}

Recall Definition~\ref{def:span}. Let $\Gamma\in\CG_{odd},$ let $v \in V(\CB \Gamma)$ and let $i \in I(\Gamma).$
Define
\begin{align}\label{eq:phigi}
\Phi_{\Gamma,v}&:=\phi_v \circ For_{\Gamma,\CB\Gamma_v}:\CM_\Gamma\to\CM_{\eta\left(v\right)}, \\ \Phi_{\Gamma,i} &:= \Phi_{\Gamma,v_i(\CB \Gamma)} : \CM_\Gamma \to \CM_{\bv{i}{\Gamma}}. \notag
\end{align}
We use the same notation for the natural extensions of these maps to the appropriate compactified moduli spaces.
\begin{nn}
Let $v \in \mathcal V_{odd}$ be an abstract vertex. We define a map
\[
\partial_v : \CG_{odd} \to 2^{\CG_{odd}}
\]
by
\[
\partial_v\Gamma = \left\{\left.\Lambda\in\partial\Gamma \right|\exists u\in V\left(\CB\Lambda\right),~\eta\left(u\right)=v \right\}.
\]
Moreover, for $\Gamma \in \CG_{odd}$ we write
\[
\partial_v\CM_{\Gamma}=\coprod_{\Lambda\in\partial_v\Gamma}\CM_{\Lambda}.
\]
By abuse of notation, we define a map
\[
\partial_v : \mathcal V_{odd} \longrightarrow 2^{\CG_{odd}}
\]
by
\[
\partial_{v}u = \partial_{v}\left(\nu\left(u\right)\right).
\]
For $u \in \mathcal V_{odd}$ an abstract vertex, we write
\[
\partial_{v}\CM_u=\phi_{\nu\left(u\right)}\left(\partial_{v}\CM_{\nu\left(u\right)}\right) \subset \oCM_u.
\]
\end{nn}
A crucial property of the base is the following.
Fix $\Gamma \in \CG_{odd}$, a label $i\in I\left(\Gamma\right),$ and $\Gamma'\in\partial\Gamma.$ Write $v = \bv{i}{\Gamma}.$ Let
\[
\Phi_{\Gamma,i}^{\Gamma'}:\CM_{\Gamma'}\to\oCM_{\nu(v)}
\]
be given by the composition
\[
\CM_{\Gamma'} \hookrightarrow \oCM_{\Gamma} \overset{\Phi_{\Gamma,i}}{\longrightarrow} \oCM_{v} \overset{\phi_{\nu(v)}^{-1}}{\longrightarrow} \oCM_{\nu(v)}.
\]
The image of $\Phi_{\Gamma,i}^{\Gamma'}$ is a unique stratum $\CM_{\Lambda} \subset \oCM_{\nu(v)},$ where $\Lambda \in \partial(\nu(v))$ or $\Lambda = \nu(v)$. Note that
\[
\bv{i}{\Gamma'} = \bv{i}{\Lambda},
\]
and denote this abstract vertex by $v'.$ Then by Observation~\ref{obs:bdry_of_abs_vrtx} we have either $v' = v$ or $v'\in\partial v$ and $\Lambda\in\partial_{v'}v.$ See Figure \ref{fig:bdry_base}.
\begin{figure}[t]
\centering
\includegraphics[scale=.25]{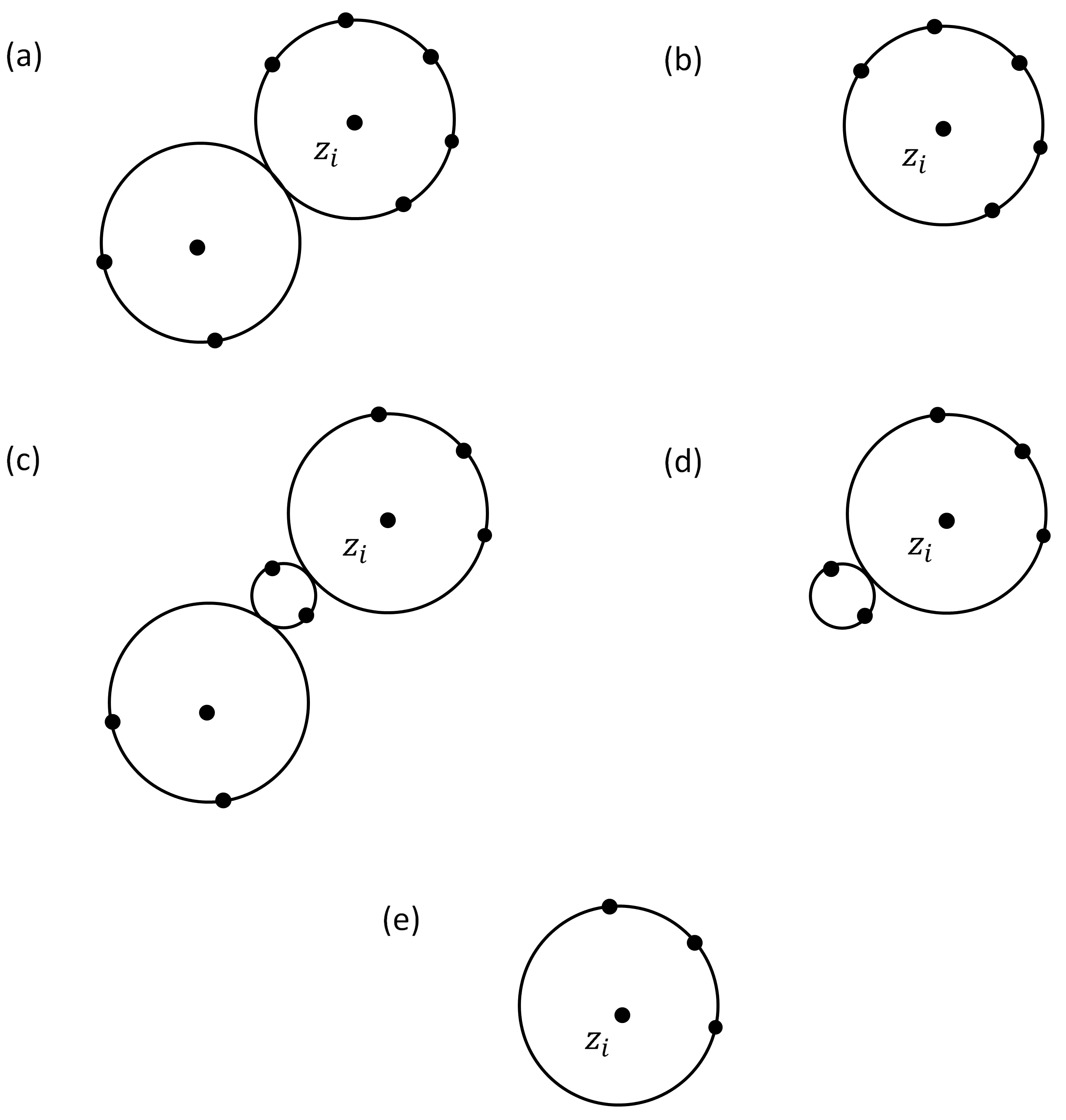}
\caption{(a) shows $\Gamma,$ (b) shows $v,$ (c)~shows $\Gamma',$ (d)~shows $\Lambda,$ and (e) shows $v'.$}
\label{fig:bdry_base}
\end{figure}

\begin{obs}\label{obs:bdry_of_base_moduli}
In the scenario described above, the diagram
\begin{equation}\label{eq:dibbm}
\xymatrix{
    \CM_{\Gamma'}   \ar[rr]^{\Phi_{\Gamma,i}^{\Gamma'}} \ar[dr]_{\Phi_{\Gamma',i}}  &  &
    \CM_{\Lambda} \ar[dl]^{\Phi_{\Lambda,i}}  \\
           &   \CM_{v'}   &
        }
\end{equation}
commutes up to the action of $\sigma\in S_{k\left(v'\right)}.$
\end{obs}
\begin{proof}
Let $u = v_i(\Gamma)$ and let $w = v_i(\CB \Gamma),$ so $\eta(w)= v.$ By Definition~\ref{def:base_operator}, we have a canonical inclusion $j : B(w) \to B(u).$ Let $\tilde \phi_w : B(w) \overset{\sim}{\to} [k(w)^\circ]$ denote the bijection that induces the diffeomorphism $\phi_w : \CM_w\to\CM_v$ as in Definition~\ref{def:CM_av}. Recalling the map $\varsigma_{\Gamma',\Gamma} : V(\Gamma') \to V(\Gamma)$ from Definition~\ref{def:si}, let $U = \varsigma_{\Gamma',\Gamma}^{-1}(u) \subset V(\Gamma').$ Let $\Gamma'_U$ be the graph spanned by $U$ as in Definition~\ref{def:span}. Observe that there is a canonical bijection $h: B(\Gamma'_U) \overset{\sim}{\to} B(u).$ Let $\Gamma'_{U,w}$ denote the stable subgraph of $\Gamma'_U$ obtained by forgetting the boundary marked points not belonging to $h^{-1}(j(B(w)))$ and stabilizing. So,
\[
B(\Gamma'_{U,w}) = h^{-1}(j(B(w))).
\]
More explicitly,
\begin{align}\label{eq:bg'uw}
B(&\Gamma'_{U,w}) =\\
& \quad\, = \quad \bigcup_{x \in U} \ell_B^\Gamma(x) \cup \left\{\left.i_x^{\Gamma'}(e)\right|e = \{x,y\}\in E^B_{legal}(x), y \notin U\right\}\notag\\
&\overset{\text{Obs.~\ref{obs:illegals_increase}}}{=} \ell_B^\Gamma(u) \cup \left\{\left.i_u^\Gamma(e)\right|e \in E^B_{legal}(u)\right\} \notag\\
&\overset{\text{Def.~\ref{def:base_operator}}}{=} B(w) \notag\\
& \quad\, \subset i_u^{\Gamma}(B(u)).\notag
\end{align}
Let
\[
\hat \phi_w : B(\Gamma'_{U,w}) \overset{\sim}{\longrightarrow} [k(w)^\circ]
\]
be the bijection given by
$
\hat \phi_w = \tilde \phi_w \circ j^{-1} \circ h.
$
Since
\[
|B(\Gamma'_{U,w})| = k(w),
\]
we have $\Gamma'_{U,w} \in \CG_{odd}.$
Also,
\[
I(\Gamma'_{U,w}) = I(\Gamma'_U) = i_u^{\Gamma}(I(u)) = i_w^{\CB\Gamma}(I(w)) = I(v).
\]
The stable graph $\Lambda$ is isomorphic to $\Gamma'_{U,w}$ as an abstract graph, the interior labels are the same, and the boundary labels are obtained by applying $\hat \phi_w$ to the boundary labels of $\Gamma'_{U,w}.$ Let
\[
\phi_w' : \CM_{\Gamma'_{U,w}} \overset{\sim}{\longrightarrow} \CM_\Lambda
\]
denote the induced diffeomorphism.
\begin{figure}[t]
\centering
\[
\xymatrix{
&\oCM_{\Gamma} \ar[r]^{For} & \oCM_{\CB\Gamma_w} \ar[r]^{\phi_w} & \oCM_{v} \ar[d]^{\phi_{\nu(v)}^{-1}}\\
& & & \oCM_{\nu(v)}\\
&\CM_{\Gamma'} \ar `l[d]`[dd]_{\Phi_{\Gamma',i}}[ddr] \ar[uu]\ar[r]_{For}\ar[d]^{For} \ar@(ur,ul)[rr]^{\Phi_{\Gamma,i}^{\Gamma'}}  & \CM_{\Gamma'_{U,w}} \ar[r]^\sim_{\phi'_w} \ar[d]_{For} & \CM_{\Lambda} \ar[d]_{For}\ar[u] \ar `r[d]`[dd]^{\Phi_{\Lambda,i}}[ddl] & \\
&\CM_{(\CB\Gamma')_i} \ar[r]^\sim\ar[dr]_{\phi_{\bv{i}{\Gamma'}}} & \CM_{(\CB\Gamma'_{U,w})_i} \ar[r]^\sim_{\bar\phi'_w} & \CM_{(\CB \Lambda)_i} \ar[dl]^{\phi_{\bv{i}{\Lambda}}} & \\
&& \CM_{v'} &
}
\]
\caption{}\label{fig:diagram}
\end{figure}
For a stable graph $\Upsilon$ and $m  \in I(\Upsilon),$ write $\Upsilon_m$ for the single vertex graph spanned by $\{v_m(\Upsilon)\},$ that is, $\Upsilon_m = \Upsilon_{v_m(\Upsilon)}.$ It follows from the first equality of~\eqref{eq:bg'uw} that
\[
(\CB \Gamma')_i = (\CB \Gamma'_{U,w})_i.
\]
Since $\Gamma'_{U,w}$ and $\Lambda$ are the same up to applying $\hat\phi_w$ to the boundary labels, the same is true for $(\CB\Gamma'_{U,w})_i$ and $(\CB\Lambda)_i.$  Let
\[
\bar\phi_w': \CM_{(\CB\Gamma'_{U,w})_i} \overset{\sim}{\longrightarrow} \CM_{(\CB\Lambda)_i}.
\]
denote the induced diffeomorphism.
We sum up the situation with the diagram of Figure~\ref{fig:diagram}. All forgetful maps are labeled by $For$ omitting the usual subscripts since they are clear from context. The left small square commutes by Observation~\ref{obs:fuf}. The right small square commutes because the forgetful map does not change when $\hat \phi_w$ is applied to boundary labels. The triangle commutes only up to the action of $S_{k(v')}$ because the maps $\phi_{\bv{i}{\Gamma'}}$ and $\phi_{\bv{i}{\Lambda}}$ depend on the arbitrary choice made in Definition~\ref{def:CM_av}. The observation follows.
\end{proof}

\begin{nn}
For $v \in \CV_{odd}$ write
\[
\partial^{eff} v = \{ v' \in \CV_{odd}|\partial_{v'}v \neq \emptyset\},
\]
and
\[
\partial^{eff}_i v = \left.\left\{v' \in \partial^{eff}v \right| i \in I(v)\right\}.
\]

For $v' \in \partial^{eff} v$ we define
\[
\Phi_{v,v'}:\partial_{v'}\oCM_v\to\CM_{v'}
\]
by
\[
\Phi_{v,v'} = \coprod_{\Lambda\in\partial_{v'}v}\Phi_{\Lambda,v'}.
\]
\end{nn}

\begin{definition}
Given $\CC \subseteq \CG_{odd},$ define
\[
\mathcal{V}_{\CC} = \{\eta\left(v\right)|v\in V\left(\CB\Gamma\right), \;\Gamma\in \CC\}.
\]
Define
\[
\mathcal{V}_\CC^i = \left.\left\{v \in \mathcal{V}_\CC\right| i \in I(v)\right\}.
\]
\end{definition}

\subsection{Special canonical boundary conditions}\label{ssec:scbc}
We return to the line bundles $\CL_i\to\oCM_{0,k,l}$ in order to define special canonical boundary conditions. We consider only the case $k$ is odd, which is necessary for $\dim_\C\oCM_{0,k,l}$ to be an integer.

Denote by $\tilde \pi_\Gamma : \CC_\Gamma \to \CM_\Gamma$ the universal curve. Thus $\tilde \pi_\Gamma^{-1}([\Sigma]) = \Sigma.$ Denote by $\CU_\Gamma \subset \CC_\Gamma$ the open subset on which $\tilde\pi_\Gamma$ is a submersion, and let $\pi_\Gamma = \tilde \pi_\Gamma|_{\CU_\Gamma}.$ Thus $\pi_\Gamma^{-1}([\Sigma])$ is the smooth locus of $\Sigma.$ For $i \in I(\Gamma),$ denote by $\mu_i : \CM_\Gamma \to \CU_\Gamma$ the section of $\pi_\Gamma$ corresponding to the $i^{th}$ interior marked point. Denote by $\CL_\Gamma \to \CU_\Gamma$ the vertical cotangent line bundle, which is by definition the cokernel of the map $d\pi_\Gamma^* : T^*\CM_\Gamma \to T^*\CU_\Gamma.$ So, $\CL_i = \mu_i^* \CL_\Gamma.$ Let $\Gamma'$ be a stable subgraph of $\Gamma.$ Then the forgetful map $For_{\Gamma,\Gamma'} : \CM_\Gamma \to \CM_{\Gamma'}$ lifts canonically to a map $\widetilde{For}_{\Gamma,\Gamma'} : \CU_{\Gamma} \to \CU_{\Gamma'}.$
Let $t_{\Gamma,\Gamma'}$ be defined by the following diagram.
\[
\xymatrix{
\pi_\Gamma^* T^*\CM_\Gamma \ar[d]^{d\pi_\Gamma^*} &&  \pi_{\Gamma}^* For_{\Gamma,\Gamma'}^* T^*\CM_{\Gamma'} \ar[ll]^(.55){\pi^*_\Gamma dFor_{\Gamma,\Gamma'}^*}\ar@{-}[r]^{\sim} & \widetilde{For}_{\Gamma,\Gamma'}^*\pi_{\Gamma}^* T^*\CM_{\Gamma'}\ar[d]^{\widetilde{For}_{\Gamma,\Gamma'}^*d\pi_{\Gamma'}^*} \\
T^*\CU_\Gamma \ar[d]  &&& \widetilde{For}_{\Gamma,\Gamma'}^* T^*\CU_{\Gamma'} \ar[d]\ar[lll]^{d\widetilde{For}_{\Gamma,\Gamma,'}} \\
\CL_\Gamma &&& \widetilde{For}_{\Gamma,\Gamma'}^*\CL_{\Gamma'} \ar@{-->}[lll]^{t_{\Gamma,\Gamma'}}
}
\]
The diagram implies the following.
\begin{obs}\label{obs:tGGp}
The morphism $t_{\Gamma,\Gamma'}$ is an isomorphism except on components of $\CU_\Gamma$ that are contracted by $\widetilde{For}_{\Gamma,\Gamma'}$, where it vanishes identically.
\end{obs}
For $i \in I(\Gamma),$ Observation~\ref{obs:parity} implies that $\widetilde{For}_{\Gamma,\CB\Gamma}$ and $\widetilde{For}_{\Gamma,\CB\Gamma_{v_i(\CB\Gamma)}}$ do not contract the component containing the $i^{th}$ interior marked point. The following is an immediate consequence.
\begin{obs}\label{obs:identification_bundle_over_base}
For $\Gamma \in \CG_{odd},$ we have isomorphisms
\[
\CL_i \simeq F_\Gamma^* \CL_i, \qquad \CL_i  \simeq \Phi_{\Gamma,i}^*\CL_i.
\]
given by $\mu_i ^* t_{\Gamma,\CB\Gamma}$ and $\mu_i^* t_{\Gamma,\CB\Gamma_{v_i(\CB\Gamma)}}.$
\end{obs}
\begin{rmk}\label{rm:equiv}
The natural action of the symmetric group $S_{k\left(v_i\left(\Gamma\right)\right)}$ on $\CM_\Gamma$ by permutations of the boundary labels and edges lifts canonically to a natural action on the bundle $\CL_i\to\CM_\Gamma.$ The same goes for the natural action of $S_{k\left(v_i\left(\Gamma\right)\right)}$ on $\CM_{v}$ and $\CM_{\eta(v)}.$ The isomorphisms of Observation~\ref{obs:identification_bundle_over_base} are equivariant with respect to these actions.
\end{rmk}
\begin{nn}
Let $\Upsilon \in \CG_{odd}.$ For a subset $C \subseteq \partial\oCM_{\Upsilon},$ a vector bundle $E \to C$, a multisection $s\in C_m^\infty\left(C,E\right)$, and $\Gamma \in \pB\Upsilon,$ we write
\[
s^\Gamma:=\left. s\right|_{\CM_\Gamma \cap C}.
\]
\end{nn}

Observation~\ref{obs:extension_of_canonical_conds} allows us to generalize the definition of canonical multisection as follows. Let $\Upsilon \in \CG_{odd}$, let
$C \subseteq \partial\oCM_{\Upsilon},$ and let
\[
\CC = \{ \Gamma \in \pB\Upsilon|C \cap \CM_{\Gamma} \neq \emptyset\}.
\]
\begin{definition}
A multisection $s$ of
\[
E = \bigoplus_{i\in\left[l\right]}\CL_i^{\oplus a_i}\rightarrow C
\]
is called \emph{canonical} if for each $\Gamma \in \CC$, there exists a section $s^{\CB\Gamma}$ of
\[
\bigoplus_{i\in\left[l\right]}\CL_i^{\oplus a_i}\rightarrow\CM_{\CB\Gamma}
\]
such that $s^\Gamma = F_\Gamma^*s^{\CB\Gamma}|_{C\cap\CM_\Gamma}.$
\end{definition}
The following refinement of canonical multisections is useful in proofs.
\begin{definition}\label{def:scm}
A multisection $s\in C_m^\infty\left(C,\CL_i\right)$ is said to be \emph{pulled back from the base component}, or \emph{pulled back from the base} for short, if for every
$v\in\mathcal{V}_{\CC}^{i}$
there exists $s^v\in C_m^\infty\left(\CM_v,\CL_i\right)$ such that for every $\Gamma\in\CC$ with $\bv{i}{\Gamma} = v$, we have
\[
s^\Gamma = \Phi_{\Gamma,i}^*s^v|_{C\cap\CM_\Gamma}.
\]
A multisection $s\in C_m^\infty\left(\CM_{0,B,I}, \CL_i\right)$ is said to be \emph{invariant} if it is invariant under the action of the permutation group $S_B.$

A multisection $s\in C_m^\infty\left(C,\CL_i\right)$ is \emph{special canonical} if it is pulled from the base, and for every $v\in\mathcal{V}_{\CC}^{i}$ the multisection $s^v$ is invariant.
We write $\CS_i = \CS_{i,k,l}$ for the vector space of special canonical multisections of~$\CL_i$ over $C = \partial\oCM_{0,k,l}.$ Below, we use the notation $s^v$ as in this definition.
\end{definition}
\begin{rmk}\label{rmk:consist}
It is straightforward to verify that a multisection which is pulled back from the base is consistent.
\end{rmk}

\begin{rmk}\label{rmk:scc}
Let $s \in C^\infty_m(C,\CL_i)$ be special canonical. By Observation~\ref{obs:fuf}, the map $\Phi_{\Gamma,i}:\CM_\Gamma\to\CM_\bv{i}{\Gamma}$ factors as the composition
\[
\xymatrix{
\CM_\Gamma \ar[rr]^{F_\Gamma} && \CM_{\CB\Gamma}\ar[rr]^{\Phi_{\CB\Gamma,i}} && \CM_{\bv{i}{\Gamma}}.
}
\]
So there exists $s^{\CB\Gamma} \in C^\infty_m(\CM_{\CB\Gamma},\CL_i)$ such that
\[
s^\Gamma = F_\Gamma^* s^{\CB\Gamma}|_{C\cap \CM_\Gamma}.
\]
It follows from Observation~\ref{obs:extension_of_canonical_conds} that the vector space $\bigoplus{\CS}_i^{\oplus a_i}$ is a subvector space of $\CS$. Below, we use the notation $s^{\CB\Gamma}$ as in this remark.
\end{rmk}

\subsection{Forgetful maps, cotangent lines and base}
We introduce notations and formulate the basic properties of pull-backs of cotangent lines by forgetful maps.
\begin{obs}\label{obs:tfor}
Let $\Gamma \in \CG_{odd}$ and $i\in I\left(\Gamma\right)$. Then $\CB for_i(\Gamma)$ is a stable subgraph of $for_i(\CB\Gamma)$ and $For_{for_i(\CB \Gamma),\CB for_i(\Gamma)}$ is a diffeomorphism. Indeed, vertices and markings are in one-to-one correspondence. Moreover, the following diagram commutes.
\[
\xymatrix{
\CM_{\Gamma} \ar[rrrr]^(.45){For_i}\ar[d]^{ F_{\Gamma}} &&&& \CM_{for_i\left(\Gamma\right)} \ar[d]^{F_{for_i\left(\Gamma\right)}}\\
\CM_{\CB\Gamma} \ar[r]^(.4){For_i} & \CM_{for_i(\CB\Gamma)} \ar[rrr]^{For_{for_i(\CB\Gamma),\CB for_i(\Gamma)}} &&& \CM_{\CB for_i\left(\Gamma\right)}
}
\]
This is a consequence of Observation~\ref{obs:fuf}.
\end{obs}

\begin{nn}\label{nn:D_i1}
Let $k$ be odd, let $I \subseteq [l+1]$ and let $i \in I\cap[l].$ If $l+1 \in I$, denote by $D_i \subset \oCM_{0,k,I}$ the locus where the marked points $z_i,z_{l+1},$ belong to a sphere component that contains only them and a unique interior node. If $l+1 \notin I,$ set $D_i = \emptyset.$ For $\Gamma \in \partial\Gamma_{0,k,l+1},$ define $D_i \subset \oCM_\Gamma, D_i \subset \CM_{\CB\Gamma},$ similarly.

Write
\[
\partial D_i = D_i\cap\partial\oCM_{0,k,l+1}.
\]
Let $\CGDi \subset \pB\Gamma_{0,k,l+1}$ be the subset such that
\[
\partial D_i = \coprod_{\Gamma \in \CG_{D_i}} \CM_\Gamma.
\]
\end{nn}

\begin{obs}\label{obs:forv}
For $\Gamma \in \pB\Gamma_{0,k,l+1}\setminus \CGDi,$ the following diagram commutes.
\[
\xymatrix{ \CM_\Gamma \ar[r]^(.4){For_{l+1}}\ar[d]^{\Phi_{\Gamma,i}} & \CM_{for_{l+1}{\Gamma}}\ar[d]^{\Phi_{for_{l+1}\Gamma,i}} \\
\CM_{\bv{i}{\Gamma}} \ar[r]^(.4){For_{l+1}} & \CM_{\bv{i}{for_{l+1}\Gamma}}
}
\]
Again, this is a consequence of Observation~\ref{obs:fuf}.
\end{obs}

\begin{nn}\label{nn:L^{'}}
Write
\[
\CL_i' = For_{l+1}^* \CL_i \to \oCM_{0,k,I}.
\]
For $\Gamma \in \partial\Gamma_{0,k,l+1}$ write
\[
\CL_i' = For_{l+1}^*\CL_i \to \oCM_{\Gamma}, \qquad \CL_i' = For_{l+1}^*\CL_i \to \oCM_{\CB\Gamma}.
\]
Denote by $\CS_i'  = \CS_{i,k,l+1}'\subset C_m^\infty(\partial\oCM_{0,k,l+1},\CL_i')$ the vector space of pull-backs of sections in $\CS_{i,k,l}$ by $For_{l+1}.$ Denote by
\[
\tilde t_i : \CL_i' \to \CL_i
\]
the morphism given by $\tilde t_i|_{\CM_\Gamma} = \mu_i^* t_{\Gamma,for_i(\Gamma)}.$
\end{nn}

\begin{lm}\label{lm:fmor}
\mbox{}
\begin{enumerate}
\item\label{it:tr}
The morphism $\tilde t_i$ considered as a section of $(\CL_i')^*\otimes\CL_i$ vanishes transversally exactly at~$D_i.$
\item\label{it:m1}
$For_{l+1}$ maps $D_i$ diffeomorphically onto $\oCM_{0,k,l}$ carrying the orientation induced on $D_i$ by $\tilde t_i$ to the orientation $o_{0,k,l}$ on~$\oCM_{0,k,l}.$
\item\label{it:pb}
The morphism $\tilde t_i$ satisfies
\[
F_{\Gamma}^*\tilde t_i = \tilde t_i,
\]
and for $\Gamma \in \pB \Gamma_{0,k,l+1}\setminus\CGDi,$
\[
\Phi_{\Gamma,i}^*\tilde t_i = \tilde t_i.
\]
Here, we have used the isomorphisms of Observation~\ref{obs:identification_bundle_over_base} to identify relevant domains and ranges of $\tilde t_i.$
\item\label{it:in}
The morphism $\tilde t_i$ is invariant under permutations of the boundary marked and nodal points as in Remark~\ref{rm:equiv}.
\end{enumerate}
\end{lm}
\begin{proof}
Observation~\ref{obs:tGGp} implies that $\tilde t_i$ vanishes exactly at $D_i.$ It follows from the definitions that $For_{l+1}$ maps $D_i$ diffeomorphically onto $\oCM_{0,k,l}.$ So the transversality and orientation statements are equivalent to the following claim. Let $p \in \oCM_{0,k,l},$ let $F_p = For_{l+1}^{-1}(p)$ and equip $F_p$ with its complex orientation. Then $\tilde t_i|_{F_p}$ vanishes with multiplicity $+1$ at the unique point $\check p \in D_i \cap F_p.$

To prove the claim, we construct a map $\alpha : D^2 \to F_p,$ that preserves complex orientations and calculate $\tilde t_i \circ \alpha$ in an explicit trivialization of $\alpha^*\CL_i.$ Indeed, let $\Sigma = (\{\Sigma_\alpha\},\sim)$ be a stable disk representing $p.$ Denote by $\Sigma_0$ the component of $\Sigma$ containing the marked point $z_{i}.$ Denote by $B_r \subset \C$ the disk of radius $r$ centered at $0.$ Let $U \subset \Sigma_0$ be an open neighborhood of $z_i$ with local coordinate
\[
\xi : U \overset{\sim}{\to} B_2, \qquad \xi(z_1) = 0.
\]
For $z \in B_1,$ let $\Sigma_0^z$ be obtained from $\Sigma_0$ as follows. If $z \neq 0,$ add the marked point $z_{l+1} = \xi^{-1}(z).$ If $z = 0,$ replace $z_i$ with a new marked point $z_0.$ Denote by $S$ the marked sphere $(\C \cup \{\infty\}, z_{-1}, z_i, z_{l+1})$ where $z_{-1} = \infty,\, z_i = 0,$ and $z_{l+1} = 1.$ For $z \neq 0,$ let $\Sigma^z$ be the stable disk $(\{\Sigma_\alpha\}_{\alpha \neq 0}\cup \{\Sigma_0^z\},\sim).$ For $z  = 0,$ let
\[
\Sigma^z = (\{\Sigma_\alpha\}_{\alpha \neq 0}\cup \{\Sigma_0^z,S\},\sim_0),
\]
where $\sim_0$ is obtained from $\sim$ by adding the relation $z_0 \sim_0 z_{-1}.$ Define $\alpha : B_1 \to F_p$ by $\alpha(z) = \Sigma^z.$

For $z \in B_1,$ the stable disk $\Sigma^z$ is the deformation of the stable disk $\Sigma^0$ obtained by removing appropriate disks around the nodal points $z_0 \in \Sigma_0$ and $z_{-1} \in S$ and identifying annuli adjacent to the resulting boundaries. More explicitly, denoting by $\zeta$ the standard coordinate on $S=\C \cup \{\infty\},$ we glue the surfaces
\[
\Sigma^z \setminus \xi^{-1}\left(B_{\sqrt{2|z|}}\right), \qquad B_{\sqrt{3/|z|}} \subset S
\]
along the map $\zeta \mapsto \xi^{-1}(z\zeta)$ for $\zeta \in B_{\sqrt{3/|z|}} \setminus B_{\sqrt{2/|z|}}.$ Thus we take $d \zeta|_{z_i} \in T^*_{z_i}S \simeq T^*_{z_i}\Sigma^z$ as a frame for $\alpha^*\CL_i.$ On the other hand, $d\xi|_{z_i} \in T^*_{z_i}\Sigma$ is a frame for $\alpha^*\CL_i'.$ Since $\xi = z\zeta,$ we have $\tilde t_i\left(d\xi|_{z_i}\right) =  z d\zeta|_{z_i}.$ Thus $\tilde t_i$ vanishes with multiplicity $1$ at $z = 0,$ which is the point $\hat p,$ as claimed. So we have proved parts~\ref{it:tr} and~\ref{it:m1} of the lemma.

Part~\ref{it:pb} follows from Observation~\ref{obs:tfor} and part~\ref{it:in} follows from the definition of $\tilde t_i.$
\end{proof}
\begin{nn}\label{nn:OD}
Denote by $\mathcal{O}\left(D_i\right)$ the line bundle
\[
Hom(\CL_i',\CL_i) = (\CL_i')^*\otimes \CL_i.
\]
So $\tilde t_i$ is a section of $\mathcal{O}\left(D_i\right)$. Write
\[
t_i = \left. \tilde t_i\right|_{\partial\oCM_{0,k,l+1}}.
\]
\end{nn}
Lemma~\ref{lm:fmor} shows that $\mathcal{O}\left(D_i\right)$ is the trivial complex line bundle twisted at $D_i$ as implied by the notation. Moreover, tautologically,
\[
\CL_i \simeq \CL_i' \otimes \mathcal{O}(D_i).
\]

The following observation is a consequence of Observations~\ref{obs:identification_bundle_over_base} and~\ref{obs:tfor} and the relevant definitions.
\begin{obs}
For $\Gamma \in \partial \Gamma_{0,k,l+1}$ and $i \in [l]$, we have
\[
F_{\Gamma}^*\mathcal{O}(D_i) \simeq \mathcal{O}(D_i).
\]
\end{obs}

\begin{obs}\label{obs:D_i}
If $s'\in\CS'_i$, then $s = s' t_i$ belongs to $\CS_i$ and vanishes on $D_i$.
\end{obs}
\begin{proof}
Using Observation~\ref{obs:forv}, we see that for $\Gamma \in \partial\Gamma_{0,k,l}\setminus\CGDi$ we may take
\[
s^\bv{i}{\Gamma} = For_{l+1}^*(s')^\bv{i}{for_{l+1}\Gamma}\tilde t_i.
\]
For $\Gamma \in \CGDi,$ we take $s^\bv{i}{\Gamma} = 0.$
\end{proof}
\begin{rmk}\label{rmk:L'}
Recall Observation~\ref{obs:extension_of_canonical_conds} and Remark~\ref{rmk:scc}. A multisection $s\in\CS'_i$ behaves similarly. Namely, for each $\Gamma \in \pB\Gamma_{0,k,l+1},$ there exists
\[
s^{\CB\Gamma}\in C^\infty_m\left(\CM_{\CB\Gamma}, \CL_i'\right)
\]
such that
\[
s^\Gamma = F_\Gamma^* s^{\CB\Gamma}.
\]
This follows from Observation~\ref{obs:tfor}.
\end{rmk}

\subsection{Construction of multisections and homotopies}
In this section we prove Theorem~\ref{thm:intersection_numbers_well_defined}, namely, the open descendent integrals are well defined. In addition we construct special canonical multisections of special types, which we later use to prove the geometric recursions.
\begin{nn}
For a bundle $E\to M$, we denote by $0$ its $0$-section. Given a multisection $s,$ the notation $s\pitchfork 0$ means that $s$ is transverse to the $0$-section. See Appendix \ref{app:euler}.
\end{nn}
\begin{prop}\label{prop:single_section}
Consider $\CL_i\to\oCM_{0,k,l}$, with $k$ odd.
\begin{enumerate}
\item\label{it:single_section_a}
For any $p\in\partial\oCM_{0,k,l}$ one can find $s \in \CS_i$ which does not vanish at $p$. Hence, one can choose finitely many such multisections which span the fiber of $\CL_i$ over each point of $\partial\oCM_{0,k,l}$.
\item\label{it:single_section_b}
For $i \in [l]$, and
\[
p\in\partial\oCM_{0,k,l+1}\setminus\partial D_{i}, \qquad q \in \partial D_{i},
\]
one can find $s \in \CS_i$ of the form
\[
s = s' t_{i},\qquad s'\in\CS'_i,
\]
that does not vanish at $p$, vanishes on $\partial D_{i}$ and such that $ds|_q$ surjects onto $(\CL_i)_q.$

Hence, one can choose finitely many such multisections that span the fiber of $\CL_i$ over each point of $\partial\oCM_{0,k,l+1}\setminus\partial D_{i}$ and such that images of their derivatives span the fiber of $\CL_i$ at each point of $D_{i}.$
\end{enumerate}
\end{prop}
\begin{proof}
In both cases the `hence' part follows immediately from the previous part because of the compactness of $\partial\oCM_{0,k,l}$. We first prove part~\ref{it:single_section_a}.
To construct the special canonical multisection $s$, it suffices to construct multisections $\bar s^v \in C_m^\infty(\oCM_{v},\CL_i)$ for each abstract vertex $v \in \mathcal{V}_{\pB\Gamma_{0,k,l}}^i$ that have certain properties.

Let $\CM_{\Gamma^*}$ be the boundary stratum of $\oCM_{0,k,l}$ that contains $p$, let
\[
v^* = \bv{i}{\Gamma^*},
\]
and write $k^*=k\left(v^{*}\right).$ Write
\[
 \hat{p}_1,\ldots,\hat{p}_{k^*!}\in\CM_{v^{*}}
\]
for $\Phi_{\Gamma^*,i}\left(p\right)$ and its conjugates under the action of $S_{k^*}$ on $\CM_{v^{*}}.$

The properties the  multisections $\left\{\bar s^v\right\}$ should satisfy are as follows.
\begin{enumerate}[label=(\roman*)]
\item\label{it:1}
For all $v \in \mathcal{V}^i_{\pB\Gamma_{0,k,l}}$ the multisection $\bar s^v$ is invariant.
\item\label{it:2}
For all $v,v' \in \mathcal{V}^i_{\pB\Gamma_{0,k,l}}$ such that $v'\in\partial^{eff}_{i} v$, we have
\[
\left.\bar s^v\right|_{\partial_{v'}\CM_v} = \Phi_{v,v'}^*\bar s^{v'}.
\]
\item\label{it:4}
No branch of $\bar s^{v^*}$ vanishes at $\hat p_1,\ldots,\hat p_{k^*!}.$
\end{enumerate}
For $v \in \mathcal{V}_{\pB\Gamma_{0,k,l}}^i,$ let
\[
s^v = \bar s^v|_{\CM_v}.
\]
Define the multisection $s$ by
\[
s^\Gamma = \Phi_{\Gamma,i}^*s^{\bv{i}{\Gamma}}, \qquad \Gamma \in \partial^B\Gamma_{0,k,l},
\]
as in Definition~\ref{def:scm}. So, if we show $s$ is smooth, then it is clearly special canonical. To show $s$ is smooth, it suffices to show that
\[
s|_{\oCM_\Gamma} = \Phi_{\Gamma,i}^*\bar s^{\bv{i}{\Gamma}}, \qquad \Gamma \in \partial^B\Gamma_{0,k,l}.
\]
Equivalently, we can prove
\begin{equation}\label{eq:cmg'}
s|_{\CM_{\Gamma'}} =  \Phi_{\Gamma,i}^*\bar s^{\bv{i}{\Gamma}}|_{\CM_{\Gamma'}}, \qquad \Gamma' \in \partial \Gamma, \qquad \Gamma \in \partial^B\Gamma_{0,k,l}.
\end{equation}
Consider the following diagram, which uses the notation of Observation~\ref{obs:bdry_of_base_moduli}. In particular, $v' = \bv{i}{\Gamma'}.$
\[
\xymatrix{
    \CM_{\Gamma'}   \ar[dd]_{\Phi_{\Gamma',i}}   \ar[rd]_{\Phi_{\Gamma,i}^{\Gamma'}}\ar@{^(->}[rr]    &   &    \oCM_\Gamma    \ar[dd]^{\Phi_{\Gamma,i}}\\
    &   \CM_{\Lambda}   \ar@{^(->}[rd]    \ar[ld]^{\Phi_{v,v'}}   \\
    \CM_{v'} &  &   \oCM_v
}
\]
The large triangle commutes by the definition of $\Phi_{\Gamma,i}^{\Gamma'},$ and the small triangle commutes up to the action of $S_{k(v')}$ by Observation~\ref{obs:bdry_of_base_moduli}.
Using the definition of $s$ and the commutativity of the large triangle, equation~\eqref{eq:cmg'} is the same as
\[
\Phi_{\Gamma',i}^*s^{\bv{i}{\Gamma'}} = (\Phi_{\Gamma,i}^{\Gamma'})^* \bar s^{\bv{i}{\Gamma}}.
\]
By the compatibility property~\ref{it:2}, this is equivalent to
\[
\Phi_{\Gamma',i}^*s^{\bv{i}{\Gamma'}} = (\Phi_{\Gamma,i}^{\Gamma'})^* \Phi_{v,v'}^*\bar s^{v'}.
\]
The preceding equation follows from commutativity of the small triangle in the diagram up to the action of $S_{k(v')},$ the invariance property~\ref{it:1}, Observation~\ref{obs:bdry_of_abs_vrtx}, and Remark~\ref{rm:equiv}. Consistency of $s$ follows from Remark~\ref{rmk:consist}. Property~\ref{it:4} implies that $s$ does not vanish at $p.$

We construct the multisections $\bar s^{v}$ by induction on $\dim_\C\CM_v$. Start the induction with $\dim_\C \CM_v = -1.$ Then the multisections $\bar s^v$ exist trivially, since there are no such $v.$

Assume we have constructed multisections $\bar s^u$ that satisfy properties \ref{it:1}-\ref{it:4} for all $u \in \mathcal{V}_{\pB\Gamma_{0,k,l}}^i$ such that $\dim_\C \CM_u \leq m.$
Let $v \in \mathcal{V}_{\pB\Gamma_{0,k,l}}^i$ be an abstract vertex such that $\dim_\C \CM_v = m+1$. By induction we have defined $\bar s^{v'}$ for all $v'\in\partial^{eff}_{i}v,$ as for such $v'$ we have
\[
\dim_\C \CM_{v'} < \dim_\C \CM_v.
\]
Define the section $s_1$ on
\[
\partial \oCM_v = \bigcup_{v'\in\partial_{i}^{eff}v}\;\partial_{v'}\oCM_v
\]
by
\begin{equation}\label{eq:defs1}
\left.s_1\right|_{\partial_{v'}\oCM_v}= \Phi_{v,v'}^*\bar s^{v'}.
\end{equation}
The induction hypotheses on compatibility~\ref{it:2} and invariance~\ref{it:1}, Observation~\ref{obs:bdry_of_abs_vrtx}, Observation~\ref{obs:bdry_of_base_moduli} and Remark~\ref{rm:equiv}, imply that the section $s_1$ thus defined is smooth on $\partial\oCM_v.$ Consistency of $s_1$ follows directly from the defining equation~\eqref{eq:defs1}. So, we may extend $s_1$ smoothly to all $\oCM_v$. If $v = v^{*},$ we make sure that the extension is non-vanishing at $\hat{p}_1,\ldots,\hat{p}_{k^*!}\in \CM_{v^*}$. We denote the resulting multisection by $s_1$ as well. It satisfies the compatibility condition~\ref{it:2} by construction.

Define $\bar s^v$ to be the $S_{k\left(v\right)}$ symmetrization of $s_1.$ See Appendix \ref{app:euler}, Definition \ref{df:sym}. So $\bar s^v$ satisfies the invariance condition~\ref{it:1}. But by the induction hypothesis on
invariance~\ref{it:1} and Remark~\ref{rm:equiv}, $\bar s^v|_{\partial \oCM_v} = s_1.$ So, $\bar s^v$ satisfies the compatibility condition~\ref{it:2} as well.

For case \ref{it:single_section_b}, write
\[
For_{l+1}(p) = p',\qquad For_{l+1}(q) = q'.
\]
Using case \ref{it:single_section_a}, construct a special canonical multisection $s_1$ of
\[
\CL_i\to\partial\oCM_{0,k,l}
\]
that does not vanish at $p'.$ Construct a second special canonical multisection $s_2$ of $\CL_i\to\partial\oCM_{0,k,l}$ that does not vanish at $q'.$ Denote by $s_3$ a linear combination of $s_1$ and $s_2$ that does not vanish at $p',q'$. Then $s = s_3 t_{i}$ satisfies our requirements by Observation \ref{obs:D_i}.

\end{proof}

Another ingredient we need for the proof of Theorem \ref{thm:intersection_numbers_well_defined} is the following transversality theorem.
\begin{thm}\label{tm:hirschnm}
Let $V$ be a manifold, let $N$ be a manifold with corners, and let $\E\to N$ be a vector bundle.  Denote by $p_N : V \times N \to N$ the projection. Let
\[
F: V\to C^\infty\left(N, \E\right), \qquad v \mapsto F_v,
\]
satisfy the following conditions:
\begin{enumerate}
\item
The section
\[
F^{ev}\in C^\infty(V\times N, p_N^*\E), \qquad
F^{ev}\left(v,x\right)= F_v\left(x\right),
\]
is smooth.
\item
$F^{ev}$ is transverse to $0.$

\end{enumerate}
Then the set
\[
\left\{\left.v\in V \right| F_v\pitchfork 0\right\}
\]
is residual.
\end{thm}
\begin{rmk}
A similar theorem may be found in \cite[pp. 79-80]{Hirsch} in the more general setting where $C^\infty(N,\E)$ is replaced by the space of smooth maps between two manifolds. However, the manifolds considered do not have boundary or corners. In~\cite{Joyce}, Joyce defines a notion of smooth maps of manifolds with corners that guarantees the existence of fiber products for transverse smooth maps. In Joyce's terminology, a map of manifolds with corners that is smooth in each coordinate chart is called weakly smooth. To be smooth, it must satisfy an additional condition at corners. Since we consider only sections of vector bundles, the section $F^{ev}$ is automatically smooth if it is weakly smooth. Thus $(F^{ev})^{-1}(0)$, being a transverse fiber product, is a manifold with corners, and the proof given in~\cite{Hirsch} goes through for our case as well.
\end{rmk}
As a consequence, we have the following theorem on multisection transversality. Relevant operations on multisections are reviewed in Appendix~\ref{app:euler}. See, in particular, Definition \ref{df:sp} for the definition of summation.
\begin{thm}\label{thm: hirsch} We continue with the notation of Theorem~\ref{tm:hirschnm} in the special case where $V$ is the vector space $\R^n.$ Fix $s_0,\ldots,s_n\in C_m^\infty(N,\E).$ Take
\[
F: V \to C_m^\infty(N,\E)
\]
to be the map
\[
\left(\lambda_{i}\right)_{i\in \left[n\right]}\mapsto s_0 + \sum \lambda_{i} s_{i}.
\]
If  the multisection
\[
F^{ev}\in C^\infty_m(V\times N, p_N^*\E), \qquad
F^{ev}\left(v,x\right)= F_v\left(x\right),
\]
is transverse to $0$, then the set
\[
\left\{\left.v\in V \right| F_v \pitchfork 0\right\}
\]
is residual.
\end{thm}
\begin{proof}
Take $p\in N.$ There exists a neighborhood $W$ of $p$ such that each multisection $s_i|_W$ is a weighted combination of $m_i$ sections. Hence $F^{ev}|_{V\times W}$ is a weighted combination of appropriately defined sections $F^{ev}_{W,j}$ for $j = 1,\ldots,\prod_{i = 1}^n m_i.$ Apply Theorem~\ref{tm:hirschnm} to each section $F_{W,j}$ individually to conclude that the set
\[
U_W = \bigcap_j\left\{\left.v\in V \right| F^{ev}_{W,j}(v,-)\pitchfork 0\right\}
\]
is residual. Choose a countable open cover $\{W_l\}$ of $N$. Then for every
\[
v \in U = \bigcap_l U_{W_l}
\]
we have $F_v \pitchfork 0.$ Moreover, $U$ is residual. The theorem follows.
\end{proof}

\begin{lm}\label{lm:tscm}
Fix a sequence of non-negative integers
\[
a_1,\ldots,a_l, \qquad  2\sum a_i = k+2l-3,
\]
and set $E = \bigoplus_{i=1}^l \CL_i^{\oplus a_i}\to \oCM_{0,k,l}.$
\begin{enumerate}
\item \label{it:basic}
One can construct special canonical multisections
\[
s_{ij}\in\CS_i,\qquad i \in [l],\quad j \in [a_i],
\]
such that $\mathbf{s} = \oplus s_{ij}$ vanishes nowhere on $\partial\oCM_{0,k,l}.$ Hence,
$e\left(E ; \mathbf{s}\right)$ is defined.
\item\label{it:better_trans}
Moreover, we may impose the following further condition on the multisections $s_{ij}.$ For all abstract vertices $v\in\mathcal{V}_{\pB\Gamma_{0,k,l}}$, and all
\[
K\subseteq \bigcup_{i\in I(v)}\left\{i\right\}\times[a_{ij}],
\]
we have
\[
\bigoplus_{ab\in K} s_{ab}^v\pitchfork 0.
\]
\end{enumerate}
\end{lm}
\begin{proof}
We begin with the proof of part~\ref{it:basic}.
Let
\[
w_{ijk} \in \CS_i, \qquad\qquad  i\in [l],\quad j\in [a_i],\quad k \in [m_{ij}],
\]
be a finite collection of special canonical multisections of the $j^{th}$ copy of $\CL_i,$ which together span its fiber $\left(\CL_i\right)_p$ for all $p\in\partial\oCM_{0,k,l}$. Such multisections exist by Proposition~\ref{prop:single_section}, case \ref{it:single_section_a}. We write
\[
J=\left\{ijk\right\}_{i\in [l],j\in [a_i],k\in [m_{ij}]}.
\]

Apply Theorem~\ref{thm: hirsch} with
\[
N = \partial\oCM_{0,k,l}, \qquad  \E = \left.E\right|_N, \qquad V = V_0 = \R^J,
\]
and $F$ given by
\[
F_\lambda = \sum_{ijk \in J} \lambda_{ijk} w_{ijk}, \qquad \lambda =\{\lambda_{ijk}\}_{ijk \in J} \in V_0.
\]
Let $\Lambda_0$ be the set of $\lambda \in V$ such that $F_\lambda \pitchfork 0.$ Theorem~\ref{thm: hirsch} implies that $\Lambda_0$ is residual. Dimension counting shows that for each $\lambda \in \Lambda_0,$ we have $F_\lambda^{-1}(0) = \emptyset.$ Thus for any $\lambda \in \Lambda_0$, we may take
\begin{equation}\label{eq:sijla}
s_{ij} = s_{ij}^\lambda = \sum_k \lambda_{ijk} w_{ijk}.
\end{equation}

We turn to the proof of part~\ref{it:better_trans}. For an abstract vertex $v\in\mathcal{V}_{\Gamma_{0,k,l}},$ and a set $K$ as in the statement of the lemma, write
\[
J_{v,K} = \left\{\left.abc \right| ab\in K, ~c\in [m_{ab}]\right\}\subseteq J.
\]
Apply Theorem~\ref{thm: hirsch} with
\[
N = \CM_v, \qquad \E = \bigoplus_{\left\{ab\in K\right\}}\CL_a, \qquad V = V_{v,K} = \R^{J_{v,K}},
\]
and $F = F_{v,K}$ given by
\[
\left(F_{v,K}\right)_{\lambda} = \sum_{ijk \in J_{v,K}} \lambda_{ijk} w^v_{ijk}, \qquad \lambda =\{\lambda_{ijk}\}_{ijk \in J_{v,K}} \in V_{v,K}.
\]
Let
\[
\Lambda_{v,K} = \{ \lambda \in V_{v,K}| (F_{v,K})_\lambda \pitchfork 0\}.
\]
Theorem~\ref{thm: hirsch} implies that $\Lambda_{v,K}$ is residual. Denote by $p_{v,K} : V_0 \to V_{v,K}$ the projection. It follows that
\[
\Lambda = \Lambda_0 \cap \bigcap_{v,K} p_{v,K}^{-1}\left(\Lambda_{v,K}\right)
\]
is residual.

For any $\lambda \in \Lambda$, take $s_{ij} = s_{ij}^\lambda$ as in equation~\eqref{eq:sijla}.
Then for any abstract vertex $v$ and set~$K$, we have
\[
\bigoplus_{ab\in K} s_{ab}^v = (F_{v,K})_\lambda \pitchfork 0,
\]
as desired.
\end{proof}

\begin{lemma}\label{lem:trickey_homotopy}
Let $E_1, E_2 \to \oCM_{0,k,l}$ be given by
\[
E_1 = \bigoplus_{i\in \left[l\right]}\CL_i^{a_i},\qquad E_2 = \bigoplus_{i\in \left[l\right]}\CL_i^{b_i}.
\]
Put $E = E_1 \oplus E_2,$ and assume $\rk{E} = \frac{k+2l-3}{2}.$ Let $\CC \subseteq \pB\Gamma_{0,k,l}$ and
\[
C = \coprod_{\Gamma\in \CC} \CM_\Gamma \subseteq\partial\oCM_{0,k,l}.
\]
Let $\mathbf{s}, \mathbf{r},$ be two multisections of $\left.E\right|_{\partial\oCM_{0,k,l}}$ which satisfy
\begin{enumerate}
\item\label{it:ass_0}
$\left.\mathbf{s}\right|_C$ and $\left.\mathbf{r}\right|_C$ are canonical.
\item\label{it:ass_1}
The projections of $\mathbf{s}, \mathbf{r},$ to $E_1$ are identical and transverse to $0.$
\end{enumerate}
Then one may find a homotopy $H$ between $\mathbf{s}, \mathbf{r},$ which is transverse to $0$ everywhere, does not vanish anywhere on $C\times\left[0,1\right]$ and such that its projection to $E_1$ is constant in time. Moreover, $H$ can be taken to be of the form
\begin{equation}\label{eq:srw}
H(p,t) = (1-t) \mathbf{s}(p) + t \mathbf{r}(p) + t(1-t) w(p),
\end{equation}
where $w(p)$ is a canonical multisection.
\end{lemma}

\begin{proof}
Denote by $\mathbf{s}_1$ the projection of $\mathbf{s}$ to $E_1.$
Let 
\[
w_{i}, \qquad   i\in [m],
\]
be a finite set of special canonical multisections which together span the fiber $\left(E_2\right)_p$ for all $p\in\partial\oCM_{0,k,l}$. Such multisections exist by Proposition~\ref{prop:single_section}, case \ref{it:single_section_a}.
Denote by $\pi : \partial\oCM_{0,k,l}\times [0,1] \to \partial\oCM_{0,k,l}$ the canonical projection. Let $\mathbf{h} \in C^\infty_m(\left.\pi^*E\right|_{\partial\oCM_{0,k,l}})$ be given by
\[
\mathbf{h}\left(p,t\right) = \left(1-t\right)\mathbf{s}\left(p\right)+t\mathbf{r}\left(p\right)\qquad p \in \partial\oCM_{0,k,l},\quad t\in \left(0,1\right).
\]
Apply Theorem~\ref{thm: hirsch} with
\[
N = \partial\oCM_{0,k,l}\times \left(0,1\right),\qquad \E = \left.\pi^*E\right|_{\partial\oCM_{0,k,l}}, \qquad V = V_0= \R^m,
\]
and $F = \CF$ given by
\[
\CF_\lambda(p,t) = \mathbf{h}\left(p,t\right) + t\left(1-t\right)\sum \lambda_i w_i, \qquad \lambda \in V_0.
\]
By assumption \ref{it:ass_1}, the derivatives of $\CF^{ev}$ in directions tangent to $\partial\oCM_{0,k,l}$ span the fiber $\left(E_1\right)_p$ at each $p$ where $\mathbf{s}_1$ vanishes. Since the multisections $w_i$ span $(E_2)_p$, the derivatives of $\CF^{ev}$ in the $V$ directions span the fiber $\left(E_2\right)_p$ for all $p \in \partial\oCM_{0,k,l}$. It follows that $\CF^{ev} \pitchfork 0.$ Thus, Theorem \ref{thm: hirsch} implies the set $\Lambda$ of all $\lambda\in V_0$ such that $\CF_\lambda \pitchfork 0$ is residual.

Let $\Gamma \in \CC.$ Denoting by $E_\Gamma \to \CM_{\CB\Gamma}$ the appropriate sum of cotangent line bundles, Observation~\ref{obs:identification_bundle_over_base} implies that
$F_\Gamma^* E_\Gamma = E.$
Write
\[
N_\Gamma = \CM_{\CB\Gamma} \times (0,1)
\]
and denote by $\pi_\Gamma : N_\Gamma \to \CM_{\CB\Gamma}$ the canonical projection. Write
\[
\E_\Gamma = \pi_\Gamma^*E_\Gamma.
\]
It follows from Observation~\ref{obs:extension_of_canonical_conds} and Remark~\ref{rmk:scc} that there exists $\CF^{\CB\Gamma} : V_0 \to C^\infty_m(N_\Gamma,\E_\Gamma)$ such that
\[
\CF_\lambda|_{\CM_\Gamma \times (0,1)} = (F_\Gamma \times \id_{(0,1)})^* \CF^{\CB\Gamma}_\lambda, \qquad \lambda \in V_0.
\]
Apply Theorem~\ref{thm: hirsch} with
\[
N = N_\Gamma, \qquad \E = \E_\Gamma, \qquad V = V_0, \qquad F = \CF^{\CB\Gamma}.
\]
Since $\mathbf{s}_1 \pitchfork 0,$ it follows that $\mathbf{s}_1^{\CB\Gamma}\pitchfork 0.$ Thus the same argument that shows $\CF^{ev}\pitchfork 0$ also shows $(\CF^{\CB\Gamma})^{ev}\pitchfork 0.$ So, the theorem implies the set $\Lambda_\Gamma$ of $\lambda \in V_0$ such that $\CF^{\CB\Gamma}_\lambda \pitchfork 0$ is residual. By Observation~\ref{obs:dim_of_base_mod}, for $\lambda \in \Lambda_\Gamma,$ the homotopy $\CF^{\CB\Gamma}_\lambda$ does not vanish anywhere. Therefore, the homotopy $\CF_\lambda|_{\CM_\Gamma \times (0,1)}$ also does not vanish anywhere. We conclude that for
\[
\lambda \in \Lambda \cap \bigcap_{\Gamma \in \CC} \Lambda_\Gamma,
\]
the homotopy $\CF_\lambda$ satisfies the requirements of the lemma.
\end{proof}

We will also need the following general lemma on the relative Euler class. For a multisection $s$ that is transverse to zero, we denote by $Z(s)$ its vanishing locus considered as a weighted branched submanifold. For a zero dimensional weighted branched submanifold $Z \subset M,$ we denote by $\# Z$ its weighted cardinality. See Appendix~\ref{app:euler} for details.
\begin{lm}\label{lm:H}
Let $E \to M$ be a vector bundle over a manifold with corners with $\rk E = \dim M$, and let $s_0,s_1 \in C_m^\infty(\partial M,E)$ vanish nowhere. Let $p : [0,1] \times M\to M$ denote the projection and let
\[
H \in C_m^\infty([0,1]\times\partial M,p_1^*E)
\]
satisfy
\[
H|_{\{i\}\times M } = s_i, \qquad i = 0,1.
\]
Moreover, assume $H$ is transverse to zero. Then
\[
\int_M e(E;s_1) - \int_M e(E;s_0) = \# Z(H).
\]
\end{lm}

\begin{proof}
For $i = 0,1,$ let $\tilde s_i \in C_m^\infty(M, E)$ be an extension of $s_i$ that is transverse to zero.
Recall that
\[
\partial([0,1]\times M ) = \{1\} \times M - \{0\}\times M - [0,1]\times\partial M.
\]
So, the multisections $\tilde s_0,\tilde s_1,H,$ fit together to give a multisection
\[
r \in C_m^\infty(\partial ([0,1]\times M ) , p_1^*E)
\]
that is transverse to zero. Let $\tilde r \in C_m^\infty([0,1]\times M,p_1^*E)$ be an extension of $r$ that is transverse to zero. Then $Z(\tilde r)$ is a weighted branched $1$-manifold with boundary. The weighted cardinality of the boundary points of such a weighted branched manifold is zero. Thus
\begin{align*}
0 = \# \partial Z(\tilde r) &= \# Z(r) = \# Z(\tilde s_1) - \# Z(\tilde s_0) - \# Z(H) \\
& = \int_M e(E;s_1) - \int_M e(E;s_0) - \# Z(H).
\end{align*}
\end{proof}

\begin{proof}[Proof of Theorem \ref{thm:intersection_numbers_well_defined}]
By Lemma~\ref{lm:tscm}\ref{it:basic} and Remark~\ref{rmk:scc} there exists a nowhere vanishing canonical multisection $\mathbf{s} \in \CS.$ It remains to show that $e(E,\mathbf{s})$ is independent of the choice of $\mathbf{s}.$  By Lemma~\ref{lm:H} it suffices to construct a nowhere vanishing homotopy between any two canonical multisections $\mathbf{s},\mathbf{r},$ that each vanish nowhere. But the existence of such a homotopy is a direct consequence of Lemma~\ref{lem:trickey_homotopy}, with the bundle $E_1 = 0,$ and the collection of boundary strata $C$ being the entire boundary $\partial\oCM_{0,k,l}.$
\end{proof}

We now consider slightly more general bundles, which we shall need later on.
\begin{lemma}\label{lem:main_lem}
Let $1\leq h\leq l$.
Let
\[
E \to \oCM_{0,k,l+1}
\]
be given by $E = E_1 \oplus E_2 \oplus E_3$ where
\[
E_1 = \bigoplus_{i=1}^{l+1}{\CL}_i^{\oplus a_i}, \qquad E_2 = \bigoplus_{i=1}^{l}\left(\CL'_i\right)^{\oplus a'_i}, \qquad E_3 = \CO\left(D_h\right)^{\oplus\varepsilon},
\]
and
\[
\varepsilon \in \left\{0,1\right\} , \qquad \left(a_1+\ldots+a_{l+1}\right)+\left(a'_1+\ldots a'_l\right)+\varepsilon = \frac{k + 2l - 1}{2}.
\]
One can construct
\begin{align*}
&s_{ij}\in\CS_i,\qquad i \in [l+1],\quad j \in [a_i], \\
&s'_{ij}\in\CS'_i, \qquad i \in [l],\quad j \in [a_i'],
\end{align*}
such that
\[
\mathbf{s} = \bigoplus s_{ij} \oplus \bigoplus s_{ij}' \oplus t_h^{\oplus \varepsilon}
\]
does not vanish anywhere. In particular, the relative Euler class
$e(E;\mathbf{s})$
is defined. Moreover, any two choices of such $s_{ij} , s'_{ij},$ define the same relative Euler class.
Furthermore, the following statements are valid simultaneously:
\begin{enumerate}
\item\label{it:D_i}
Suppose $1 \leq i_0 \leq l$ and $1 \leq j_0 \leq a_{i_0}$. If $\varepsilon = 1,$ suppose that $i_0 \neq h.$
Then we may assume
\[
s_{i_0j_0} = s't_{i_0},\qquad s'\in\CS'_{i_0}.
\]
and $s_{i_0j}$ does not vanish anywhere on $\partial D_{i_0}$ for $j \neq j_0.$
\item\label{it:b}
Suppose $a_{l+1}>0$. Then we may assume $s_{(l+1)1}$ does not vanish anywhere on $\partial D_{i}$ for all $i.$
\item\label{it:D_i_2}
Suppose $\rk (E_1 \oplus E_3) = 1$. Then we may assume $\bigoplus_{i = 1}^{l} \bigoplus_{j = 1}^{a'_i} s'_{ij}$ does not vanish anywhere on $\partial \oCM_{0,k,l+1}.$
\end{enumerate}
\end{lemma}
\begin{proof}
The proof is very similar to that of Lemma~\ref{lm:tscm} and Lemma~\ref{lem:trickey_homotopy}.

First, we prove cases~\ref{it:D_i} and~\ref{it:b}.
Using Proposition~\ref{prop:single_section} case~\ref{it:single_section_a}, choose
\[
w_{ijk}\in\CS_{i}, \qquad i\in [l+1],\quad j\in [a_i],\quad {k\in [m_{ij}]},\quad (i,j) \neq (i_0,j_0),
\]
such that for each $i,j,$ the multisections $w_{ijk}$ for $k \in [m_{ij}]$ span the fiber $\left(\CL_i\right)_p$ for all $p\in\partial\oCM_{0,k,l+1}$. Choose
\[
w'_k \in \CS'_{i_0}, \qquad w_{i_0j_0k} = w'_k t_{i_0} \in \CS_{i_0},\qquad k\in [m_{i_0j_0}],
\]
as in Proposition \ref{prop:single_section}, case \ref{it:single_section_b}, that span the fiber $\left(\CL_{i_0}\right)_p$ for all $p$ not in $D_{i_0},$ and such that the images of their derivatives at every $q\in D_{i_0}$ span $\left(\CL_{i_0}\right)_q.$
Using Proposition~\ref{prop:single_section} case~\ref{it:single_section_a} over $\oCM_{0,k,l}$ and pulling back by $For_{l+1},$ choose
\[
w'_{ijk}\in\CS'_i, \qquad i\in [l+1],\quad j\in [a'_i],\quad {k\in [m'_{ij}]},
\]
such that for each $i,j,$ the multisections $w'_{ijk}$ for $k \in [m'_{ij}]$ span the fiber $\left(\CL'_i\right)_p$ for all $p\in\partial\oCM_{0,k,l+1}$.

Write
\[
J=\left\{ijk\right\}_{i\in [l+1],j\in [a_i],k\in [m_{ij}]}, \qquad J'=\left\{ijk\right\}_{i\in [l],j\in [a'_i],k\in [m'_{ij}]}.
\]
Apply Theorem~\ref{thm: hirsch} with
\[
N = \partial\oCM_{0,k,l+1}, \qquad \E = E, \qquad V = \R^{J \cup J'},
\]
and $F$ given by
\[
F_\lambda = \sum_{ijk \in J} \lambda_{ijk} w_{ijk}+\sum_{ijk \in J'}\lambda'_{ijk} w'_{ijk} + \delta_{\varepsilon,1} t_h,
\]
for
\[
\lambda = (\{\lambda_{ijk}\}_{ijk \in J},\{\lambda'_{ijk}\}_{ijk \in J'}) \in V.
\]
We claim that $F^{ev} \pitchfork 0.$ Indeed, if $p \in \oCM_{0,k,l+1}\setminus (D_{i_0} \cup D_h)$ then the derivatives of $F^{ev}$ in the directions tangent to $V$ span the fiber $E_p$. If $p \in D_{i_0}$, then $p \notin D_h$. So, the derivatives of $F^{ev}$ in the directions tangent to $\partial\oCM_{0,k,l+1}$ span the fiber of the $j_0^{th}$ copy of $L_{i_0}$ at $p,$ while the derivatives in the directions tangent to $V$ span the complementary summand of the fiber $E_p.$ If $p \in D_h$, then $p \notin D_{i_0}$. So, the derivatives of $F^{ev}$ in the directions tangent to $\partial\oCM_{0,k,l+1}$ span the fiber $\CO\left(D_h\right)^{\oplus\varepsilon}_p$, while the derivatives in the directions tangent to $V$ span the complementary summand of the fiber $E_p.$

Theorem~\ref{thm: hirsch} implies there exists a residual subset $\Lambda \subset V$ such that if $\lambda \in \Lambda$ then $F_\lambda\pitchfork 0.$ By dimension counting, transversality is equivalent to non-vanishing.

Write $v_i$ for the closed abstract vertex with
\[
I(v_i) = \{i,l+1,[l]\setminus\{i\}\}.
\]
So, $v_i = \bv{i}{\Gamma}$ for all $\Gamma \in \CGDi.$ Let
\[
\Lambda' = \left\{\lambda \in V\left|
\begin{array}{ll}
\sum_k \lambda_{i_0jk}w^{v_{i_0}}_{i_0jk}\neq 0, & j\neq j_0 \\
\sum_k \lambda_{(l+1)jk} w^{v_{i}}_{{l+1}jk} \neq 0, & 1 \leq i \leq l \end{array}
\right.\right\}.
\]
Since $\CM_{v_i}$ is a point, $\Lambda'$ is the complement of a finite union of linear subspaces $U_{j},W_{i} \subseteq V,$ one for each inequality. By choice of the sections $w_{ijk},$ for each $j \geq 2,$ there is a $k \in [m_{i_0j}]$ such that $w_{i_0jk}^{v_{i_0}} \neq 0.$ So $U_j$ is a proper subspace for $j~\geq~2.$ Similarly, for each $i \in [l],$ there is a $k \in [m_{(l+1)1}]$ such that $w_{(l+1)jk}^{v_i} \neq 0.$ So $W_i$ is a proper subspace for $i \in [l].$ It follows that $\Lambda'$ is open and dense in $V.$ Thus we may choose $\lambda \in \Lambda \cap \Lambda'$ and set
\[
s_{ij} = \sum_k \lambda_{ijk} w_{ijk}, \qquad s_{ij}' = \sum_k \lambda'_{ijk} w'_{ijk}, \qquad s' = \sum_k \lambda_{i_0j_0k} w'_k.
\]
This proves cases~\ref{it:D_i} and~\ref{it:b}.

Case~\ref{it:D_i_2} follows from a similar argument and the fact that the multisections $s_{ij}'$ are pulled back from $\partial\oCM_{0,k,l}$, which has complex dimension one less. So, transversality implies non-vanishing even with one less section.

Using Remark~\ref{rmk:L'}, the proof of the existence of non-vanishing homotopies in the present case is analogous to the proof of Lemma~\ref{lem:trickey_homotopy}.
\end{proof}

\section{Geometric recursions}\label{sec:geo_rec}
\subsection{Proof of string equation}
Recall Notations~\ref{nn:D_i1} and~\ref{nn:OD}.
\begin{obs}\label{obs:1}
$D_i \cap D_j =\emptyset$ for $i\neq j$. An immediate consequence is the following. Let
$E \to \oCM_{0,k,l+1}$
be a bundle containing $\mathcal{O}\left(D_i\right)\oplus\mathcal{O}\left(D_j\right)$ as a summand. Let $\mathbf{s} \in C_m^\infty(\partial\oCM_{0,k,l+1},E)$ be a nowhere vanishing multisection that upon projection to $\mathcal{O}\left(D_i\right)\oplus\mathcal{O}\left(D_j\right)$ agrees with $t_i \oplus t_j.$ Then
\[
e\left(E;\mathbf{s}\right) = 0.
\]
\end{obs}
\begin{obs}\label{obs:2}
Let $s_i$ be a special canonical multisection of $\CL_i\to\partial\oCM_{0,k,l+1}$ that does not vanish on $\partial D_i.$
Let $E \to \oCM_{0,k,l+1}$ be a bundle that contains $\mathcal{O}\left(D_i\right)\oplus\CL_i$ as a summand. Let $\mathbf{s} \in C_m^\infty\left(\partial\oCM_{0,k,l+1},E\right)$ be a nowhere vanishing multisection that upon projection to $\mathcal{O}\left(D_i\right)\oplus\CL_i,$ agrees with $s_i\oplus t_i.$
Then
\[
e\left(E;\mathbf{s}\right) = 0.
\]
The same holds if we replace $s_i, \CL_i,$ everywhere with $s_{l+1},\CL_{l+1},$ respectively.
\end{obs}
\begin{proof}
Let $\Gamma_i \in \partial \Gamma_{0,k,l+1}$ be the stable graph such that $D_i = \oCM_{\Gamma_i}$, and let $v_i$ be the abstract closed vertex with \[
I(v_i) = \{i,l+1,[l]\setminus\{i\}\}.
\]
So $v_i = \bv{i}{\Gamma_i} = \bv{i}{\Gamma}$ for $\Gamma \in \partial \Gamma_i.$
By the definition of a special canonical multisection, there exists a multisection $s^{v_i}$ of $\CL_i \to \CM_{v_i}$ such that for each $\Gamma \in \CG_{D_i}$ we have $s_i^\Gamma = \Phi_{\Gamma,i}^*s^{v_i}.$ So, we may extend $s_i$ to a multisection $\tilde s_i \in C_m^\infty\left(\oCM_{0,k,l},\CL_i\right)$ such that $\tilde s_i|_{\oCM_{\Gamma_i}} = \Phi_{{\Gamma_i},i}^*s^v$. Since $s_i$ does not vanish anywhere on $\partial D_i,$ it follows that $s_i^{v_i}$ does not vanish and thus $\tilde s_i$ does not vanish anywhere on $D_i.$ Therefore, $Z(\tilde s_i) \cap Z(\tilde t_i) = \emptyset,$ which implies the Euler class of $E$ vanishes. The same argument works for the case of $s_{l+1},\CL_{l+1}.$
\end{proof}
\begin{lm}\label{lm:prod}
Let $E \to X$ be a vector bundle over a manifold with corners with $\rk E = \dim X.$ Suppose that $E = L \oplus E',$ where $L\to X$ is a line bundle, and $L = L_1\otimes L_2$ for line bundles $L_1,L_2 \to X.$ Let $\mathbf{s} \in C^\infty_m(\partial X,E)$ vanish nowhere and satisfy $\mathbf{s} = s \oplus \mathbf{s}'$, where $s \in C^\infty_m(\partial X,L)$, and $s = s_1 \otimes s_2$ for $s_1 \in C^\infty(\partial X,L_1)$ and $s_2 \in C^\infty_m(\partial X,L_2).$ Then
\[
e(E;\mathbf{s}) = e(L_1 \oplus E';s_1\oplus \mathbf{s}') + e(L_2 \oplus E';s_2 \oplus \mathbf{s}').
\]
Since the multisection $\mathbf{s}$ vanishes nowhere, the multisections $s_i\oplus \mathbf{s}'$ for $i = 1,2$ also vanish nowhere. Thus the relative Euler classes on the right-hand side are well-defined.
\end{lm}
\begin{proof}
Let $\tilde s_1,\tilde s_2$ and $\tilde \s',$ be extensions to $X$ of $s_1,s_2$ and $\s'$ respectively, such that
\[
\tilde s_i \oplus \s' \pitchfork 0, \quad i = 1,2, \qquad \tilde s_1 \oplus \tilde s_2 \oplus \tilde \s' \pitchfork 0.
\]
By assumption $\rk L_1 \oplus L_2 \oplus E' > \dim X,$ so $\tilde s_1 \oplus \tilde s_2 \oplus \tilde \s'$ vanishes nowhere. Therefore,
\[
Z(\tilde s_1 \oplus \tilde\s') \cap Z(\tilde s_2 \oplus \tilde \s') = \emptyset.
\]
Setting $\tilde s = \tilde s_1 \tilde s_2,$ it follows that $\tilde s \oplus \tilde \s'$ is transverse to zero. Thus
\[
Z(\tilde s \oplus \tilde \s') = Z(\tilde s_1 \oplus \tilde \s') \cup Z(\tilde s_2 \oplus \tilde \s'),
\]
which implies the claim.
\end{proof}
\begin{rmk}\label{rmk:prb}
In the proof of the preceding Lemma, we cannot make $\tilde s$ by itself transverse to zero at any point where both $\tilde s_1$ and $\tilde s_2$ vanish. Such points are unavoidable in general, but generically they do not intersect $Z(\tilde \s').$
\end{rmk}

In the following, given pairs
\[
\CE_i = (E_i,\mathcal{Y}_i), \qquad i = 1,2,
\]
of vector bundles $E_i \to X$ and affine subspaces $\mathcal{Y}_i \subset C_m^\infty(E_i|_{\partial X}),$ we write
\[
\CE_1 \oplus \CE_2 = (E_1 \oplus E_2, \mathcal{Y}_1 \oplus \mathcal{Y}_2).
\]
For
\begin{gather*}
\ba = (a_1,\ldots,a_{l+1}) \in \Z_{\geq 0}^{l+1}, \\
\bb = (b_1,\ldots,b_{l+1})\in \Z_{\geq 0}^{l+1}, \qquad \bc = (c_1,\ldots,c_{l},0) \in \Z_{\geq 0}^{l+1},\\
\bb + \bc = \ba,
\end{gather*}
write
\[
E_{\bb,\bc} = \bigoplus_{i = 1}^{l+1} \CL_i^{\oplus b_i} \oplus \bigoplus_{i = 1}^l (\CL_i')^{\oplus c_i} \to \oCM_{0,k,l+1}.
\]
Let
\[
\CS_{\bb,\bc} = \bigoplus_{i= 1}^{l+1} \CS_i^{\oplus b_i} \oplus \bigoplus_{i = 1}^l (\CS'_i)^{\oplus c_i} \subset C_m^\infty\left(E_{\bb,\bc}|_{\partial\oCM_{0,k,l+1}}\right).
\]
and
\[
\CE_{\bb,\bc} = (E_{\bb,\bc}, \CS_{\bb,\bc}).
\]
Let $\be_i \in \Z^{l+1}_{\geq 0}$ be the vector with $1$ for its $i^{th}$ coordinate and $0$ for the others. Let $\bze \in \Z^{l+1}$ denote the zero vector and abbreviate
\[
E_\ba = E_{\ba,\bze}, \qquad \CS_\ba = \CS_{\ba,\bze}, \qquad \CE_{\ba} = \CE_{\ba,\bze}.
\]
Let
\[
\hat \ba = (a_1,\ldots,a_l) \in \Z_{\geq 0}^l.
\]
Thus $\CE_{\hat \ba}$ is a vector bundle over $\oCM_{0,k,l}.$ Finally, let
\[
\CODt{i} = (\CO(D_i),t_i).
\]

Write $|\ba| = \sum_i a_i.$ For $|\ba|  = k + 2l - 1,$ Lemma~\ref{lem:main_lem} shows that there exists $\s \in \CS_{\bb,\bc}$ that vanishes nowhere, so the relative Euler class $e(E_{\bb,\bc};\s)$ is defined. Furthermore, the same lemma shows that $e(E_{\bb,\bc};\s)$ is independent of the choice of such $\s.$ So we define
\[
e(\CE_{\bb,\bc}) = e(E_{\bb,\bc};\s).
\]
Similarly, for $|\ba| = k + 2l- 3,$ Lemma~\ref{lem:main_lem} allows us to define
\[
e(\CE_{\bb,\bc}\oplus \CODt{i}) = e(E_{\bb,\bc}\oplus \CO(D_i);\s \oplus t_i),
\]
for some $\s \in \CS_{\bb,\bc}$ such that $\s \oplus t_i$ vanishes nowhere.

\begin{proof}[Proof of Theorem \ref{thm:string_dilaton}, string equation]
We start with the open string equation.
Consider the intersection number
\[
\left\langle \tau_0\prod_{i = 1}^l \tau_{a_i}\sigma^k\right\rangle_{0}^o = 2^{-\frac{k-1}{2}}\int_{\oCM_{0,k,l+1}}e\left(\CE_\ba\right),
\]
where $\ba = (a_1,\ldots,a_l,0),$ and $|a| = k + 2l-3.$

Let $\bb,\bc \in \Z_{\geq 0}^{l+1}$ satisfy $\bb + \bc = \ba.$
For $q \in [l]$ such that $b_q \geq 1,$ Lemma~\ref{lm:prod} and case~\ref{it:D_i} of Lemma~\ref{lem:main_lem} with $i_0 = q$ and $j_0$ arbitrary, imply that
\[
e(\CE_{\bb,\bc}) = e(\CE_{\bb-\be_q,\bc + \be_q}) + e(\CE_{\bb-\be_q,\bc} \oplus \CODt{q}).
\]
For $b_q \geq 2,$ Observation~\ref{obs:2} implies the second summand vanishes. Similarly, Lemma~\ref{lem:main_lem}\ref{it:D_i}, Lemma~\ref{lm:prod} and Observation~\ref{obs:1}, imply that for $q\neq r \in [l],$
\[
e(\CE_{\bb,\bc} \oplus \CODt{q}) =  e(\CE_{\bb-\be_r,\bc + \be_r} \oplus \CODt{q}).
\]
By induction,
\begin{equation}\label{eq:eEa}
e(\CE_\ba) = e(\CE_{\bze,\ba}) + \sum_{\substack{q \in [l],\\ a_q \geq 1}} e(\CE_{\bze,\ba-\be_q}\oplus \CODt{q}).
\end{equation}

By definition,
\[
e(\CE_{\bze,\ba}) = PD[Z(\tilde \s)],
\]
where $\tilde \s$ is any transverse extension of a nowhere vanishing section $\s \in \CS_{\bze,\ba}.$ By definition of $\CS_{\bze,\ba},$ there exists $\hat \s \in \CS_{\hat \ba}$ such that $\s = For_{l+1}^*\hat \s.$ Since $\rk E_{\hat \ba} > \dim \oCM_{0,k,l},$ we can choose a nowhere vanishing extension $\bar \s$ of $\hat \s$ by transversality. Taking $\tilde \s = For_{l+1}^*\bar \s,$ we obtain
\begin{equation}\label{eq:0a}
e(\CE_{\bze,\ba}) = 0.
\end{equation}

Similarly,
\[
e(\CE_{\bze,\ba-\be_q}\oplus \CODt{q}) = PD [Z(\tilde t_i) \cap Z(\tilde \s)],
\]
where $\tilde \s$ is any transverse extension of a nowhere vanishing section $\s \in \CS_{\bze,\ba-\be_q}.$ Such $\s$ exists by Lemma~\ref{lem:main_lem}\ref{it:D_i_2}. Let $\hat \s \in \CS_{\hat \ba  - \hat \be_q}$ such that $\s = For_{l+1}^*\hat \s.$ Denote by $\bar \s$ a transversal extension of $\hat \s$ and choose $\tilde \s = For_{l+1}^*\bar \s.$ Using Lemma~\ref{lm:fmor}\ref{it:m1}, we obtain
\begin{multline}\label{eq:aq}
\int_{\oCM_{0,k,l+1}}e(\CE_{\bze,\ba-\be_q}\oplus \CODt{q}) = \# Z(\tilde t_i) \cap Z(\tilde \s) = \\
 = \# Z(\hat \s)
 = \int_{\oCM_{0,k,l}} e\left(\CE_{\hat\ba-\hat \be_q}\right) = 2^{\frac{k-1}{2}}\left\langle \tau_{a_q-1} \prod_{i\neq q} \tau_{a_i}\sigma^k\right\rangle_{0}^o.
\end{multline}
Equations~\eqref{eq:eEa},~\eqref{eq:0a},~\eqref{eq:aq} together imply the open string equation.
\end{proof}

\subsection{Proof of dilaton equation}
We continue with the notations of the previous section. The following lemma is the key additional ingredient in the proof of the dilaton equation. In the case $k = 3$ and $l = 0,$ the proof of the following lemma calculates the integral $\langle \tau_1 \sigma^3 \rangle$ directly from the definition.
\begin{lm}\label{lm:cool}
Let $p \in \CM_{0,k,l}$ and $F_p = For_{l+1}^{-1}(p)$ equipped with its complex orientation. Let $s$ be a nowhere vanishing special canonical multisection of $\CL_i|_{\partial F_p}$ and let $\tilde s$ be an extension of $s$ to $F_p$ that is transverse to zero. Then
\[
\# Z(\tilde s) = k + l - 1.
\]
\end{lm}
\begin{proof}
The section $s$ is determined by its value at a single point. Indeed, on each stratum of $\partial F_p,$ the section $s$ is pulled back from a zero dimensional moduli space. In particular, $s$ can only vanish at a given point if it vanishes identically.

It follows that if $s'$ is another section satisfying the same hypotheses as $s,$ then $s$ and $s'$ can be connected by a non-vanishing homotopy. Indeed, the complement of zero in a single fiber of $\CL_{l+1}$ is connected. Thus $\# Z(\tilde s)$ is independent of the choice of $s$ by Lemma~\ref{lm:H}.

To begin, we reduce the calculation to the case $l = 0.$ Applying Lemma~\ref{lm:fmor}\ref{it:tr} with interior labels $l$ and $l+1$ switched, we obtain a canonical map of line bundles
\[
\tilde t: For_l^* \CL_{l+1} \to \CL_{l+1},
\]
which vanishes transversally exactly at $D_l.$ Write $t = \tilde t|_{\partial\oCM_{0,k,l+1}}.$

Let $\hat p  = For_l(p)$ and let $F_{\hat p} = For_{l+1}^{-1}(p) \subset \oCM_{0,k,[l+1]\setminus \{l\}}.$ Let $\hat s$ be a special canonical multisection of $\CL_{l+1} \to \CM_{0,k,[l+1]\setminus \{l\}}$ that vanishes nowhere on $\partial F_{\hat p}.$ By Observation~\ref{obs:D_i} with interior labels $l$ and $l+1$ switched, we conclude that $(For_l^* \hat s) t \in \CS_{l+1}$ and vanishes nowhere on $\partial F_p.$ Thus we may take $s = (For_l^*\hat s) t.$ Let $q \in F_p$ be the unique point in $D_l \cap F_p$ and let $\hat q = For_l(q) \in F_{\hat p}.$ Let $\bar s$ be a transverse extension of $\hat s$ that does not vanish at $\hat q$. Then we may take $\tilde s = (For_l^* \bar s) \tilde t$ and
\[
\# Z(\tilde s|_{F_p})  =  \# Z(For_l^*\bar s|_{F_{p}}) + \#Z(\tilde t|_{F_{p}}) = \# Z(\bar s|_{F_{\hat p}}) + 1.
\]
In the last equality, we have used the fact that $For_l$ maps $F_p$ diffeomorphically to $F_{\hat p}$ as well as Lemma~\ref{lm:fmor}\ref{it:m1}. By induction, appropriately relabelling interior marked points, it suffices to prove the lemma when $l = 0.$

The case $l = 0,\,k = 1,$ is exceptional. In this case we take $F_p = \CM_{0,1,1},$ which is a point, and the claim is trivial. Below, we assume $l = 0$ and $k \geq 3.$

Let $(\Sigma,\mathbf x,\emptyset)$, where $\mathbf x = \{x_1,\ldots,x_k\},$ be a marked surface representing $p\in \CM_{0,k,0}.$ Then $F_p$ is diffeomorphic to the oriented real blowup $\widetilde \Sigma$ of $\Sigma$ at the boundary marked points $x_1,\ldots,x_k.$ Indeed, denote by $\pi : \widetilde\Sigma \to \Sigma$ the blowup map. Denote by $\HH\subset \C$ the upper half-plane. Denote by $\HB_r(s)\subset \HH$ the half-disk of radius $r$ centered at $s \in \R =  \partial\HH.$ For $i = 1,\ldots,k,$ let $U_i\subset \Sigma$ be an open neighborhood of $x_i$ with a local coordinate
\[
\xi_i : U_i \overset{\sim}{\longrightarrow} \HB_2(0), \qquad \xi_i(x_i) = 0.
\]
Possibly shrinking the $U_i,$ we arrange that $U_i \cap U_j = \emptyset$ for $i \neq j.$
Write $\widetilde U_i = \pi^{-1}(U_i).$ Then we have coordinates
\[
r_i : \widetilde U_i \to [0,2), \qquad \theta_i : \widetilde U_i \to [0,\pi],
\]
such that $\xi_i\circ\pi(r_i,\theta_i) = r_i e^{\sqrt{-1}\theta_i}.$ For $z \in \mathring{\Sigma},$ denote by $\Sigma_z$ the marked surface $(\Sigma,\mathbf x,\{z_1\})$ where $z_1 = z.$ For $z \in \partial \Sigma\setminus \mathbf x,$ denote by $Q_z$ the  marked surface $(\Sigma,\{x_0,x_1,\ldots,x_k\},\emptyset),$ where $x_0 = z.$ For $i \in [k]$ denote by $P_i$ the marked surface $(\Sigma,(\mathbf x\setminus \{x_i\})\cup \{x_0\},\emptyset),$ where $x_0 = x_i.$ For $\theta = 0,\pi,$ define
\[
R_{i,\theta} = (\HH\cup \{\infty\},\{x_{-2},x_{-3},x_i\},\emptyset), \quad x_{-2} = \infty, \quad x_{-3} = \cos{\theta}, \quad x_i = 0.
\]
Furthermore, we define
\begin{gather*}
S = (\HH\cup \{\infty\},\{x_{-1}\},\{z_1\}), \qquad x_{-1} = \infty, \quad z_1 = \sqrt{-1}, \\
T_{i,\theta} = (\HH \cup \{\infty\},\{x_{-1},x_i\},\{z_1\}), \quad x_{-1} = \infty, \quad x_i = 0, \quad z_1 = e^{\sqrt{-1}\theta}.
\end{gather*}
For $z \in \partial \Sigma\setminus \mathbf x,$ let $\Sigma_z$ denote the stable surface $(\{Q_z,S\},\sim)$ where $x_0 \sim x_{-1}.$ For $i \in [k]$ and $\theta \in (0,\pi),$ let $\Sigma_{i,\theta}$ denote the stable surface $(\{P_i,T_{i,\theta}\},\sim)$ where $x_0 \sim x_{-1}.$ For $i \in [k]$ and $\theta = 0,\pi,$ let $\Sigma_{i,\theta}$ denote the stable surface $(\{P_i,R_{i,\theta},S\},\sim)$ where $x_0\sim x_{-2}$ and $x_{-1} \sim x_{-3}.$
We define a diffeomorphism $f : \widetilde \Sigma \to F_p$ by
\[
f(z) =
\begin{cases}
[\Sigma_{\pi(z)}], & z \in \pi^{-1}(\Sigma\setminus \mathbf x)\\
[\Sigma_{i,\theta_i(z)}] & z \in \pi^{-1}(x_i).
\end{cases}
\]
Write $g = f^{-1}.$ Then there is a tautological isomorphism
\[
g^* T^*\widetilde\Sigma|_{\mathring{\widetilde\Sigma}} \simeq \CL_i|_{\mathring{F_p}}.
\]

We aim to construct a section $\bar s$ of $T^*\widetilde\Sigma|_{\ior \widetilde\Sigma}$ such that $g^*\bar s$ extends to a continuous section $\check s$ of $\CL_i|_{F_p}$ with $\check s|_{\partial F_p}$ special canonical. Indeed, let $\hat \nu : \Sigma \to \R$ satisfy
\begin{gather*}
\hat \nu(z) > 0, \quad z \in \mathring{\Sigma}, \qquad \qquad   \hat\nu(z) = 0, \quad z \in \partial \Sigma,\\
d\hat\nu_z \neq 0, \quad z \in \partial\Sigma,
\end{gather*}
and set $\nu = \hat \nu \circ \pi : \widetilde \Sigma \to \R.$
Let $\{\eta_0,\ldots,\eta_k\}$ be a partition of unity on~$\widetilde \Sigma$ subordinate to the cover $\{\widetilde\Sigma\setminus \pi^{-1}(\mathbf x),\widetilde U_1,\ldots,\widetilde U_k\}.$ Let
\[
\tau = \eta_0 + \sum_i \eta_i e^{-2\sqrt{-1}\theta_i}, \qquad \qquad \bar s = \tau \frac{d\nu}{\nu}.
\]

For $w \in \partial F_p,$ we calculate $\lim_{w' \to w} g^*\bar s(w')$ as follows. Write $z = g(w)$ and $z' = g(w').$ Suppose first that $z \in \partial \widetilde \Sigma \setminus \pi^{-1}(\mathbf x).$ Let $U \subset \Sigma$ be an open neighborhood of $\pi(z)$ with a local coordinate $\xi : U \overset{\sim}{\rightarrow} \HB_2(0)$ such that $\xi(\pi(z)) = 0.$ For $\epsilon >0$ and $a \in \R,$ let $\mu_{\epsilon,a} : \HH \to \HH$ be given by $\zeta \mapsto \epsilon \zeta + a.$
For $\pi(z') \in U,$ taking
\[
\epsilon = \epsilon(w')  = \im (\xi(\pi(z'))), \qquad a = a(w') = \re(\xi(\pi(z'))),
\]
we have
\[
\mu_{\epsilon,a}^{-1}(\xi(z')) = \sqrt{-1} = z_1 \in S.
\]
For $w'$ sufficiently close to $w,$ the smooth surface $\Sigma_{\pi(z')}$ is a deformation of the nodal surface $\Sigma_{\pi(z)}$ obtained by removing half-disks around the nodal points $x_0 \in Q_z$ and $x_{-1} \in S,$ and identifying half-annuli adjacent to the resulting boundaries. More explicitly, let
\[
A_\epsilon = \HB_{\sqrt{2/\epsilon}}(0)\setminus \HB_{1/\sqrt{2\epsilon}}(0).
\]
We glue the surfaces
\[
Q_z\setminus \xi^{-1}\left(\HB_{\sqrt{\epsilon/2}}(a)\right), \qquad\qquad \HB_{\sqrt{2/\epsilon}}(0)\subset S,
\]
along the map $\xi^{-1}\circ \mu_{\epsilon,a}|_{A_\epsilon}.$  The identification of $\Sigma_{\pi(z')}$ with the above deformation of $\Sigma_{\pi(z)}$ trivializes $\CL_i|_{F_p}$ near $w.$ We use this trivialization to compute
\begin{equation}\label{eq:dy1}
\lim_{w' \to w} g^*\bar s(w') = \lim_{\substack{\epsilon \to 0 \\ a \to 0}} \left. \left(\pi^{-1} \circ \xi^{-1}\circ\mu_{\epsilon,a}\right)^* \bar s \right|_{\sqrt{-1}} \in T^*_{\sqrt{-1}} S = T_{z_1}^*\Sigma_z.
\end{equation}
Writing $\xi = x + iy,$ we have $\hat\nu(x,y) = y \chi(x,y)$ where $\chi(x,0) > 0.$ So,
\begin{multline}\label{eq:dy2}
\lim_{\substack{\epsilon \to 0 \\ a \to 0}} \left. \left(\pi^{-1} \circ \xi^{-1}\circ\mu_{\epsilon,a}\right)^* \bar s\right|_{\sqrt{-1}} = \lim_{\substack{\epsilon \to 0 \\ a \to 0}} \left. \left(\xi^{-1}\circ\mu_{\epsilon,a}\right)^* \frac{d\hat\nu}{\hat\nu} \right|_{\sqrt{-1}} = \\ =\lim_{\substack{\epsilon \to 0 \\ a \to 0}} \left.\frac{dy}{y} + \frac{d(\chi \circ \mu_{\epsilon,a})}{\chi\circ \mu_{\epsilon,a}}\right|_{\sqrt{-1}} = dy.
\end{multline}
Here, the first equality holds because $\tau|_{\partial \widetilde \Sigma\setminus \pi^{-1}(\mathbf x)} \equiv 1.$

If $z \in \pi^{-1}(x_i)$ and $\theta_i(z) \neq 0, \pi,$ we proceed as follows. If $z' \in \widetilde U_i,$ taking $\epsilon = \epsilon(w') = r_i(z'),$ we have
\[
\lim_{w' \to w} \mu_{\epsilon,0}^{-1}(\xi_i(\pi(z'))) = e^{\sqrt{-1}\theta_i(z)} = z_1 \in T_{i,\theta_i(z)}.
\]
Thus by reasoning similar to the above, we have
\begin{align}\label{eq:scint}
\lim_{w' \to w} g^*
\bar s(w') &= \lim_{\epsilon \to 0} \left.(\pi^{-1}\circ\xi_i^{-1} \circ \mu_{\epsilon,0})^* \bar s \right|_{e^{\sqrt{-1}\theta_i(z)}} \\
 &= \tau(z)\lim_{\epsilon \to 0} \left. (\xi_i^{-1}\circ\mu_{\epsilon,a})^* \frac{d\hat\nu}{\hat\nu} \right|_{e^{\sqrt{-1}\theta_i(z)}} \notag \\
&= e^{-2\sqrt{-1}\theta_i(z)} \left.\frac{dy}{y}\right|_{e^{\sqrt{-1}\theta_i(z)}} \notag\\
&= \frac{e^{-2\sqrt{-1}\theta_i(z)}}{\sin \theta_i(z)} dy \in T_{e^{\sqrt{-1}\theta_i(z)}}^*T_{i,\theta_i(z)} = T_{z_1}^*\Sigma_{i,\theta_i(z)}. \notag
\end{align}
If $\theta_i(z) = 0,\pi,$ the situation is slightly more complicated because of the double bubble, but similar reasoning still shows that
\[
\lim_{w' \to w} g^*\bar s(w') = dy \in T^*_{\sqrt{-1}}S = T^*_{z_1}\Sigma_{i,\theta_i(z)}.
\]
Therefore, $g^*\bar s$ does indeed extend to a continuous section $\check s$ of $\CL_i|_{F_p}.$

Moreover, we deduce from the preceding calculations that $\check s|_{\partial F_p}$ is special canonical. Indeed, for $z \in \widetilde \Sigma \setminus \pi^{-1}(\mathbf x),$ equations~\eqref{eq:dy1} and~\eqref{eq:dy2} show that $\check s(w) = dy,$ independent of $w.$ Thus $\check s$ is pulled back from the base on the corresponding components of $\partial F_p.$ For $z \in \pi^{-1}(x_i)$ and $\theta_i(z) \neq 0,\pi,$ equation~\eqref{eq:scint} shows that
\[
\check s(w) = \frac{e^{-2\sqrt{-1}\theta_i(z)}}{\sin \theta_i(z)} dy.
\]
The map to the base component forgets $x_{-1} \in T_{i,\theta_i(z)}.$ So the remaining marked points $x_i$ and $z_1$ can be brought to a standard position by a Mobius transformation. Explicitly, let $\beta_\theta : \HH \to \HH$ be given by
\[
\beta_\theta(\zeta) = \frac{\zeta}{\zeta \cos \theta + \sin \theta}.
\]
So,
\[
\beta_\theta(\sqrt{-1}) = e^{i\theta} = z_1 \in T_{i,\theta}, \qquad \beta_{\theta}(0) = 0 = x_i \in T_{i,\theta}.
\]
Then
\[
\beta_\theta'(\zeta) = \frac{\sin \theta}{(\zeta \cos \theta + \sin \theta)^2}.
\]
In particular, $\beta'_\theta(\sqrt{-1}) = -e^{2\sqrt{-1}\theta}\sin \theta.$ It follows that
\[
\beta_{\theta_i(z)}^* \check s(w) = -dy,
\]
independent of $w.$ So, $\check s$ is pulled back from the base on the remaining components of $\partial F_p.$ The case $\theta_i(z) = 0,\pi,$ corresponds to a codimension $2$ corner of $F_p,$ so it follows by continuity of $\check s.$

Finally, choose $\tilde s$ to be a transverse perturbation of $\check s$ that agrees with $\check s$ in a neighborhood $V$ of $\partial F_p$ where $\check s$ does not vanish. We calculate $\# Z(\tilde s)$ by expressing it as a winding number. Let $\Xi$ be a Riemann surface, let $L \to \Xi$ be a complex line bundle, and let $\gamma \subset \Xi$ be a homologically trivial curve. Then $L|_{\gamma}$ has a distinguished trivialization. Thus if $\sigma$ is a section of $L,$ the winding number $W(\sigma,\gamma)$ of $\sigma$ around $\gamma$ is well defined. Let $\gamma \subset V \cap \ior F_p$ be a curve isotopic to $\partial F_p$ and let $\hat \gamma$ be the corresponding curve in $\ior \widetilde \Sigma.$ Then
\[
\# Z(\tilde s) = W(\tilde s,\gamma) = W(\check s,\gamma) = W(\bar s, \hat \gamma) = W(d\nu, \hat \gamma) + k.
\]
But it is well known that $W(d \nu,\hat \gamma)$ is negative the Euler characteristic of $\widetilde \Sigma,$ which in our case is $-1.$ The lemma follows.
\end{proof}

\begin{proof}[Proof of Theorem~\ref{thm:string_dilaton}, dilaton equation]
We have
\[
2^{\frac{k-1}{2}}\left\langle \tau_1\prod_{i = 1}^l \tau_{a_i}\sigma^k\right\rangle_0^o = \int_{\oCM_{0,k,l+1}} e(\CE_{\ba}),
\]
where $\ba = (a_1,\ldots,a_l,1).$ Let $\bb,\bc \in \Z_{\geq 0}^{l+1}$ satisfy $\bb + \bc = \ba.$ For all $q \in [l]$ such that $b_q \geq 1,$ by Lemma~\ref{lem:main_lem} cases~\ref{it:D_i} and~\ref{it:b}, Lemma~\ref{lm:prod} and Observation~\ref{obs:2}, we have
\[
e(\CE_{\bb,\bc}) = e(\CE_{\bb - \be_q,\bc + \be_q}).
\]
By induction, we obtain
\[
e(\CE_{\ba}) = e(\CE_{\be_{l+1},\ba - \be_{l+1}}).
\]
Let $\s' \in \CS_{\bze,\ba - \be_{l+1}}$ and $s \in \CS_{l+1}$ be such that $s\oplus\s'$ vanishes nowhere. By definition of $\CS_{\bze,\ba - \be_{l+1}},$ there exists $\hat \s' \in \CS_{\hat \ba}$ such that $\s' = For_{l+1}^* \hat \s'.$ Let $\bar \s'$ be a transverse extension of $\hat \s'$ to $\oCM_{0,k,l}$ and let $\tilde \s' = For_{l+1}^* \bar \s'.$ Since $Z(\hat \s') \subset \CM_{0,k,l}$ and $For_{l+1}|_{For_{l+1}^{-1}(\CM_{0,k,l})}$ is a submersion, it follows that $\tilde \s'$ is a transverse extension of $\s'.$ Choose an extension $\tilde s$ of $s$ such that $\tilde s \oplus \tilde \s'$ is transverse. Then,
\begin{multline*}
\int_{\oCM_{0,k,l+1}} e(\CE_{\be_{l+1},\ba - \be_{l+1}}) =
\# Z(\tilde s) \cap Z(\tilde \s') = (k + l - 1) \# Z(\hat \s') = \\
 = (k + l - 1) \int_{\oCM_{0,k,l}} e(\CE_{\hat \ba}) = (k + l-1) 2^{\frac{k-1}{2}}\left\langle \prod_{i = 1}^l \tau_{a_i} \sigma^k\right\rangle_0^o,
\end{multline*}
where in the second equality, we have used Lemma~\ref{lm:cool}
\end{proof}

\subsection{Proofs of TRR I and II}
Let
\[
E=\bigoplus\CL_i^{\oplus a_i}\rightarrow\oCM_{0,k,l},
\]
with $a_1 = n,$ and
\[
E_1 =  \CL_1^{\oplus n-1} \oplus \bigoplus_{i=2}^l \CL_i^{\oplus a_i} \to \oCM_{0,k,l}.
\]
Take
\[
\mathbf{s} = \bigoplus_{i\in\left[l\right]~,j\in \left[a_i\right]}s_{ij}
\]
with $s_{ij}\in \CS_i$, and
\[
\sr=\bigoplus_{\substack{i\in\left[l\right]~,j\in\left[a_i\right],\\ \left(i,j\right)\neq\left(1,1\right)}}s_{ij}.
\]
The proof hinges on a section $\tilde \tr \in C^\infty(\oCM_{0,k,l},\CL_1)$ defined as follows. At a smooth marked disk $D = \left(D,\mathbf{x},\mathbf{z}\right)$, which we identify with the upper half plane, set
\begin{equation}\label{eq:t}
\tilde \tr\left(D\right) = dz\left.\left(\frac{1}{z - x_1}-\frac{1}{z-\bar{z}_1}\right)\right|_{z = z_1} \in T_{z_1}^*D.
\end{equation}
We show the section $\tilde \tr$ extends smoothly to the compactified moduli space. Indeed, let $(\Sigma,\mathbf x,\mathbf z)$ be a stable marked disk, and let
\[
\Sigma_\C = \Sigma \coprod_{\partial \Sigma} \overline \Sigma
\]
be its complex double, a stable marked sphere. Consider the unique meromorphic differential $\omega_\Sigma$ on the normalization of $\Sigma_\C$ with the following properties. At $x_1$ it has a simple pole with residue $1$. At $\bar{z}_1$ it has a simple pole with residue $-1$. For any node the two preimages have at most simple poles, and the residues at these poles sum to zero. Apart from these points, $\omega_\Sigma$ is holomorphic. Then $\tilde \tr\left(\Sigma\right)$ is the evaluation of $\omega_\Sigma$ at $z_1$. As $z_1$ never coincides with a node, $\bar z_1$ or $x_1,$ it follows that $\tilde \tr$ is smooth. Write $\tr= \tilde \tr|_{\partial\oCM_{0,k,l}}$.

Let $\widetilde T \subset \partial \Gamma_{0,k,l}$ be the collection of stable graphs $\Gamma$ with exactly one open vertex $v^o_\Gamma$ and exactly one closed vertex $v^c_\Gamma$, such that $1 \in \ell_I(v^c_\Gamma).$ So for $\Gamma \in \widetilde T,$ we have $\oCM_\Gamma = \oCM_{v^o_\Gamma} \times \oCM_{v^c_\Gamma}.$ Equip $\oCM_\Gamma$ with the orientation $o_\Gamma$ given by the product of $o_{0,k(v^o_\Gamma),l(v^o_\Gamma)}$ and the complex orientation on $\oCM_{v^c_\Gamma}.$
\begin{lm}\label{lm:tr}
The zero locus of $\tilde \tr$ is $\bigcup_{\Gamma \in \widetilde T} \oCM_{\Gamma}.$ For $\Gamma \in \widetilde T,$ the subspace $\CM_\Gamma\subset \oCM_{0,k,l}$ is cut out transversally by $\tilde \tr$ with induced orientation $o_\Gamma.$
\end{lm}
\begin{proof}
On a component of $\Sigma_\C$ containing $x_1$ or $\bar{z}_1,$ the differential $\omega_\Sigma$ vanishes nowhere. Similarly, $\omega_\Sigma$ vanishes nowhere on components whose removal disconnects $x_1$ from $\bar{z}_1$. On other components it vanishes identically. Thus, $\tilde\tr$ vanishes exactly on stable disks $\Sigma$ such that in $\Sigma_\C$ the component containing $z_1$ is not on the route of components between the components of $x_1$ and $\bar{z}_1$. This is the case if and only if $z_1$ belongs to a sphere component of~$\Sigma.$ So the vanishing locus of $\tilde\tr$ is as claimed. The proof of transversality is similar to the proof of Lemma~\ref{lm:fmor}. The equality of orientations follows from induction on dimension by an argument similar to the proof of Lemma~\ref{lm:Bor}.
\end{proof}

Choose $\mathbf{s}$ satisfying the strong transversality condition of Lemma~\ref{lm:tscm} part \ref{it:better_trans}. Put $\trs = \tr\oplus\sr$.
We show that $\trs$ does not vanish, so $e\left(E;\trs\right)$ is defined. Indeed, $Z(\tr)$ consists of boundary strata of codimension at least $3$ in $\oCM_{0,k,l}.$ So, the transversality requirement of Lemma \ref{lm:tscm} part \ref{it:better_trans} guarantees that on such boundary strata $\sr$ does not vanish. Thus, $\trs$ does not vanish on $\partial \oCM_{0,k,l}.$

\begin{lm}\label{lm:br}
We have
\[
\int_{\oCM_{0,k,l}} e(E;\trs) = 2^{\frac{k-1}{2}}\sum_{S \coprod R = \{2,\ldots,l\}}\left\langle \tau_0\tau_{n-1}\prod_{i \in S}\tau_{a_i}\right\rangle^c_0
\left\langle \tau_0\prod_{i\in R}\tau_{a_i}\sigma^k\right\rangle^o_0.
\]
\end{lm}
\begin{proof}
Choose an extension $\tsr$ of $\sr$ to $\oCM_{0,k,l}$ that does not vanish on $\oCM_\Gamma \setminus \CM_\Gamma$ for $\Gamma \in \tilde T$ and such that $\tilde{\trs} = \tilde \tr \oplus \tsr$ is transverse. Such transversality is generic because $\tilde \tr$ is transverse along $\CM_\Gamma$ and non-zero outside of $\oCM_\Gamma$ by Lemma~\ref{lm:tr}. Again by Lemma~\ref{lm:tr}, we obtain
\begin{equation}\label{eq:etst}
\int_{\oCM_{0,k,l}} e(E;\trs) = \# Z(\tilde \tr \oplus \tsr) = \sum_{\Gamma \in \widetilde T}\int_{\oCM_\Gamma} e\left(E_1|_{\oCM_\Gamma};\sr|_{\partial\oCM_\Gamma}\right).
\end{equation}
Recall Definitions~\ref{def:si} and~\ref{def:span}. For $\Gamma \in \widetilde T,$ and $\Lambda \in \pu \Gamma,$ abbreviate
\[
\Lambda^c = \Lambda_{\varsigma_{\Lambda,\Gamma}^{-1}(v_\Gamma^c)}, \qquad \Lambda^o = \Lambda_{\varsigma_{\Lambda,\Gamma}^{-1}(v_\Gamma^o)}.
\]
Thus we have a bijection
\[
\pu \Gamma \overset{\sim}{\longrightarrow} \pu \Gamma^o \times \pu \Gamma^c, \qquad \Lambda \mapsto (\Lambda^o,\Lambda^c),
\]
and a corresponding diffeomorphism
\[
\tilde b : \oCM_\Gamma \longrightarrow \oCM_{\Gamma^o} \times \oCM_{\Gamma^c}.
\]
given by
\[
\tilde b|_{\oCM_\Lambda} = For_{\Lambda,\Lambda^o}\times For_{\Lambda,\Lambda^c}, \qquad \Lambda \in \pu\Gamma.
\]
Additionally, we have a bijection
\begin{equation}\label{eq:biG}
\pB\Gamma \overset{\sim}{\longrightarrow} \pB \Gamma^o \times \pu \Gamma^c, \qquad \Lambda \mapsto (\Lambda^o,\Lambda^c),
\end{equation}
and a corresponding diffeomorphism
\[
b : \partial \oCM_\Gamma \to \partial \oCM_{\Gamma^o} \times \oCM_{\Gamma^c}
\]
given by $b = \tilde b |_{\partial \oCM_\Gamma}.$
Denote by
\begin{gather*}
\tilde p_o : \oCM_{\Gamma^o} \times \oCM_{\Gamma^c} \longrightarrow \oCM_{\Gamma^o}, \qquad \tilde p_c: \oCM_{\Gamma^o} \times \oCM_{\Gamma^c} \longrightarrow \oCM_{\Gamma^c}, \\
p_o : \partial \oCM_{\Gamma^o} \times \oCM_{\Gamma^c} \longrightarrow \partial \oCM_{\Gamma^o}, \qquad p_c: \partial \oCM_{\Gamma^o} \times \oCM_{\Gamma^c} \longrightarrow \oCM_{\Gamma^c},
\end{gather*}
the projection maps.
Let
\begin{gather*}
E^c_\Gamma = \CL_1^{\oplus n-1}\oplus \bigoplus_{i \in \ell_I(v^c_\Gamma)\setminus \{1\}} \CL_i^{ \oplus a_i} \longrightarrow \oCM_{\Gamma^c}, \\
E^o_\Gamma = \bigoplus_{i \in \ell_I(v^o_\Gamma)} \CL_i^{\oplus a_i} \longrightarrow \oCM_{\Gamma^o}.
\end{gather*}
Thus
\[
\tilde b^* (\tilde p_o^* E^o \oplus \tilde p_c^* E^c) = E_1|_{\oCM_\Gamma}.
\]
For $\Lambda \in \pu\Gamma,$ we have
\begin{gather*}
\bv{i}{\Lambda} = \bv{i}{\Lambda^c}, \quad i \in \ell_I(\Lambda^c), \qquad \bv{i}{\Lambda} = \bv{i}{\Lambda^o}, \quad i \in \ell_I(\Lambda^o), \\
\ell_I(\Lambda) = \ell_I(\Lambda^c) \cup \ell_I(\Lambda^o).
\end{gather*}
So, bijection~\eqref{eq:biG} implies
\begin{equation}\label{eq:biV}
\CV^i_{\pB\Gamma} = \CV^i_{\pB\Gamma^o} \cup \CV^i_{\pu\Gamma^c}, \qquad i \in [l].
\end{equation}
Since $\pB\Gamma \subset \pB\Gamma_{0,k,l},$ by the definition of a special canonical multisection, we have
\[
s_{ij}^v \in C^\infty_m(\CM_v,\CL_i), \qquad v \in \CV^i_{\pB\Gamma},
\]
such that $s_{ij}^\Lambda = \Phi_{\Lambda,i}^* s_{ij}^v$ for all $\Lambda \in \pB\Gamma$ with $\bv{i}{\Lambda} = v.$  Let $\s_\Gamma^c \in C^\infty_m(E^c_\Gamma)$ and $\s_\Gamma^o \in C^\infty_m(E^o_\Gamma|_{\partial\oCM_{\Gamma^o}})$ be given by
\begin{gather*}
(\s^c_\Gamma)^{\Psi} = \bigoplus_{\substack{i \in \ell_I(v^c_\Gamma),\; j \in [a_{ij}]\\ (i,j) \neq (1,1)}} \Phi_{\Psi,i}^*s_{ij}^{\bv{i}{\Psi}}, \qquad \Psi \in \pu \Gamma^c, \\
(\s^o_\Gamma)^\Omega = \bigoplus_{i \in \ell_I(v^o_\Gamma),\; j \in [a_{ij}]} \Phi_{\Omega,i}^*s_{ij}^{\bv{i}{\Omega}}, \qquad \Omega \in \pB \Gamma^o.
\end{gather*}
Here, $s_{ij}^{\bv{i}{\Psi}}$ and $s_{ij}^{\bv{i}{\Omega}}$ are predetermined because of equation~\eqref{eq:biV}.
Since we have chosen $\s$ to satisfy the strong transversality condition of Lemma~\ref{lm:tscm} part \ref{it:better_trans}, the multisections $\s^c_\Gamma$ and $\s^o_\Gamma$ are transverse to $0.$
For $\Lambda \in \pB\Gamma,$ Observation~\ref{obs:fuf} and equation~\eqref{eq:phigi} imply that
\begin{gather*}
\Phi_{\Lambda,i} = \Phi_{\Lambda^o,i} \circ p_o \circ b, \qquad i \in \ell_I(v^o_\Gamma), \\
\Phi_{\Lambda,i} = \Phi_{\Lambda^c,i} \circ p_c \circ b, \qquad i \in \ell_I(v^c_\Gamma).
\end{gather*}
It follows that for $\Gamma \in \widetilde T,$ we have
\begin{equation}\label{eq:s'pro}
\sr|_{\partial\oCM_\Gamma} = b^* (p_o^* \s_\Gamma^o \oplus p_c^* \s_\Gamma^c).
\end{equation}
Choose a transverse extension $\tilde \s_\Gamma^o$ of $\s_\Gamma^o$ to $\oCM_{\Gamma^o}.$ Then equation~\eqref{eq:s'pro} implies that $\tilde b^*(\tilde p_o^*\tilde\s_\Gamma^o \oplus \tilde p_c^*\s_\Gamma^c)$ is a transverse extension of $\sr|_{\partial \oCM_\Gamma}$ to $\oCM_\Gamma.$
Therefore,
\begin{equation}
\int_{\oCM_\Gamma} e\left(E_1|_{\oCM_\Gamma};\sr|_{\partial\oCM_\Gamma}\right) = \# Z(\tilde p_o^*\tilde\s_\Gamma^o) \cap Z(\tilde p_c^*\s_\Gamma^c). \label{eq:gpr}
\end{equation}
Dimension counting and transversality show this number vanishes unless $\rk E^o_\Gamma = \dim_\C \oCM_{\Gamma^o}$ and $\rk E^c_\Gamma = \dim_\C \oCM_{\Gamma^c}.$ In that case, transversality implies that $\s^o_\Gamma$ vanishes nowhere. Thus
\begin{equation}\label{eq:prdi}
\# Z(\tilde p_o^*\tilde\s_\Gamma^o) \cap Z(\tilde p_c^*\s_\Gamma^c) = \left( \int_{\oCM_{\Gamma^o}} e(E^o_\Gamma,\s^o_\Gamma)\right)\left(\int_{\oCM_{\Gamma^c}} e(E^c_\Gamma) \right).
\end{equation}
The graph $\Gamma^o$ (resp. $\Gamma^c$) has a single vertex $v^o_\Gamma$ (resp. $v^c_\Gamma$). By construction, $\s_\Gamma^o$ is a canonical boundary condition. So,
\begin{gather}
\int_{\oCM_{\Gamma^o}} e(E^o_\Gamma,\s^o_\Gamma) = 2^{\frac{k-1}{2}}\left\langle \tau_0\prod_{i\in \ell_I(v^o_\Gamma)}\tau_{a_i}\sigma^k\right\rangle^o_0, \label{eq:go}\\
\int_{\oCM_{\Gamma^c}} e(E^c_\Gamma) =
\left\langle \tau_0\tau_{n-1}\prod_{\substack{i\in \ell_I(v^c_\Gamma)\\ i \neq 1}}\tau_{a_i}\right\rangle^c_0. \label{eq:gc}
\end{gather}
Equations~\eqref{eq:etst},~\eqref{eq:gpr},~\eqref{eq:prdi},~\eqref{eq:go} and~\eqref{eq:gc}, together imply the lemma.
\end{proof}

It remains to analyze the difference between $\tr$ and a canonical multisection. Let $U \subset \pB\Gamma_{0,k,l}$ be the collection of graphs $\Gamma$ with exactly two vertices $v^\pm_\Gamma,$ both open, such that
\[
1 \in \ell_I(v^-_\Gamma), \qquad \rz{1} \in \ell_B(v^{+}_\Gamma),
\]
and the unique edge $e_\Gamma$ of $\Gamma$ is legal for $v^+_\Gamma$ and thus illegal for $v^-_\Gamma.$ Let
\begin{equation}\label{eq:UV}
V = \pB \Gamma_{0,k,l}\setminus {\pu U}.
\end{equation}
\begin{lm}\label{lm:tcan}
Let $\Gamma \in V.$ Then $\tr|_{\CM_\Gamma}$ is canonical.
\end{lm}
\begin{proof}
If $\Gamma \in V$ and $(\Sigma,\mathbf{x},\mathbf z) \in \CM_\Gamma$, then either $z_1$ and $x_1$ are in the same component, or the nodal point is legal for the component of $z_1$. In the first case, $\omega_\Sigma$ has does not depend on the position of the nodal point, so neither does $\tr.$ In the second case, $\omega_\Sigma$ may have a pole at the nodal point and so $\tr$ may depend on the position of the nodal point on the component of $z_1.$ But the nodal point is legal for that component. In both cases, $\tr$ does not depend on the position of an illegal nodal point, so it is canonical.
\end{proof}

The following observation and lemma quantify the difference between $\tr$ and a canonical multisection on $\oCM_\Gamma$ for $\Gamma \in U.$ Let $p \in \CM_{\CB\Gamma}.$ Let $F_p$ be the fiber over $p$ of the map $F_\Gamma : \oCM_\Gamma \to \oCM_{\CB\Gamma}$ equipped with its natural orientation. So $F_p$ is a collection of $a = |\ell_B(v^-_\Gamma)|$ closed intervals corresponding to the $a$ segments between marked points where the illegal nodal point can move.
The following observation is a consequence of Observation~\ref{obs:identification_bundle_over_base}.
\begin{obs}\label{obs:tatr}
The restriction of the tautological line $\CL_i|_{F_p}$ is canonically trivial.
\end{obs}
So, we think of sections of $\CL_i|_{F_p}$ as complex valued functions well-defined up to multiplication by a constant in $\C^\times.$ The following observation is immediate from the definition of a canonical section.
\begin{obs}\label{obs:ccon}
A canonical section of $\CL_i|_{F_p}$ is constant.
\end{obs}
On the other hand, the TRR section $\tr$ twists non-trivially around $F_p$ as follows. For $i \in \ell_B(v^{-}_\Gamma),$ let $\Gamma_i$ be the unique stable graph in $\pB\Gamma$ with three open vertices $v_i^0,v_i^\pm$ and two boundary edges $e^\pm = \{v_i^\pm,v_i^0\}$ such that $\ell_B(v_i^0) = i.$
The boundary $\partial F_p$ consists of two stable disks modelled on each graph $\Gamma_i$ for $i \in \ell_B(v^-_\Gamma),$ one for each cyclic order of the $3 = k(v^0_i)$ boundary marked points on the component corresponding to $v^0_i.$ Let $\hat F_p$ be the quotient space of $F_p$ obtained by identifying the two boundary points corresponding to $\Gamma_i$ for each $i \in \ell_B(v^-_\Gamma).$ Thus $\hat F_p$ is homeomorphic to the circle $S^1.$ The following observation follows from the definition of $\tr.$
\begin{obs}\label{obs:desc}
The section $\tr|_{F_p}$ descends to a continuous function $\hat \tr_p : \hat F_p \to \C^\times.$
\end{obs}
\begin{lm}\label{lm:wino}
The winding number of $\hat \tr_p$ is $-1.$
\end{lm}
\begin{proof}
We define a map from a subset $B \subset (0,2\pi),$ to $\ior F_p$ as follows. To each $b \in B,$ we assign a surface $\Sigma_b = (\Sigma^-_b,\Sigma^+)$. The component $\Sigma^+$ corresponds to the vertex $v^+_\Gamma$ and the component $\Sigma_b^-$ corresponds to the vertex $v^-_\Gamma.$ As implied by the notation, $\Sigma^+$ is independent of~$b.$ The exact form of $\Sigma^+$ is not important for the present calculation, and its isomorphism class is determined uniquely by $p.$ We fix $\Sigma^-_b$ as follows. Let $\nu = i_{v^-_\Gamma}\left(e_\Gamma\right) \in \LL$ and choose an arbitrary $\rz{i_0} \in \ell_B(v^-_\Gamma).$ Identify $\Sigma^-_b$ with the unit disk $D^2 \subset \C$ in such a way that $z_1 = 0$ and $x_{i_0} = 1.$ The position of the remaining boundary marked points in $\partial D^2$ is then uniquely determined by $p.$ Take $B$ to be the set of arguments of the complement of the marked points in $\partial D^2.$ The parameter $b \in B$ determines the nodal point $x_\nu \in \partial \Sigma^-_b$ by the formula $x_\nu = e^{\sqrt{-1}b}.$ The complex double $(\Sigma^-_b)_\C$ is naturally identified with the extended complex plane $\C\cup \infty$. The point conjugate to $z_1$ is $\infty$ and
\[
\omega_{\Sigma_b}|_{\Sigma_b^-} = \frac{dz}{z - x_\nu} = \frac{dz}{z - e^{\sqrt{-1}b}}.
\]
So
\begin{equation}\label{eq:tform}
\tr(\Sigma_b) = \omega_{\Sigma_b}|_{z_1}  =  -e^{-\sqrt{-1}b}.
\end{equation}
The continuity of $\hat \tr_p$ and formula~\eqref{eq:tform} imply that $\hat \tr_p$ rotates once in the negative direction around the fiber $\hat F_p.$
\end{proof}

\begin{lm}\label{lm:brbs}
We have
\begin{multline*}
\int_{\oCM_{0,k,l}} e(E;\mathbf{s}) - \int_{\oCM_{0,k,l}} e(E;\trs) =  \\
= 2^{\frac{k-1}{2}}\sum_{\substack{S \coprod R = \{2,\ldots,l\} \\ k_1 + k_2 = k-1}} \binom{k}{k_1} \left\langle \tau_{n-1} \prod_{i \in S} \tau_{a_i} \sigma^{k_1}\right\rangle^o_0 \left\langle \prod_{i \in R} \tau_{a_i} \sigma^{k_2+2}\right\rangle^o_0.
\end{multline*}
\end{lm}
\begin{proof}
Let
\[
E_2 = \CL_1 \to \oCM_{0,k,l},
\]
so $E = E_1 \oplus E_2.$ Let $\CC =  V.$ Lemma~\ref{lm:tcan} shows that $\s$ and $\trs$ satisfy the hypotheses of Lemma~\ref{lem:trickey_homotopy} with the preceding choice of $E_1,E_2,\CC.$ So, we obtain a homotopy $H$ between $\mathbf{s}$ and $\trs$ of the form~\eqref{eq:srw} that is transverse to zero, vanishes nowhere on $\CM_\Gamma \times [0,1]$ for $\Gamma \in V,$ and such that the projection of $H$ to $E_1$ equals $\sr$ independent of time. By Lemma~\ref{lm:H} and equation~\eqref{eq:UV}, we have
\begin{multline}\label{eq:rsH}
\int_{\oCM_{0,k,l}} e(E;\mathbf{s}) - \int_{\oCM_{0,k,l}} e(E;\trs) =\\
 = -\# Z(H) = -\sum_{\Gamma \in U} \# Z\left(H|_{\oCM_\Gamma\times[0,1]}\right).
\end{multline}
Denote by $\pi : \partial \oCM_{0,k,l} \times [0,1] \to \partial\oCM_{0,k,l}$ the projection. Decompose $H = H_1 \oplus H_2,$ where $H_i \in C^\infty_m(\pi^* E_i).$ Then $H_1 = \pi^* \sr.$ Transversality implies that
\[
Z\left(\sr|_{\oCM_\Gamma}\right) \subset \CM_\Gamma.
\]
By Remark~\ref{rmk:scc}, for each $\Gamma \in \pB\Gamma_{0,k,l}$ we have $\sr^{\Gamma} = F^*_\Gamma \sr^{\CB\Gamma}.$ Write
\[
\# Z\left(\sr^{\CB\Gamma}\right) = \sum_{p \in Z\left(\sr^{\CB\Gamma}\right)} \epsilon(p),
\]
where $\epsilon(p)$ is the weight of $p$ as in Definition~\ref{df:or}.
It follows from Lemma~\ref{lm:Bor} that for $\Gamma \in U,$ we have
\begin{align}\label{eq:HG}
\# Z\left(H|_{\oCM_\Gamma\times[0,1]}\right) &= \# Z\left(\pi^* F_\Gamma^* \sr^{\CB\Gamma}\right) \cap Z(H_2) \\
& = \sum_{p \in Z\left(\sr^{\CB\Gamma}\right)}\epsilon(p)\cdot \# Z\left(H_2|_{F_p \times [0,1]}\right).\notag
\end{align}
Since $H$ is of the form~\eqref{eq:srw}, we have
\begin{equation}\label{eq:H2}
H_2(q,t) = \rho(q)t + s_{11}(q)(1-t) + t(1-t) w_2(q),
\end{equation}
where $w_2$ is a canonical multisection.
Let $p \in Z(\sr^{\CB\Gamma}).$ It follows from equation~\eqref{eq:H2} and Observations~\ref{obs:tatr},~\ref{obs:ccon} and~\ref{obs:desc}, that $H_2|_{F_p \times [0,1]}$ descends to a homotopy $\hat H_{2,p}$ on $\hat F_p \times [0,1],$ which we may think of as taking values in $\C^\times.$ Thus
\begin{equation}\label{eq:HhH}
\#Z\left(H_2|_{F_p \times [0,1]}\right) = \#Z\left(\hat H_{2,p}\right).
\end{equation}
Since $\hat H_{2,p}|_{\hat F_p \times \{0\}}$ is canonical and $\hat H_{2,p}|_{\hat F_p\times \{1\}} = \hat \rho_p,$ Observation~\ref{obs:ccon} and Lemma~\ref{lm:wino} imply that
\begin{equation}\label{eq:hH-1}
\#Z\left(\hat H_{2,p}\right) = -1.
\end{equation}
Combining equations~\eqref{eq:rsH},~\eqref{eq:HG},~\eqref{eq:HhH} and~\eqref{eq:hH-1}, we obtain
\begin{equation}\label{eq:sumU}
\int_{\oCM_{0,k,l}} e(E;\mathbf{s}) - \int_{\oCM_{0,k,l}} e(E;\trs) = \sum_{\Gamma \in U} \# Z\left(\sr^{\CB \Gamma}\right).
\end{equation}

It remains to analyze $\# Z\left(\sr^{\CB \Gamma}\right)$ for $\Gamma \in U.$ Denote by $\hat v^\pm_\Gamma \in V(\CB \Gamma)$ the vertices corresponding to $v^\pm_\Gamma \in V(\Gamma).$ Recall Definitions~\ref{def:si} and~\ref{def:span}. For $\Lambda \in \pu\CB\Gamma,$ abbreviate
\[
\Lambda^\pm = \Lambda_{\varsigma_{\Lambda,\CB\Gamma}^{-1}\left(\hat v^\pm_\Gamma\right)}.
\]
Thus we have a bijection
\begin{equation}\label{eq:biG2}
\pu\CB\Gamma \overset{\sim}{\longrightarrow} \pu \CB\Gamma^+ \times \pu\CB\Gamma^-, \qquad \Lambda \mapsto (\Lambda^+,\Lambda^-),
\end{equation}
and a corresponding diffeomorphism
\[
d : \oCM_{\CB\Gamma} \longrightarrow \oCM_{\CB\Gamma^+} \times \oCM_{\CB\Gamma^-}
\]
given by
\[
d|_{\oCM_\Lambda} = For_{\Lambda,\Lambda^+}\times For_{\Lambda,\Lambda^-}, \qquad \Lambda \in \pu\CB\Gamma.
\]
Denote by
\[
p_\pm : \oCM_{\CB\Gamma^+} \times \oCM_{\CB\Gamma^-} \longrightarrow \oCM_{\CB\Gamma^\pm}
\]
the projection maps.
Let
\begin{gather*}
E^+_{\CB\Gamma} = \CL_1^{\oplus n-1}\oplus\bigoplus_{i \in \ell_I(v^+_\Gamma)\setminus\{1\}}\CL_i^{\oplus a_i} \longrightarrow \oCM_{\CB\Gamma^+}, \\
E^-_{\CB\Gamma} = \bigoplus_{i \in \ell_I(v^-_\Gamma)} \CL_i^{\oplus a_i} \longrightarrow \oCM_{\CB\Gamma^-}, \\
E_{\CB\Gamma} = \CL_1^{\oplus n-1} \oplus \bigoplus_{i = 2}^l \CL_i^{\oplus a_i} \longrightarrow \oCM_{\CB\Gamma}.
\end{gather*}
Thus
\[
d^*\left(p_+^*E^+ \oplus p_-^* E^-\right) = E_{\CB\Gamma}.
\]
Observation~\ref{obs:BpB} and bijection~\eqref{eq:biG2} imply that
\begin{equation}\label{eq:V2}
\CV^i_{\pu \Gamma} = \CV^i_{\pu \CB\Gamma} = \CV^i_{\pu \CB\Gamma^+} \cup \CV^i_{\pu \CB\Gamma^-}.
\end{equation}
Since $\pu \Gamma \subset \pB \Gamma_{0,k,l},$ by definition of a special canonical multisection, we have
\[
s^v_{ij} \in C^\infty_m(\CM_v,\CL_i), \qquad v \in \CV^i_{\pu\Gamma},
\]
such that $s_{ij}^\Lambda = \Phi^*_{\Lambda,i}s^v_{ij}$ for all $\Lambda \in \pB \Gamma$ with $\bv{i}{\Lambda} = v.$
Let
\[
\s^\pm_{\CB\Gamma} \in C^\infty_m(E^\pm_{\CB\Gamma})
\]
be given by
\begin{gather*}
\left(\s^+_{\CB\Gamma}\right)^\Omega = \bigoplus_{\substack{i \in \ell_I(v_\Gamma^+),\; j \in [a_{ij}], \\ (i,j) \neq (1,1)}} \Phi_{\Omega,i}^*s_{ij}^{\bv{i}{\Omega}}, \qquad \Omega \in \pu\CB\Gamma^+, \\
\left(\s^-_{\CB\Gamma}\right)^\Omega = \bigoplus_{i \in \ell_I(v_\Gamma^-),\; j \in [a_{ij}]} \Phi_{\Omega,i}^*s_{ij}^{\bv{i}{\Omega}}, \qquad \Omega \in \pu\CB\Gamma^-.
\end{gather*}
Here, $s_{ij}^\bv{i}{\Omega}$ are predetermined because of equation~\eqref{eq:V2}. Since we have chosen $\s$ to satisfy the strong transversality condition of Lemma~\ref{lm:tscm} part \ref{it:better_trans}, the multisections $\s^\pm_{\CB\Gamma}$ are transverse to $0.$
For $\Omega \in \pu \CB\Gamma,$ Observation~\ref{obs:fuf} and equation~\eqref{eq:phigi} imply that
\[
\Phi_{\Omega,i} = \Phi_{\Omega^\pm,i} \circ p_\pm \circ d, \qquad i \in \ell_I(v^\pm_\Gamma).
\]
It follows that
\begin{equation*}
\sr^{\CB\Gamma} = d^* (p_+^* \s_{\CB\Gamma}^+ \oplus p_-^* \s_{\CB\Gamma}^-).
\end{equation*}
Thus
\begin{equation}\label{eq:spBG}
\# Z\left(\sr^{\CB\Gamma}\right) = \# Z\left(p_+^* \s_{\CB \Gamma}^+\right) \cap Z\left(p_-^* \s_{\CB \Gamma}^-\right).
\end{equation}
Dimension counting and transversality show this number vanishes unless $\rk E_{\CB \Gamma}^\pm = \dim_\C \oCM_{\CB\Gamma}^\pm.$ In that case, transversality implies that $\s_{\CB\Gamma}^\pm|_{\partial \oCM_{\CB\Gamma}^\pm}$ vanishes nowhere. Thus
\begin{multline}\label{eq:pmEu}
\# Z\left(p_+^* \s_{\CB \Gamma}^+\right) \cap Z\left(p_-^* \s_{\CB \Gamma}^-\right) = \\ =\left(\int_{\oCM_{\CB\Gamma^+}} e\left(E_{\CB\Gamma}^+,\s_{\CB\Gamma^+}|_{\partial \oCM_{\CB\Gamma}^+}\right)\right) \left(\int_{\oCM_{\CB\Gamma^-}} e\left(E_{\CB\Gamma}^-,\s_{\CB\Gamma^-}|_{\partial \oCM_{\CB\Gamma}^-}\right)\right).
\end{multline}
The graph $\CB\Gamma^\pm$ has a single vertex $v_{\CB \Gamma}^\pm.$ By construction, $\s_{\CB\Gamma}^\pm|_{\partial \oCM_{\CB\Gamma}^\pm}$ is a canonical boundary condition. So,
\begin{equation}\label{eq:pmde}
\int_{\oCM_{\CB\Gamma^\pm}} e\left(E_{\CB\Gamma}^\pm,\s_{\CB\Gamma^\pm}|_{\partial \oCM_{\CB\Gamma}^\pm}\right) = 2^\frac{k\left(\hat v_{\Gamma}^\pm\right)-1}{2} \left\langle \prod_{i \in \ell_I(\hat v_{\Gamma}^\pm)} \tau_{a_i} \sigma^{k\left(\hat v_{\Gamma}^\pm\right)} \right \rangle_0^o.
\end{equation}
For each $\Gamma \in U,$ we have $\rz{1} \in \ell_B(v^{+}_\Gamma)$ and $e_\Gamma$ is legal for $v^+_\Gamma$. It follows that
\[
k\left(\hat v_{\Gamma}^+\right) \geq 2, \qquad \Gamma \in U.
\]
Keeping in mind that
\[
k\left(\hat v_{\Gamma}^+\right) + k\left(\hat v_{\Gamma}^-\right) = k + 1, \qquad \ell_I(\hat v_\Gamma^+) \cup \ell_I(\hat v_\Gamma^-) = \{2,\ldots,l\},
\]
equations~\eqref{eq:sumU},~\eqref{eq:spBG},~\eqref{eq:pmEu} and~\eqref{eq:pmde}, imply the lemma.
\end{proof}
\begin{proof}[Proof of Theorem \ref{thm:trr1_2}]
The differential equation TRR I is equivalent to the following recursion relation:
\begin{multline}\label{eq:TRRr}
\left\langle \tau_n \prod_{i = 2}^l \tau_{a_i} \sigma^k \right \rangle^o_0 = \\
\begin{aligned}
= &\sum_{S \coprod R = \{2,\ldots,l\}}\left\langle \tau_0\tau_{n-1}\prod_{i \in S}\tau_{a_i}\right\rangle^c_0
\left\langle \tau_0\prod_{i\in R}\tau_{a_i}\sigma^k\right\rangle^o_0 +\\
& + \sum_{\substack{S \coprod R = \{2,\ldots,l\} \\ k_1 + k_2 = k-1}} \binom{k}{k_1} \left\langle \tau_{n-1} \prod_{i \in S} \tau_{a_i} \sigma^{k_1}\right\rangle^o_0 \left\langle \prod_{i \in R} \tau_{a_i} \sigma^{k_2+2}\right\rangle^o_0.
\end{aligned}
\end{multline}
By definition
\[
\left\langle \tau_n \prod_{i = 2}^l \tau_{a_i} \sigma^k \right \rangle^o_0 = \int_{\oCM_{0,k,l}} e(E;\mathbf{s}).
\]
So, recursion~\eqref{eq:TRRr} follows immediately from Lemmas~\ref{lm:br} and~\ref{lm:brbs}. The proof of TRR II is similar, except that we define $\omega_\Sigma$ to be the unique meromorphic differential on the normalization of $\Sigma$ with the following properties. At $\bar z_1$ it has a simple pole with residue $-1$, and at $z_2$ it has a simple pole with residue $1.$ For any node the two preimages have at most simple poles and the residues at these poles sum to zero. Apart from these points, $\omega_\Sigma$ is holomorphic. As in the proof of TRR I, the section $\tilde \tr\left(\Sigma\right)$ is the evaluation of $\omega_\Sigma$ at $z_1.$
\end{proof}

\section{Proof of Theorem \ref{thm:vir}}  \label{VirSec}
\subsection{Virasoro in genus \texorpdfstring{$0$}{0}}
The open Virasoro operators $\mathcal{L}_n$ and the
 partitions functions $F_0^c$ and $F_0^o$ were
defined in Section \ref{deq}. Define
$$ G_r = \mathcal{L}_r \ \exp(u^{-2}F^c_0+u^{-1}F^o_0)$$
for $r\geq -1$.
The genus 0 term of $G_r$ is defined by
$$\text{Coeff}_{\ u^{-1}} \Big(  G_r \exp(-u^{-2}F^c_0-u^{-1}F^o_0)
\Big) \ .$$
The claim of Theorem \ref{thm:vir} is:
$$\forall r\geq -1, \ \ \ \text{Coeff}_{\ u^{-1}} \Big(  G_r \exp(-u^{-2}F^c_0-u^{-1}F^o_0)
\Big) = 0 \ .$$

By the open string and dilaton equations, genus 0 terms of
$G_{-1}$ and $G_0$ vanish. Using the Virasoro bracket, Theorem \ref{thm:vir}
follows from
the vanishing
\begin{equation} \label{fvvt2}
\text{Coeff}_{\ u^{-1}} \Big(  G_2 \exp(-u^{-2}F^c_0-u^{-1}F^o_0)
\Big) =0\ .
\end{equation}
However, for the proof of \eqref{fvvt2}, we will require
the vanishing
\begin{equation} \label{fvvt1}
\text{Coeff}_{\ u^{-1}} \Big(  G_1 \exp(-u^{-2}F^c_0-u^{-1}F^o_0)
\Big) =0\ .
\end{equation}

\subsection{Vanishing for \texorpdfstring{$r=1$}{r=1}}
By unravelling the definition of $\mathcal{L}_1$, we can write the
vanishing \eqref{fvvt1} explicitly for $G_1$.
Using the Virasoro bracket
$$[\mathcal{L}_{-1}, \mathcal{L}_1]= -2 \mathcal{L}_0, \ $$
we need only check the vanishing of $G_1$ at coefficients
independent of $t_0$. The resulting equation
is
\begin{multline} \label{zzzz}
-\frac{15}{4} \blangle \tau_2  \prod_{i=1}^l \tau_{a_i}
\ \sigma^k \brangle_0^o
+\sum_{i=1}^l  \frac{(2a_i+1)(2a_i+3)}{4}
\blangle \tau_{a_i+1}  \prod_{j\neq i}\tau_{a_j} \
\sigma^k\brangle _0^o\\
+ \sum_{S\cup T=\{1,\ldots,l\}}
\blangle \prod_{i\in S} \tau_{a_i} \ \sigma^{k_S}\brangle_0^o\
k\binom{k-1}{k_S-1} \blangle  \prod_{i\in T} \tau_{a_i} \ \sigma^{k_T}
\brangle_0^o\ = 0.
\end{multline}
for $a_i\geq 1$ for all $i$.

Following the notation \eqref{zzr},
the number of boundary markings in \eqref{zzzz},
is set by the dimension constraint
\begin{eqnarray*}
k&= &5+ 2A -2l\\
k_S &=&3+2A_S - 2 l_S \ ,
\end{eqnarray*}
where the conventions
\begin{equation*}
 A= \sum_{i=1}^l a_i,    \ \ \ \ \ \ \forall i \ a_i \geq 1\ , \
\end{equation*}
\begin{equation*}
A_S= \sum_{i\in S} a_i, l_S=|S|, \forall S\subseteq\{1,2,\ldots,l\}\
\end{equation*}
are used.

After substituting the evaluation of Theorem \ref{ttxx} and
cancelling factors and simplifying, we reduce \eqref{zzzz} to the
identity:
\begin{multline}\label{vxxz2}
\frac{20+8A-8l}{4} (3+2A-l)! =
\\ \sum_{S\cup T=\{1,\ldots,l\}}
(5+2A-2l) \binom{4+2A-2l}{2+2A_S-2l_S} (1+2A_S-l_S)!(1+2A_T-l_T)!\ .
\end{multline}

\subsection{Closed TRR}
In order to prove \eqref{vxxz2},
we use the closed TRR in
genus 0 to derive combinatorial identities. The following identity is obtained from closed TRR:
\begin{multline*}
\blangle \tau_2 \tau_2 \tau_0 \prod_{i=1}^l
\tau_{2a_i-1} \ \tau_0^{4+2A-2l}\brangle_0^c = \\
\sum_{S\cup T=\{1,\ldots,l\}}
\blangle \tau_1 \tau_0 \prod_{i\in S} \tau_{2a_i-1}\
\tau_0^{2+2A_S-2l_S}\brangle_0^c
\binom{4+2A-2l}{2+2A_S-2l_S}
\\ \cdot \blangle \tau_2 \tau^2_0 \prod_{i\in T} \tau_{2a_i-1}\
\tau_0^{2+2A_T-2l_T}\brangle_0^c
\end{multline*}
After substituting the closed genus 0 evaluation,
cancelling factors, and simplifying, we find:
\begin{multline}\label{vxxz3}
\frac{4+2A-l}{4} (3+2A-l)! =
\\ \sum_{S\cup T=\{1,\ldots,l\}}
\frac{4+2A-l}{4} \binom{4+2A-2l}{2+2A_S-2l_S} (1+2A_s-l_s)!(1+2A_T-l_T)!\ .
\end{multline}
And Equation \ref{vxxz3} is clearly equivalent to Equation \ref{vxxz2}, as needed.

The proof of the vanishing \eqref{fvvt1} for $r=1$ is complete.
Hence, the open Virasoro constraint $\mathcal{L}_1$
is established in genus 0.

\subsection{Vanishing for \texorpdfstring{$r=2$}{r=2}}
By the definition of $\mathcal{L}_2$, we can write the
vanishing \eqref{fvvt2} explicitly for $G_2$.
Using the Virasoro bracket
$$[\mathcal{L}_{-1}, \mathcal{L}_2]= -3 \mathcal{L}_1 \ $$
and the validity of the constraint $\mathcal{L}_1$ in
genus 0,
we need only check the vanishing of $G_2$ at coefficients
independent of $t_0$.

After unravelling the definition of $\mathcal{L}_2$ (just as we
did for $\mathcal{L}_1$), we must prove the following
identity:
\begin{multline}\label{xz2}
\frac{42+12A-12l}{8} (5+2A-l)! =
\\ \sum_{S\cup T\cup U=\{1,\ldots,l\}}
(7+2A-2l) \binom{6+2A-2l}{2+2A_S-2l_S,
2+2A_T-2l_T,2+2A_U-2L_U}\\ \cdot (1+2A_S-l_S)!(1+2A_T-l_T)!
(1+2A_U-l_U)!
\ .
\end{multline}

By applying the closed TRR twice, we obtain the following
relation among closed descendent invariants:
\begin{multline*}
\blangle \tau_2 \tau_2 \tau_2 \prod_{i=1}^l
\tau_{2a_i-1} \ \tau_0^{6+2A-2l}\brangle_0^c = \\
\sum_{S\cup T\cup U=\{1,\ldots,l\}}
\binom{6+2A-2l}{2+2A_S-2l_S,
2+2A_T-2l_T,2+2A_U-2l_U}
\\ \cdot \blangle \tau_1 \tau_0 \prod_{i\in S} \tau_{2a_i-1}\
\tau_0^{2+2A_S-2l_S}\brangle_0^c
\\ \cdot \blangle \tau_2 \tau^2_0 \prod_{i\in T} \tau_{2a_i-1}\
\tau_0^{2+2A_T-2l_T}\brangle_0^c
\\ \cdot \blangle \tau_1 \tau_0 \prod_{i\in U} \tau_{2a_i-1}\
\tau_0^{2+2A_U-2l_U}\brangle_0^c\ .
\end{multline*}
After substituting the closed genus 0 evaluation,
we find the identity
\begin{multline}\label{xz23}
\frac{6+2A-l}{8} (5+2A-l)! =
\\ \sum_{S\cup T\cup U=\{1,\ldots,l\}}
\frac{6+2A-l}{6} \binom{6+2A-2l}{2+2A_S-2l_S,
2+2A_T-2l_T,2+2A_U-2L_U}\\ \cdot (1+2A_S-l_S)!(1+2A_T-l_T)!
(1+2A_U-l_U)!
\ .
\end{multline}
Identity \ref{xz23} is clearly equivalent to Identity \ref{xz2}.

The proof vanishing \eqref{fvvt2} for $r=2$ is complete.
Hence, the open Virasoro constraint $\mathcal{L}_2$
is established in genus 0, and the proof of Theorem \ref{thm:vir} is
complete. \qed

\section{Proof of Theorem \ref{tt22}} \label{KDVSec}

\subsection{KdV}
Our goal is to prove the open KdV relation in genus 0:
\begin{equation} \label{lala3}
(2n+1)\blangle \blangle \tau_n  \brangle \brangle_0^o
=  \blangle \blangle \tau_{n-1} \tau_0\brangle \brangle^c_0
\blangle \blangle \tau_0  \brangle \brangle^o_0
+ 2    \bblangle \tau_{n-1}  \bbrangle^o_0
\bblangle\sigma\bbrangle^o_0 \
\end{equation}
for $n\geq 1$.
After differentiating both sides by $s$, we obtain
\begin{multline} \label{jqq2}
(2n+1)\blangle \blangle \tau_n \sigma \brangle \brangle_0^o
= \\ \blangle \blangle \tau_{n-1} \tau_0\brangle \brangle^c_0
\blangle \blangle \tau_0 \sigma \brangle \brangle^o_0
+ 2    \bblangle \tau_{n-1} \sigma \bbrangle^o_0
\bblangle\sigma\bbrangle^o_0
+ 2    \bblangle \tau_{n-1}  \bbrangle^o_0
\bblangle\sigma^ 2 \bbrangle^o_0 \
\end{multline}
for $n\geq 1$.
Since all nonvanishing genus 0 open invariants have
at least a single $\sigma$ insertion, equation \eqref{jqq2}
implies the open KdV \eqref{lala3} in genus 0.

Since we already have proven the TRR relation
$$\blangle \blangle \tau_n \sigma \brangle\brangle^o_0  =
\blangle\blangle \tau_{n-1} \tau_0 \brangle\brangle^c_0
\blangle
\blangle \tau_0 \sigma \brangle\brangle^o_0 + \blangle\blangle
\tau_{n-1} \brangle\brangle^o_0
\blangle\blangle\sigma^2 \brangle\brangle^o_0\ ,$$
equation \eqref{jqq2}  follows from the
differential equation
\begin{equation} \label{jjxx12}
2n \blangle \blangle \tau_n \sigma \brangle \brangle_0^o
= \\
 2    \bblangle \tau_{n-1} \sigma \bbrangle^o_0
\bblangle\sigma\bbrangle^o_0
+   \bblangle \tau_{n-1}  \bbrangle^o_0
\bblangle\sigma^ 2 \bbrangle^o_0 \
\end{equation}
for $n\geq 1$.

We observe
equation \eqref{jjxx12}  holds trivially for $n=0$.
The
compatibility of \eqref{jjxx12} with the open string
equation is easily checked.
Hence, to prove equation \eqref{jjxx12}, we need only
consider additional insertions $\tau_{a_i}$ with $a_i\geq 1$.
Using \eqref{ckkq} for such insertions, we reduce
\eqref{jjxx12} to the relation
\begin{multline}\label{ckkqq}
(2n-1)\blangle \tau_n \tau_{a_1} \ldots \tau_{a_l} \sigma^k\brangle_0^o =
\\
2\sum_{S \cup T=\{1,\ldots, l\}}
\blangle \tau_{n-1} \prod_{i \in S} \tau_{a_i} \
\sigma^{k_S}
\brangle^o_0\  \binom{k-1}{k_S-1}
\blangle \prod_{i \in T} \tau_{a_i}\ \sigma^{k-k_S+1}\brangle^o_0
\ .
\end{multline}
The sum is over all disjoint unions $S\cup T$ of the
index set $\{1,\ldots, l\}$.
The number of boundary markings in \eqref{ckkqq},
\begin{eqnarray*}
k&= &2n+ 2A -2l +1\\
k_S &=&2n+2A_S - 2 l_S -1 \ ,
\end{eqnarray*}
is as in \eqref{ckkq}. As before, we use the notation \eqref{zzr}.

\subsection{Binomial identities}
Recall the evaluation of Theorem 3,
\begin{equation} \label{gzzzq}
\blangle \tau_n \tau_{a_1} \ldots \tau_{a_l} \sigma^k\brangle_0^o =
\frac{(2n+ 2A -l)!}{(2n-1)!!\prod_{i=1}^l (2a_i-1)!!}
\end{equation}
in case $n\geq 1$ and $a_i\geq 1$ for all $i$.
After substituting evaluation \eqref{gzzzq}, relation \eqref{ckkqq}
reduces to the following binomial identity (after
cancelling all the equal factors on both sides):
\begin{equation}\label{xxzz}
{2n+2A-l} =2 \sum_{S\cup T=\{1,\ldots,l\}}
\frac{\binom{2n+2A-2l}{2n+2A_S-2l_S-2}}
{\binom{2n+2A-l-1}{2n+2A_S-l_S-2}}\ .
\end{equation}
The sum is over all disjoint unions $S\cup T$ of the
index set $\{1,\ldots, l\}$.

\subsection{Closed TRR}
As before, instead of a combinatorial proof of \eqref{xxzz},
we present a geometric argument using the following closed  genus 0 topological recursion
relation
 in genus 0,
\begin{equation}\label{ggssz}
\blangle \blangle \tau_{2n-2} \tau_0 \tau_2 \brangle \brangle^c_0
= \blangle \blangle \tau_{2n-2} \tau_0^2 \brangle \brangle^c_0
\blangle \blangle \tau_0 \tau_1 \brangle \brangle^c_0\ .
\end{equation}
Expanding  \eqref{ggssz} explicitly, we find
\begin{multline}\label{ggtsss}
\blangle \tau_{2n-2} \tau_0\tau_2 \prod_{i=1}^l \tau_{2a_i-1}
 \cdot \tau_0^{2n+2A-2l} \brangle^c_0
= \\ \sum_{S\cup T=\{1,\ldots,l\}}
\blangle \tau_{2n-2} \tau_0^2  \prod_{i\in S}\tau_{2a_i-1}
\cdot \tau_0^ {2n+2A_S-2l_S-2}
\brangle^c_0  \\
\hspace{75pt} \cdot \binom{2n+2A-2l}{2n+2A_S-2l_S-2}
\\ \cdot \blangle \tau_0\tau_1\prod_{i\in T} \tau_{2a_i-1}
\cdot  \tau_0^ {2A_T-2l_T+2}
\brangle^c_0\ .
\end{multline}
We substitute the closed genus 0 formula
$$\blangle \tau_{b_1} \ldots \tau_{b_m} \brangle_0^c =
\binom{m-3}{b_1,\ldots,b_m}  $$
in \eqref{ggtsss}. After cancelling equal
factors on both sides, we arrive exactly at the
desired binomial identity \eqref{xxzz}.
\qed

\section{Proof of Theorem \ref{ttxx}} \label{EVSec}
\subsection{TRR}
Our goal is to prove the evaluation
\begin{equation} \label{gzzq}
\blangle \tau_{a_1} \ldots \tau_{a_l} \sigma^k\brangle_0^o =
\frac{(\sum_{i=1}^l 2a_i -l +1)!}{\prod_{i=1}^l (2a_i-1)!!}
\end{equation}
in case $a_i\geq 1$ for all $i$.
The  dimension constraint for the bracket  \eqref{gzzq} yields
$$-3+k+2l = \sum_{i=1}^l 2a_i \ .$$
 Hence, $k$ must be odd (and at least 1).
The dilaton equation,
$$\blangle \tau_1 \tau_{a_1} \ldots \tau_{a_l} \sigma^k\brangle_0^o
= (-1+k+l)\ \blangle \tau_{a_1} \ldots \tau_{a_l} \sigma^k\brangle_0^o,$$
is easily seen to be compatible with the evaluation \eqref{gzzq}.

Writing the TRR relation
$$\blangle \blangle \tau_n \sigma \brangle\brangle^o_0  =
\blangle\blangle \tau_{n-1} \tau_0 \brangle\brangle^c_0
\blangle
\blangle \tau_0 \sigma \brangle\brangle^o_0 + \blangle\blangle
\tau_{n-1} \brangle\brangle^o_0
\blangle\blangle\sigma^2 \brangle\brangle^o_0\ $$
of Theorem 4
explicitly, we find
\begin{multline}\label{ckkq}
\blangle \tau_n \tau_{a_1} \ldots \tau_{a_l} \sigma^k\brangle_0^o =
\\
\sum_{S \cup T=\{1,\ldots, l\}}
\blangle \tau_{n-1} \prod_{i \in S} \tau_{a_i}\
\sigma^{k_S}
\brangle^o_0\  \binom{k-1}{k_S}
\blangle \prod_{i \in T} \tau_{a_i} \ \sigma^{k-k_S+1}\brangle^o_0
\ .
\end{multline}
The sum is over all disjoint unions $S\cup T$ of the
index set $\{1,\ldots, l\}$.
The number of boundary markings in \eqref{ckkq},
\begin{eqnarray*}
k&= &2n+ 2\sum_{i=1}^l a_i -2l +1\\
k_S &=&2n+2\sum_{i\in S} a_i - 2 |S| -1 \ ,
\end{eqnarray*}
is set by the dimension constraint.
The condition $a_i\geq 1$ forces the term
$$\blangle\blangle \tau_{n-1} \tau_0 \brangle\brangle^c_0
\blangle
\blangle \tau_0 \sigma \brangle\brangle^o_0$$
of the TRR to vanish.
The right side of \eqref{ckkq} is obtained from the
second term of the TRR.

\subsection{Induction}
We prove the evaluation \eqref{gzzq} by descending
induction on the $a_i$.
The base of the induction is when $a_i=1$ for all $i$.
By the compatibility of the evaluation
\eqref{gzzq} and the dilation equation, the base case is
easily established.

By further use of the compatibility with the dilaton equation,
we need only consider invariants
$$\blangle \tau_n \tau_{a_1} \ldots \tau_{a_l} \sigma^k\brangle_0^o$$
where $n \geq 2$ and $a_i\geq 1$. We will prove the induction step by applying
the TRR relation \eqref{ckkq}.
We observe the right side of \eqref{ckkq} contains
{\em no}
disk invariants with $\tau_0$ insertions.
To complete the induction step, we need only prove
the evaluation \eqref{gzzq} satisfies the TRR relation \eqref{ckkq}.
We are left with a combinatorial formula to verify.

\subsection{Binomial identities}

The combinatorial formula which arises in the induction step
can be written as the following binomial identity (after
cancelling all the equal factors on both sides):
\begin{equation}\label{xxz}
\frac{2n+2A-l}{2n-1} =\sum_{S\cup T=\{1,\ldots,l\}}
\frac{\binom{2n+2A-2l}{2n+2A_S-2l_S-1}}
{\binom{2n+2A-l-1}{2n+2A_S-l_S-2}}\ .
\end{equation}
The sum is over all disjoint unions $S\cup T$ of the
index set $\{1,\ldots, l\}$, and
\begin{equation}\label{zzr}
 A= \sum_{i=1}^l a_i, \ \ A_S= \sum_{i\in S} a_i, \ \
A_T= \sum_{i\in T} a_i, \ \
l_S=|S|, \ \ l_T=|T|\ .
\end{equation}

Instead of a direct combinatorial proof of \eqref{xxz},
we present a geometric argument using the closed topological
recursion relations in genus 0,
\begin{equation}\label{ggss}
\blangle \blangle \tau_{2n-1} \tau_0\tau_1 \brangle \brangle^c_0
= \blangle \blangle \tau_{2n-2} \tau_0 \brangle \brangle^c_0
\blangle \blangle \tau_0^2\tau_1 \brangle \brangle^c_0\ .
\end{equation}
First, we write \eqref{ggss} explicitly in the following specially
chosen case:
\begin{multline}\label{ggtss}
 \blangle \tau_{2n-1} \tau_0\tau_1 \prod_{i=1}^l \tau_{2a_i-1}
 \cdot \tau_0^{2n+2A-2l} \brangle^c_0
= \\ \sum_{S\cup T=\{1,\ldots,l\}}
\blangle \tau_{2n-2} \tau_0  \prod_{i\in S}\tau_{2a_i-1}
\cdot \tau_0^ {2n+2A_S-2l_S-1}
\brangle^c_0  \\
\hspace{75pt} \cdot \binom{2n+2A-2l}{2n+2A_S-2l_S-1}
\\ \cdot \blangle \tau^2_0\tau_1\prod_{i\in T} \tau_{2a_i-1}
\cdot  \tau_0^ {2A_T-2l_T+1}
\brangle^c_0\ .
\end{multline}
Second, we substitute the closed genus 0 formula
$$\blangle \tau_{b_1} \ldots \tau_{b_m} \brangle_0^c =
\binom{m-3}{b_1,\ldots,b_m}  $$
in \eqref{ggtss}. After cancelling equal
factors on both sides, we arrive precisely at the
binomial identity \eqref{xxz}. \qed

\appendix
\section{Multisections and the relative Euler class}\label{app:euler}
We summarise relevant definitions concerning multisections and their zero sets. For the most part, we follow~\cite{CRS}.  As usual, all manifolds may have corners.
\begin{definition}
Let $M$ be a $n-$dimensional manifold. A \emph{weighted branched submanifold} $N$ of dimension $k$ is a function
\[
\mu: M\to\Q\cap\left[0,\infty\right),\qquad \supp(\mu)= N,
\]
which satisfies the following condition. For each $x\in M$ there exists an open neighborhood $U$ of $x,$ a finite collection of $k-$dimensional submanifolds, $N_1,\ldots,N_m,$ of $M$ which are relatively closed in $U$ and positive rational numbers $\mu_1,\ldots,\mu_m,$ such that
\[
\forall y\in U, \qquad \left.\mu\right|_U = \sum_{i=1}^m\mu_i\chi_{N_i}.
\]
Here, $\chi_{N_i}$ is the characteristic function of $N_i.$

We call the submanifolds $N_i$ \emph{branches} of $N$ in $U$ and the numbers $\mu_i$ their \emph{weights}.

A weighted branched submanifold is \emph{compact} if the support of $\mu$ is compact.
\end{definition}
Throughout the article we refer to branched weighted manifolds by their support, $N.$ In this appendix, however, it is more convenient to work with the representing function $\mu,$ and this is indeed what we do. We say that $N$ is represented by $\mu$ and we use both notations for the same notion.
\begin{rmk}\label{rmk:stndrd_mfld}
A usual submanifold $N\hookrightarrow M$ is a special case of a weighted branched submanifold.
Indeed, take
\[
\mu = \chi_N, \qquad m=1,\qquad  N_1 = N,\qquad  \mu_1 = 1.
\]
\end{rmk}

\begin{nn}
For a vector space $V,$ denote by $Gr_k\left(V\right)$ the Grassmannian of $k$-dimensional vector subspaces of $V,$ and by $Gr^+_k\left(V\right)$ the Grassmannian of oriented $k$-dimensional vector subspaces of $V.$ The oriented Grassmannian of zero dimensional subspaces $Gr^+_0$ consists of two points labeled $+$ and $-.$ Given a vector bundle $E\to M,$ we denote the associated (oriented) Grassmannian bundle by
\[
Gr_k^{\left(+\right)}\left(E\right) = \left\{\left(x,W\right)\left|x\in M,W\in Gr_k^{\left(+\right)}\left(E_x\right)\right.\right\}.
\]
\end{nn}
\begin{definition}\label{df:or}
Let $M$ be a manifold of dimension $n$, and $\mu$ a weighted branched submanifold of dimension $k.$
\begin{enumerate}
\item
The \emph{tangent bundle} of $\mu$ is the unique $k$-dimensional weighted branched submanifold $T{\mu}$ of $Gr_{k}\left(TM\right)$, such that
\[
T{\mu}\left(x,W\right) = \sum_{T_x N_i = W}\mu_i,
\]
where $\mu_i,N_i$ are weights and branches at $x$ respectively.
\item
An \emph{orientation} of $\mu$ is a function
\[
\mu^+:Gr^+_k\left(TM\right)\to\Q,
\]
which satisfies the following condition.
For all
\[
\left(x,W\right)\in Gr_k^+\left(TM\right),
\]
there exists an open neighborhood $U$ of $x$ in which there are branches $N_i$ of $\mu$ each with a given orientation, and weights $\mu_i$ of $\mu,$
such that
\[
\mu^+\left(x,W\right) = \sum_{T_x N_i = W}\mu_i - \sum_{T_x N_i = -W}\mu_i.
\]
Here, vector spaces are oriented and $-W$ stands for the vector space $W$ with orientation reversed.
\item
If $\mu$ is compact, oriented, of dimension $0$ and $(\supp \mu) \cap \partial M = \emptyset,$ the \emph{weighted cardinality} of $\mu$ is given by
\[
\# \mu = \sum_{x \in M} \mu^+(x,+).
\]
\end{enumerate}
\end{definition}
The existence of the tangent bundle was established in~\cite{CRS} .
\begin{rmk}
Again, the definitions generalize the standard ones for submanifolds. Indeed, let $\mu$ be as in Remark~\ref{rmk:stndrd_mfld}. We take $T\mu\left(x,W\right)$ to be $1$ if and only if $\mu\left(x\right)=1$ and $W = T_x N.$ Otherwise, it is $0.$ Similarly, if $N$ is oriented, we define
\[
\mu^+(x,W) =
\begin{cases}
1, & \mu(x) = 1 \text{ and } W = T_xN,\\
-1, & \mu(x) = 1 \text{ and } W = - T_xN,\\
0, & \text{otherwise}.
\end{cases}
\]
\end{rmk}
We can now define weighted versions of unions and intersections.
\begin{definition}
Let $\mu,\lambda,$ be two branched weighted submanifolds of $M$ of dimensions $k,l,$ respectively. We say that $\mu$ is \emph{transverse} to $\lambda$ and write $\mu\pitchfork\lambda$ if for all $x\in M,$ $W\in Gr_k\left(TM\right), V\in Gr_l\left(TM\right),$ with
\[
T\mu\left(x,W\right),T\lambda\left(x,V\right)>0,
\] W and V intersect transversally.

If $\mu \pitchfork \lambda$, we define the \emph{intersection} $\mu\cap\lambda$ of $\mu$ and $\lambda$ by
\[
\mu\cap\lambda\left(x\right) = \mu\left(x\right)\lambda\left(x\right).
\]
Given orientations $\mu^+, \lambda^+$ of $\mu, \lambda,$ respectively, and given an orientation on $M,$ we define the \emph{induced orientation} of $\mu \cap \lambda$,
\[
\mu^+\cap\lambda^+: Gr_{k+l-\dim M}^+\to\Q,
\]
by
\[
\mu^+\cap\lambda^+ \left(x,U\right) =\sum_{U = V\cap W}\mu^+\left(W\right)\lambda^+\left(V\right).
\]
Here, we need the orientation on $M$ in order to induce the orientation on $V\cap W.$ If $k+l = \dim M,$ we define the \emph{intersection number}
by
\[
\left(\mu,\mu^+\right)\cdot\left(\lambda,\lambda^+\right)
=
\# (\mu\cap\lambda, \mu^+\cap\lambda^+).
\]

If $k = l,$ we define the \emph{union} of $\mu$ and $\lambda$ by
\[
\mu\cup\lambda(x) = \mu(x) + \lambda(x).
\]

\end{definition}
The transverse intersection of branched weighted manifolds of dimensions $k$ and $l$ has dimension
$k+l-\dim M.$
\begin{rmk}
It is easy to see that both intersection and union are commutative and associative. In addition, we have the distributive property. That is, any three branched weighted submanifolds $\lambda,\mu,\nu,$ satisfy
\[
\left(\mu\cup\lambda\right)\cap\nu = \left(\mu\cap\nu\right)\cup\left(\lambda\cap\nu\right).
\]
\end{rmk}
We now move to multisections and operations between them.
\begin{definition}\label{def:multisection_stuff}
Let $p:E\to M$ be a rank $k$ vector bundle over an $n$-dimensional manifold. A \emph{multisection} $s$ of $E,$
is a weighted branched submanifold
\[
\sigma:E\to\Q\cap\left[0,\infty\right),
\]
of the following special form. For all $x\in M$ there exists a neighborhood $U$, smooth sections $s_1,\ldots,s_m:U\to E$ called \emph{branches}, and rational numbers $\sigma_1\ldots,\sigma_m,$ called \emph{weights}, with sum $1,$ such that
\[
\sigma\left(x,v\right) =\sum_{s_i\left(x\right)=v}\sigma_i , \qquad \forall \left(x,v\right)\in \left.E\right|_U.
\]
That is, the total weight of the fiber is $1.$
We say that $s$ is represented by $\sigma$ and we use both notations for the same notion.

Given a submanifold $N\subseteq M$ and a multisection $s$ of $E\to M,$ we define the \emph{restriction} of $s$ to $N$ by
\[
\left. s\right|_{N } = \left.\sigma\right|_{ p^{-1}\left(N\right)}.
\]

Let $f : M \to N$ be a map of smooth manifolds with corners and let $E \to N$ be a vector bundle. Denote by $\tilde f : f^*E \to E$ the canonical map covering $f.$ Let $\sigma$ be a multisection of $E.$ Then the \emph{pull-back} $f^*\sigma$ is the multisection of $f^*E$ given by
\[
(f^*\sigma)(x,v) = \sigma(\tilde f(x,v)).
\]
A multisection is said to be \emph{transverse} if it and the zero section are transverse as weighted branched manifolds.
\end{definition}

\begin{definition}\label{df:sp}
Let $\chi_0$ denote the indicator function of the zero section. Given a scalar $a$ in the base field and a multisection $\sigma,$ we define the product multisection $a\sigma$ by
\[
(a\sigma)(x,v) =
\begin{cases}
\sigma(x,a^{-1}v), & a \neq 0,\\
\chi_0, & a = 0.
\end{cases}
\]

Given several multisections $\sigma_1,\ldots,\sigma_m,$ we define their \emph{sum}
\[
\sigma = \sigma_1+\ldots +\sigma_m
\]
by
\[
\sigma\left(x,v\right) = \sum_{v_1+\ldots+v_m = v}\prod_{i = 1}^m \sigma_i\left(x,v_i\right).
\]
The sum of multisections is commutative and associative.
\end{definition}

Let $pr: [0,1] \times M \to M$ denote the projection. A \emph{homotopy} between two multisections $\sigma_1,\sigma_2,$ of $E \to M$ is a multisection $\sigma$ of
\[
pr^*E\to M\times\left[0,1\right],
\]
such that
\[
\left.\sigma\right|_{E\times\left\{0\right\}} = \sigma_1,\qquad \left.\sigma\right|_{E\times\left\{1\right\}} = \sigma_2.
\]

We say that a multisection \emph{vanishes} at a point if one of its branches vanishes there.

Given multisections $\sigma_i$ of $E_i\to M$ for $i=1,2,$ we define the multisection $\sigma_1\oplus \sigma_2$ of $E_1 \oplus E_2$ by
\[
\sigma\left(\left(x,v_1\oplus v_2\right)\right) = \sigma_1\left(x,v_1\right)\sigma_2\left(x,v_2\right).
\]

Given a multisection $\sigma$ of $E\to M,$ and a section $t$ of a line bundle $L\to M$, we define the multisection
$\sigma t$ of $E\otimes L,$ by
\[
(\sigma t)\left(x,v\otimes w\right) = \sigma\left(x,v\right)\delta_{t\left(x\right)-w},
\]
where $\delta_{t\left(x\right)-w} = 1$ if $t\left(x\right)=w,$ and otherwise it is $0.$

Let $G$ be a discrete group, and let $E\to M$ be a $G$-equivariant vector bundle. Given a multisection $\sigma$ of $E,$ we define the multisection $g\cdot \sigma$ by
\[
\left(g\cdot\sigma\right)\left(x,v\right) = \sigma\left(g^{-1}\cdot\left(x,v\right)\right).
\]
We say that $\sigma$ is $G$-\emph{equivariant} if
\[
\sigma = g\cdot\sigma, \qquad \forall g\in G.
\]
\begin{definition}\label{df:sym}
In case $G$ is finite we define the $G-$\emph{symmetrization} of $\sigma$ by
\[
\sigma^G\left(x,v\right) = \frac{1}{\left|G\right|}\sum_{g \in G} g\cdot \sigma\left(x,v\right).
\]
The symmetrization is $G$ invariant.
\end{definition}

Given a multisection $s$ of $E\to\partial M,$ an extension of $s$ to all $M$ is a multisection $s'$ whose restriction to $\partial M$ is $s.$ It is shown in Chapter~14 of~\cite{HWZ21} that multisections that satisfy a condition called structurability admit extensions. A multisection is structurable if one can equip it with a collection of auxiliary data called a structure. See Definition 13.3.36 of~\cite{HWZ21}. A multisection that arises from an ordinary section admits a natural structure, and the operations on multisections described above act naturally on structures as well. It will be shown in~\cite{HS} that the notion of structure can be modified so that if $\mathscr{S}$ is a structure on the multisection $s,$ there exists an extension $s'$ with structure $\mathscr{S}'$ such that the restriction of $\mathscr{S'}$ to $\partial M$ coincides with $\mathscr{S}.$ If multisections with structures are given on the individual components of $\partial M$ that agree with each other when restricted to the corners, then they piece together to form a multisection with structure on the whole $\partial M.$

\begin{nn}
We denote by $C_m^\infty\left(E\right),$ the space of multisections of $E.$ If a group $G$ acts on $E,$ we use the notation $C_m^\infty\left(E\right)^G$ for the $G-$invariant multisections.
\end{nn}
In case $M$ is oriented of dimension $n$, the image of a section $s$ of a vector bundle $E \to M$ inherits a canonical orientation through the diffeomorphism
\[
s:M\to s\left(M\right).
\]
In a similar manner, every multisection $s\in C_m^\infty\left(E\right),$ carries a natural orientation described as follows.
Assume $s$ is represented by $\sigma,$ take $x\in M, W\in Gr^+_n\left(T_x M\right),$ and let $U,\sigma_i,s_i$ be as in the definition of a multisection. We define
\[
\sigma^+\left(x,W\right) = {\sum}^+\sigma_i-{\sum}^-\sigma_i,
\]
where $\sum^{\pm}$ is taken over indices $i$ such that
\[
W = \pm \left(ds_i\left(T_xU\right)\right).
\]
This definition agrees with the usual orientation for sections.
With these definitions in hand we define the zero set of a multisection as follows.
\begin{definition}
Let $s\in C_m^\infty(E)$ be a transverse multisection. We define its \emph{unoriented zero set} $\tilde{Z}\left(s\right),$ as the intersection of the multisections $s$ and $0$ as branched weighted submanifolds.

In case $M$ and $E\to M$ are oriented we define the zero set $Z\left(s\right),$ to be $\tilde{Z}\left(s\right)$ with the orientation induced from the canonical orientations of $s$ and $0.$
\end{definition}
\begin{rmk}
Let $E \to M$ be a vector bundle with $\rk E = \dim M$ and let $s \in C^\infty_m(E)$ be transverse. Suppose that at a point $x$ several branches $s_{i_j}$ vanish. Then the weight of $x$ in the zero set of $s$ is the signed sum of $\sigma_{i_j}.$ The sign is the sign of the intersection of $s_{i_j}$ and the zero section at $x.$
\end{rmk}

We will use the following theorem. In \cite{CRS}, a proof of this theorem is given in the case that $M$ has no boundary. The proof for a manifold with corners is similar and will be omitted.
\begin{thm}
Let $E\to M$ be a rank $n$ bundle over a manifold of dimension $n.$ Let $s\in C_m^\infty\left(E|_{\partial M}\right)$ vanish nowhere and let $\tilde s \in C_m^\infty(E)$ be a transverse extension. Then $\# Z(\tilde s)$ depends only on $s$ and not on the choice of $\tilde s.$
\end{thm}
In other words, the homology class $[Z(\tilde s)] \in H_0(M)$ depends only on $E$ and $s.$ It is Poincar\'e dual to a relative cohomology class in $H^n(M,\partial M),$ which we call the \emph{relative Euler class of $E$} with respect to $s$.

\bibliographystyle{../../amsabbrvcnobysame}
\bibliography{../../bibli}

\providecommand{\bysame}{\leavevmode\hbox to3em{\hrulefill}\thinspace}
\providecommand{\MR}{\relax\ifhmode\unskip\space\fi MR }
\providecommand{\MRhref}[2]{%
  \href{http://www.ams.org/mathscinet-getitem?mr=#1}{#2}
}
\providecommand{\href}[2]{#2}
\begin{thebibliography}{10}

\bibitem{AhS60}
L.~V. Ahlfors and L.~Sario, \emph{Riemann surfaces}, Princeton Mathematical
  Series, No. 26, Princeton University Press, Princeton, N.J., 1960.

\bibitem{ABT17}
A.~Alexandrov, A.~Buryak, and R.~J. Tessler, \emph{Refined open intersection
  numbers and the {K}ontsevich-{P}enner matrix model}, J. High Energy Phys.
  (2017), no.~3, 123, front matter+40, \href
  {http://dx.doi.org/10.1007/JHEP03(2017)123}
  {\path{doi:10.1007/JHEP03(2017)123}}.

\bibitem{BaB19}
A.~Basalaev and A.~Buryak, \emph{Open {WDVV} equations and {V}irasoro
  constraints}, Arnold Math. J. \textbf{5} (2019), no.~2-3, 145--186, \href
  {http://dx.doi.org/10.1007/s40598-019-00115-w}
  {\path{doi:10.1007/s40598-019-00115-w}}.

\bibitem{Bur}
A.~Buryak, \emph{Equivalence of the open {K}d{V} and the open {V}irasoro
  equations for the moduli space of {R}iemann surfaces with boundary}, Letters
  in Mathematical Physics (2015), 1--22, \href
  {http://dx.doi.org/10.1007/s11005-015-0789-3}
  {\path{doi:10.1007/s11005-015-0789-3}}.

\bibitem{BCT20a}
A.~Buryak, E.~Clader, and R.~J. Tessler, \emph{Open r-spin theory {I}:
  Foundations}, \href {http://arxiv.org/abs/2003.01082}
  {\path{arXiv:2003.01082}}.

\bibitem{BCT20b}
A.~Buryak, E.~Clader, and R.~J. Tessler, \emph{Open $r$-spin theory {II}: The
  analogue of {W}itten's conjecture for $r$-spin disks}, \href
  {http://arxiv.org/abs/1809.02536} {\path{arXiv:1809.02536}}.

\bibitem{ABRT}
A.~Buryak and R.~J. Tessler, \emph{Matrix models and a proof of the open analog
  of {W}itten's conjecture}, Comm. Math. Phys. \textbf{353} (2017), no.~3,
  1299--1328, \href {http://dx.doi.org/10.1007/s00220-017-2899-5}
  {\path{doi:10.1007/s00220-017-2899-5}}.

\bibitem{BCTZ20}
A.~Buryak, A.~N. Zernik, R.~Pandharipande, and R.~J. Tessler, \emph{Open
  $\mathbb{CP}^1$ descendent theory {I}: The stationary sector}, \href
  {http://arxiv.org/abs/2003.00550} {\path{arXiv:2003.00550}}.

\bibitem{Che18}
X.~{Chen}, \emph{{Steenrod Pseudocycles, Lifted Cobordisms, and Solomon's
  Relations for Welschinger's Invariants}}, \href
  {http://arxiv.org/abs/1809.08919} {\path{arXiv:1809.08919}}.

\bibitem{ChZ21}
X.~Chen and A.~Zinger, \emph{W{DVV}-type relations for disk {G}romov-{W}itten
  invariants in dimension 6}, Math. Ann. \textbf{379} (2021), no.~3-4,
  1231--1313, \href {http://dx.doi.org/10.1007/s00208-020-02130-1}
  {\path{doi:10.1007/s00208-020-02130-1}}.

\bibitem{Cho}
C.-H. Cho, \emph{Counting real {$J$}-holomorphic discs and spheres in dimension
  four and six}, J. Korean Math. Soc. \textbf{45} (2008), no.~5, 1427--1442,
  \href {http://dx.doi.org/10.4134/JKMS.2008.45.5.1427}
  {\path{doi:10.4134/JKMS.2008.45.5.1427}}.

\bibitem{CRS}
K.~Cieliebak, I.~Mundet~i Riera, and D.~A. Salamon, \emph{Equivariant moduli
  problems, branched manifolds, and the {E}uler class}, Topology \textbf{42}
  (2003), no.~3, 641--700, \href
  {http://dx.doi.org/10.1016/S0040-9383(02)00022-8}
  {\path{doi:10.1016/S0040-9383(02)00022-8}}.

\bibitem{DM}
P.~Deligne and D.~Mumford, \emph{The irreducibility of the space of curves of
  given genus}, Inst. Hautes \'Etudes Sci. Publ. Math. (1969), no.~36, 75--109.

\bibitem{DiW18}
R.~Dijkgraaf and E.~Witten, \emph{Developments in topological gravity},
  Internat. J. Modern Phys. A \textbf{33} (2018), no.~30, 1830029, 63, \href
  {http://dx.doi.org/10.1142/S0217751X18300296}
  {\path{doi:10.1142/S0217751X18300296}}.

\bibitem{EGH00}
Y.~Eliashberg, A.~Givental, and H.~Hofer, \emph{Introduction to symplectic
  field theory}, no. Special Volume, Part II, 2000, GAFA 2000 (Tel Aviv, 1999),
  pp.~560--673, \href {http://dx.doi.org/10.1007/978-3-0346-0425-3\_4}
  {\path{doi:10.1007/978-3-0346-0425-3\_4}}.

\bibitem{Eli07}
Y.~Eliashberg, \emph{Symplectic field theory and its applications},
  International {C}ongress of {M}athematicians. {V}ol. {I}, Eur. Math. Soc.,
  Z\"{u}rich, 2007, pp.~217--246, \href {http://dx.doi.org/10.4171/022-1/10}
  {\path{doi:10.4171/022-1/10}}.

\bibitem{Fab11}
O.~Fabert, \emph{Gravitational descendants in symplectic field theory}, Comm.
  Math. Phys. \textbf{302} (2011), no.~1, 113--159, \href
  {http://dx.doi.org/10.1007/s00220-010-1180-y}
  {\path{doi:10.1007/s00220-010-1180-y}}.

\bibitem{FaR11}
O.~Fabert and P.~Rossi, \emph{String, dilaton, and divisor equation in
  symplectic field theory}, Int. Math. Res. Not. IMRN (2011), no.~19,
  4384--4404, \href {http://dx.doi.org/10.1093/imrn/rnq251}
  {\path{doi:10.1093/imrn/rnq251}}.

\bibitem{FaR13}
O.~Fabert and P.~Rossi, \emph{Topological recursion relations in
  non-equivariant cylindrical contact homology}, J. Symplectic Geom.
  \textbf{11} (2013), no.~3, 405--448.

\bibitem{FO09}
K.~Fukaya, Y.-G. Oh, H.~Ohta, and K.~Ono, \emph{Lagrangian intersection {F}loer
  theory: anomaly and obstruction. {P}arts {I} and {II}}, AMS/IP Studies in
  Advanced Mathematics, vol.~46, American Mathematical Society, Providence, RI,
  2009.

\bibitem{Ge16}
P.~Georgieva, \emph{Open {G}romov-{W}itten disk invariants in the presence of
  an anti-symplectic involution}, Adv. Math. \textbf{301} (2016), 116--160,
  \href {http://dx.doi.org/10.1016/j.aim.2016.06.009}
  {\path{doi:10.1016/j.aim.2016.06.009}}.

\bibitem{GetzlerPan}
E.~Getzler and R.~Pandharipande, \emph{Virasoro constraints and the {C}hern
  classes of the {H}odge bundle}, Nuclear Phys. B \textbf{530} (1998), no.~3,
  701--714, \href {http://dx.doi.org/10.1016/S0550-3213(98)00517-3}
  {\path{doi:10.1016/S0550-3213(98)00517-3}}.

\bibitem{GiS}
P.~Giterman and J.~Solomon, \emph{Descendent open {G}romov-{W}itten
  invariants}, to appear.

\bibitem{HM}
J.~Harris and I.~Morrison, \emph{Moduli of curves}, Graduate Texts in
  Mathematics, vol. 187, Springer-Verlag, New York, 1998.

\bibitem{Hirsch}
M.~W. Hirsch, \emph{Differential topology}, Graduate Texts in Mathematics,
  vol.~33, Springer-Verlag, New York, 1994, Corrected reprint of the 1976
  original.

\bibitem{HS}
H.~Hofer and J.~Solomon, \emph{Remarks on inductive perturbation theory via
  multisections}, in preparation.

\bibitem{HWZ21}
H.~Hofer, K.~Wysocki, and E.~Zehnder, \emph{Polyfold and {F}redholm theory},
  Ergebnisse der Mathematik und ihrer Grenzgebiete. 3. Folge. A Series of
  Modern Surveys in Mathematics [Results in Mathematics and Related Areas. 3rd
  Series. A Series of Modern Surveys in Mathematics], vol.~72, Springer, Cham,
  [2021] \copyright 2021, \href {http://dx.doi.org/10.1007/978-3-030-78007-4}
  {\path{doi:10.1007/978-3-030-78007-4}}.

\bibitem{HS12}
A.~Horev and J.~P. Solomon, \emph{The open {G}romov-{W}itten-{W}elschinger
  theory of blowups of the projective plane}, \href
  {http://arxiv.org/abs/1210.4034} {\path{arXiv:1210.4034}}.

\bibitem{Joyce}
D.~Joyce, \emph{On manifolds with corners}, Advances in geometric analysis,
  Adv. Lect. Math. (ALM), vol.~21, Int. Press, Somerville, MA, 2012,
  pp.~225--258.

\bibitem{Kont}
M.~Kontsevich, \emph{Intersection theory on the moduli space of curves and the
  matrix {A}iry function}, Comm. Math. Phys. \textbf{147} (1992), no.~1, 1--23.

\bibitem{Li03}
C.-C.~M. Liu, \emph{Moduli of {J}-holomorphic curves with {L}agrangian boundary
  conditions and open {G}romov-{W}itten invariants for an {$S^1$}-equivariant
  pair}, \href {http://arxiv.org/abs/math/0210257} {\path{arXiv:math/0210257}}.

\bibitem{Mir}
M.~Mirzakhani, \emph{Weil-{P}etersson volumes and intersection theory on the
  moduli space of curves}, J. Amer. Math. Soc. \textbf{20} (2007), no.~1,
  1--23, \href {http://dx.doi.org/10.1090/S0894-0347-06-00526-1}
  {\path{doi:10.1090/S0894-0347-06-00526-1}}.

\bibitem{Zer17b}
A.~Netser~Zernik, \emph{Equivariant {O}pen {G}romov-{W}itten {T}heory of
  $\mathbb{R}\mathbb{P}^{2m} \hookrightarrow \mathbb{C}\mathbb{P}^{2m}$}, \href
  {http://arxiv.org/abs/1709.09483} {\path{arXiv:1709.09483}}.

\bibitem{Zer17a}
A.~Netser~Zernik, \emph{Fixed-point {L}ocalization for $\mathbb{RP}^{2m}
  \subset \mathbb{CP}^{2m}$}, \href {http://arxiv.org/abs/1703.02950}
  {\path{arXiv:1703.02950}}.

\bibitem{Zer17}
A.~Netser~Zernik, \emph{Moduli of {O}pen {S}table {M}aps to a {H}omogeneous
  {S}pace}, \href {http://arxiv.org/abs/1709.07402} {\path{arXiv:1709.07402}}.

\bibitem{OP}
A.~Okounkov and R.~Pandharipande, \emph{Gromov-{W}itten theory, {H}urwitz
  numbers, and matrix models}, Algebraic geometry---{S}eattle 2005. {P}art 1,
  Proc. Sympos. Pure Math., vol.~80, Amer. Math. Soc., Providence, RI, 2009,
  pp.~325--414.

\bibitem{R}
B.~Riemann, \emph{Theorie der {A}bel'schen functionen}, J. Reine Angew. Math.
  \textbf{54} (1857), 101--155.

\bibitem{Seidel}
P.~Seidel, \emph{Fukaya categories and {P}icard-{L}efschetz theory}, Zurich
  Lectures in Advanced Mathematics, European Mathematical Society (EMS),
  Z\"urich, 2008, \href {http://dx.doi.org/10.4171/063}
  {\path{doi:10.4171/063}}.

\bibitem{So06}
J.~P. Solomon, \emph{{Intersection theory on the moduli space of holomorphic
  curves with {L}agrangian boundary conditions}}, MIT thesis, \href
  {http://arxiv.org/abs/math.SG/0606429} {\path{arXiv:math.SG/0606429}}.

\bibitem{So07}
J.~P. Solomon, \emph{A differential equation for the open {G}romov-{W}itten
  potential}, preprint, 2007.

\bibitem{JSRT}
J.~P. Solomon and R.~J. Tessler, to appear.

\bibitem{ST19}
J.~P. {Solomon} and S.~B. {Tukachinsky}, \emph{{Relative quantum cohomology}},
  \href {http://arxiv.org/abs/1906.04795} {\path{arXiv:1906.04795}}.

\bibitem{ST21}
J.~P. Solomon and S.~B. Tukachinsky, \emph{Point-like bounding chains in open
  {G}romov-{W}itten theory}, Geom. Funct. Anal. \textbf{31} (2021), no.~5,
  1245--1320, \href {http://dx.doi.org/10.1007/s00039-021-00583-3}
  {\path{doi:10.1007/s00039-021-00583-3}}.

\bibitem{RT}
R.~J. Tessler, \emph{The combinatorial formula for open gravitational
  descendents}, \href {http://arxiv.org/abs/1507.04951}
  {\path{arXiv:1507.04951}}.

\bibitem{We05}
J.-Y. Welschinger, \emph{Invariants of real symplectic 4-manifolds and lower
  bounds in real enumerative geometry}, Invent. Math. \textbf{162} (2005),
  no.~1, 195--234, \href {http://dx.doi.org/10.1007/s00222-005-0445-0}
  {\path{doi:10.1007/s00222-005-0445-0}}.

\bibitem{Wi88}
E.~Witten, \emph{Topological sigma models}, Comm. Math. Phys. \textbf{118}
  (1988), no.~3, 411--449.

\bibitem{Witten}
E.~Witten, \emph{Two-dimensional gravity and intersection theory on moduli
  space}, Surveys in differential geometry ({C}ambridge, {MA}, 1990), Lehigh
  Univ., Bethlehem, PA, 1991, pp.~243--310.

\bibitem{Wi91}
E.~Witten, \emph{Algebraic geometry associated with matrix models of
  two-dimensional gravity}, Topological methods in modern mathematics ({S}tony
  {B}rook, {NY}, 1991), Publish or Perish, Houston, TX, 1993, pp.~235--269.

\end{thebibliography}

\end{document}